\documentclass[3p,times]{amsart}

\usepackage{amssymb}

\usepackage{amsthm}

\usepackage{enumerate} 

\usepackage{amsmath}
\usepackage{amssymb}

\newtheorem{theorem}{Theorem}[section] 
\newtheorem{lemma}[theorem]{Lemma}   
\newtheorem{corollary}[theorem]{Corollary}
\newtheorem{proposition}[theorem]{Proposition}
\newtheorem{definition}[theorem]{Definition}

\newtheorem{main-theorem}[theorem]{Theorem}

\newtheorem*{problem*}{Problem}

\theoremstyle{definition}

\newtheorem*{question*}{Question}

\newtheorem{example}[theorem]{Example}

\renewcommand{\mod}{\operatorname{mod}}

\newcommand{\alg}{\operatorname{alg}}
\newcommand{\rad}{\operatorname{rad}}
\newcommand{\soc}{\operatorname{soc}}
\newcommand{\topp}{\operatorname{top}}
\newcommand{\umod}{\operatorname{\underline{mod}}}

\newcommand{\Ext}{\operatorname{Ext}}
\newcommand{\Hom}{\operatorname{Hom}}
\newcommand{\Ker}{\operatorname{Ker}}
\newcommand{\T}{\operatorname{T}}
\newcommand{\Tr}{\operatorname{Tr}}
\newcommand{\GL}{\operatorname{GL}}
\newcommand{\op}{\operatorname{op}}

\newcommand{\intt}{\operatorname{int}}

\renewcommand{\Im}{\operatorname{Im}}

\newcommand{\bA}{\mathbb{A}}
\newcommand{\bD}{\mathbb{D}}
\newcommand{\bK}{\mathbb{K}}
\newcommand{\bN}{\mathbb{N}}
\newcommand{\bP}{\mathbb{P}}

\newcommand{\bR}{\mathbb{R}}
\newcommand{\bS}{\mathbb{S}}
\newcommand{\bT}{\mathbb{T}}
\newcommand{\bZ}{\mathbb{Z}}

\newcommand{\cO}{\mathcal{O}}
\newcommand{\cC}{\mathcal{C}}
\newcommand{\cD}{\mathcal{D}}
\newcommand{\cT}{\mathcal{T}}

\newcommand{\cB}{\mathcal{B}}

\newcommand{\overbar}[1]{\mkern 5mu\overline{\mkern-5mu#1\mkern-5mu}\mkern 5mu}

\usepackage{etex} 

\usepackage{dsfont}         
\usepackage[h]{esvect}

\usepackage[matrix,arrow,curve,frame,arc]{xy}

\usepackage{pgf}
\usepackage{tikz}
\usepackage{xcolor}

\usetikzlibrary{arrows,shapes,automata,backgrounds,petri,calc}

\newcommand{\tikzAngleOfLine}{\tikz@AngleOfLine}
\def\tikz@AngleOfLine(#1)(#2)#3{%
\pgfmathanglebetweenpoints{%
\pgfpointanchor{#1}{center}}{%
\pgfpointanchor{#2}{center}}
\pgfmathsetmacro{#3}{\pgfmathresult}%
}

\begin{document}

\title{Weighted surface algebras}

{\def\thefootnote{}
\footnote{The research was supported by the research grant
DEC-2011/02/A/ST1/00216 of the National Science Center Poland.}
}

\author[K. Edrmann]{Karin Erdmann}
\address[Karin Erdmann]{Mathematical Institute,
   Oxford University,
   ROQ, Oxford OX2 6GG,
   United Kingdom}
\email{erdmann@maths.ox.ac.uk}

\author[A. Skowro\'nski]{Andrzej Skowro\'nski}
\address[Andrzej Skowro\'nski]{Faculty of Mathematics and Computer Science,
   Nicolaus Copernicus University,
   Chopina~12/18,
   87-100 Toru\'n,
   Poland}
\email{skowron@mat.uni.torun.pl}

\begin{abstract}
A finite-dimensional algebra $A$ over an algebraically
closed field $K$ is called periodic if it is periodic under the action
of the syzygy operator in the category of $A$-$A$-bimodules.
The periodic algebras are self-injective and occurred naturally
in the study of tame blocks of group algebras, actions of finite
groups on spheres, hypersurface singularities of finite
Cohen-Macaulay type,
and Jacobian algebras of quivers with potentials.
Recently, the tame periodic algebras of polynomial growth
have been classified and it is natural to attempt to classify
all tame periodic algebras.
We introduce the weighted surface algebras of triangulated surfaces
with arbitrarily oriented triangles and
describe their basic properties.
In particular, we prove that all these algebras,
except the singular tetrahedral algebras,
are symmetric tame periodic algebras of period $4$.
Moreover, we describe the socle deformations of the weighted
surface algebras and prove that all these algebras
are also symmetric tame periodic algebras of period $4$.
The main results of this paper
form an important step towards
a classification of all periodic
symmetric tame algebras of non-polynomial growth,
and lead to a complete description of all algebras
of generalized quaternion type
with $2$-regular Gabriel quivers \cite{ESk5}.

\bigskip

 

\end{abstract}


\maketitle

\begin{center}
\vspace*{-5mm}
\textit{Dedicated to Idun  Reiten on the occasion of her 75th birthday}
\vspace*{5mm}
\end{center}

\section{Introduction and the main results}\label{sec:intro}

Throughout this paper, $K$ will denote a fixed algebraically closed field.
By an algebra we mean an associative finite-dimensional $K$-algebra
with an identity.
For an algebra $A$, we denote by $\mod A$ the category of
finite-dimensional right $A$-modules and by $D$ the standard
duality $\Hom_K(-,K)$ on $\mod A$.
An algebra $A$ is called \emph{self-injective}
if $A_A$ is injective in $\mod A$, or equivalently,
the projective modules in $\mod A$ are injective.
A prominent class of self-injective algebras is formed
by the \emph{symmetric algebras} $A$ for which there exists
an associative, non-degenerate symmetric $K$-bilinear form
$(-,-): A \times A \to K$.
Classical examples of symmetric algebras are provided
by the blocks of group algebras of finite groups and
the Hecke algebras of finite Coxeter groups.
In fact, any algebra $A$ is the quotient algebra
of its trivial extension algebra $\T(A) = A \ltimes D(A)$,
which is a symmetric algebra.
Two self-injective algebras $A$ and $\Lambda$ are said
to be \emph{socle equivalent} if the quotient algebras
$A/\soc(A)$ and $\Lambda/\soc(\Lambda)$ are isomorphic.

From the remarkable Tame and Wild Theorem of Drozd
(see \cite{CB1,Dr})
the class of algebras over $K$ may be divided
into two disjoint classes.
The first class consists of the \emph{tame algebras}
for which the indecomposable modules occur in each dimension $d$
in a finite number of discrete and a finite number of one-parameter
families.
The second class is formed by the \emph{wild algebras}
whose representation theory comprises
the representation theories of all
algebras over $K$.
Accordingly, we may realistically hope to classify
the indecomposable finite-dimensional modules only
for the tame algebras.
Among the tame algebras we may distinguish the
\emph{representation-finite algebras}, having
only finitely many isomorphism classes of indecomposable
modules, for which the representation theory is rather
well understood.
On the other hand, the representation theory of arbitrary
tame algebras is still only emerging.
The most accessible ones
amongst the tame algebras are 
algebras of polynomial growth \cite{Sk0} for which the number
of one-parameter families
of indecomposable modules in each dimension $d$ is bounded by $d^m$,
for some positive integer $m$ (depending only on the algebra).

Let $A$ be an algebra.
Given a module $M$ in  $\mod A$, its \emph{syzygy}
is defined to be the kernel $\Omega_A(M)$ of a minimal
projective cover of $M$ in $\mod A$.
The syzygy operator $\Omega_A$ is a very important tool
to construct modules in $\mod A$ and relate them.
For $A$ self-injective, it induces an equivalence
of the stable module category $\umod A$,
and its inverse is the shift of a triangulated structure
on $\umod A$ \cite{Ha1}.
A module $M$ in $\mod A$ is said to be \emph{periodic}
if $\Omega_A^n(M) \cong M$ for some $n \geq 1$, and if so
the minimal such $n$ is called the \emph{period} of $M$.
The action of $\Omega_A$ on $\mod A$ can effect
the algebra structure of $A$.
For example, if all simple modules in $\mod A$ are periodic,
then $A$ is a self-injective algebra.
Sometimes one can even recover the algebra $A$ and its
module category from the action of $\Omega_A$.
For example, the self-injective Nakayama algebras
are precisely the algebras $A$ for which $\Omega_A^2$ permutes
the isomorphism classes of simple modules in $\mod A$.
An algebra $A$ is defined to be \emph{periodic} if it is periodic
viewed as a module over the enveloping algebra
$A^e = A^{\op} \otimes_K A$, or equivalently,
as an $A$-$A$-bimodule.
It is known that if $A$ is a periodic algebra of period $n$ then
for any indecomposable non-projective module $M$
in $\mod A$  the syzygy $\Omega_A^n(M)$ is isomorphic to $M$.

Finding or possibly classifying periodic algebras
is an important problem. It is very interesting because of
connections
with group theory, topology, singularity theory
and cluster algebras.
Periodicity of an algebra, and its period,
are invariant under derived equivalences
\cite{Ric2} (see also \cite{ESk3}).
Therefore, to study periodic algebras 
we may assume that the algebras are basic
and indecomposable.

Preprojective algebras of Dynkin type are periodic
and their periods divide 6 (see \cite{AR3,ESn}).
They belong to  a larger class of periodic algebras, 
 the deformed preprojective algebras of generalized
Dynkin type (see \cite{BES1,ESk3}).
With the exception of few small cases, all these
algebras are wild (see \cite{ESk2}).
Preprojective algebras
of Dynkin type occur in other contexts, in 
particular they are the stable Auslander algebras
of the categories of maximal Cohen-Macaulay
of the Kleinian $2$-dimensional hypersurface
singularities
(see \cite{AR1,AR2}).
We refer to
\cite{AR3,Bu,Du2}
for periodicity results on the stable Auslander algebras
of arbitrary
hypersurface singularities of finite Cohen-Macaulay type.
It would be interesting to understand
connections
between the stable Auslander algebras
of hypersurface singularities
of finite Cohen-Macaulay type
and the deformed mesh algebras of generalized
Dynkin type introduced in \cite{ESk3}.
For the simple plane curve singularities of Dynkin type
$\mathbb{A}_n$ this was clarified in \cite{BES2,BES3}.

In \cite{Du1} Dugas proved that every representation-finite
self-injective algebra, without simple blocks, is a periodic
algebra, this extended partial results from 
\cite{BBK,EH,EHS,ESk3} to the general case.
We note that, by general theory (see \cite[Section~3]{Sk2}),
a basic, indecomposable, non-simple, symmetric algebra $A$
is representation-finite if and only if $A$ is socle equivalent
to an algebra $\T(B)^G$ of invariants of the trivial extension
algebra $\T(B)$ of a tilted algebra $B$ of Dynkin type
with respect to free action of a finite cyclic group $G$.
Moreover, there are representation-finite indecomposable
symmetric algebras of arbitrary large period (see \cite{BBK}).
Recently, the representation-infinite, indecomposable,
periodic algebras of polynomial growth were classified
by Bia\l kowski, Erdmann and Skowro\'nski in \cite{BES4}
(see also \cite{Sk1,Sk2}).
In particular, it follows from \cite{BES4} (see also
\cite{BHS,BS1,BS2,Sk1} 
and
\cite[Section~5]{Sk2}) that every basic, indecomposable,
representation-infinite periodic tame symmetric algebra
of polynomial growth is socle
equivalent to an algebra $\T(B)^G$
of invariants of the trivial extension algebra $\T(B)$
of a tubular algebra $B$ of tubular type
$(2,2,2,2)$, $(3,3,3)$, $(2,4,4)$, $(2,3,6)$
(introduced by Ringel \cite{R}) with respect to free
action of a finite cyclic group $G$.
Then one knows that there is a common bound of the periods
of all representation-infinite indecomposable symmetric
algebras of polynomial growth (see \cite{BES4}).

It would be interesting to classify all indecomposable
periodic symmetric tame algebras of non-polynomial growth.
We ask whether the following might hold.

\begin{problem*}
Let $A$ be an indecomposable symmetric tame algebra 
of non-polynomial growth for which  all simple modules in $\mod A$ are periodic.
Is it  true that  $A$ is a periodic algebra of period $4$?
\end{problem*}

Motivated by known properties  of blocks with generalized
quaternion defect groups in the group algebras of finite
groups, Erdmann introduced and investigated in
\cite{E1,E2,E3} the
\emph{algebras of quaternion type},
being the indecomposable, representation-infinite tame
symmetric algebras $A$ with non-singular Cartan matrix $C_A$
for which every indecomposable non-projective
module in $\mod A$ is periodic of period dividing $4$.
In particular, Erdmann proved that every algebra $A$
of quaternion type has at most $3$ non-isomorphic
simple modules and its basic algebra is isomorphic
to an algebra belonging to 12 families of symmetric
representation-infinite algebras defined by quivers
and relations.
Subsequently it has been proved in \cite{Ho} (see also
\cite{NeS} for the polynomial growth cases) that all these
algebras are tame, and are in fact periodic of period $4$
(see \cite{BES4,ESk2}). 
In particular it shows that a  finite group $G$
is periodic with respect to the group cohomology
$H^*(G,\mathbb{Z})$ if and only if all blocks
with non-trivial defect groups of its group algebras
$K G$  over an arbitrary algebraically closed
field $K$ are periodic algebras.
By the famous result of Swan \cite{Sw}
 periodic groups can be characterized as the
finite groups acting freely on finite CW-complexes
homotopically equivalent to spheres
(see \cite[Section~4]{ESk3} for more details).
Some of the algebras
of quaternion type occur as  endomorphism
algebras of cluster tilting objects in the stable
categories of maximal Cohen-Macaulay modules over
odd-dimensional isolated hypersurface singularities
(see \cite[Section~7]{BIKR}).

New interesting families of tame symmetric algebras
with all indecomposable non-projective finite-dimensional
modules periodic of period dividing $4$ appeared surprisingly in
the theory of cluster algebras.
In  \cite{DWZ1,DWZ2},
Derksen, Weyman and Zelevinsky introduced quivers
with potentials and the associated Jacobian algebras,
and established links between the theory of cluster
algebras (invented by Fomin and Zelevinsky \cite{FZ})
and the representation theory of algebras.
On the other hand, in the beautiful paper \cite{FoST},
Fomin, Shapiro and Thurston associated to each bordered
surface with marked points a cluster algebra, each of
whose exchange matrices is defined in terms of the signed
adjacencies between the arcs of an ideal triangulation
of the surface, and such that the flips of triangulations
correspond to the mutations of the associated skew-symmetric
matrices (equivalently, mutations of the associated
quivers).
In particular, a wide class of $2$-acyclic quivers
of finite mutation type has been exhibited in \cite{FoST}.
Moreover, Felikson, Shapiro and Tumarkin proved in
\cite{FeST} that there are only $11$ mutation
equivalence classes of $2$-acyclic quivers of finite
mutation type not coming from triangulations of marked
surfaces.
Further, in \cite{LF} Labardini-Fragoso associated
a quiver with potential to any ideal triangulation
of a surface with marked points in such a way that
flips of triangulations correspond to mutations
of the associated quivers with potentials.
Finally, Ladkani proved in \cite{La1} that the
Jacobian algebras associated to ideal triangulations
of surfaces with empty boundary and punctures and
Labardini-Fragoso potentials are finite-dimensional
tame symmetric algebras with singular Cartan matrices.
Moreover, Valdivieso-D\'iaz proved in \cite{VD}
that the stable Auslander-Reiten quivers of these
Jacobian algebras consist of stable tubes of
ranks $1$ and $2$.
In particular, this showed that the list of tame symmetric 
algebras with periodic module categories announced
in Theorem~6.2 of our article \cite{ESk3} (and hence also in
\cite[Theorem~8.7]{Sk2}) 
is not complete.
In fact, this omission was pointed to us first by S.~Ladkani.

The aim of this  paper is to introduce a more general class
of algebras, called weighted surface algebras,
and describe their basic properties.
In this paper, by a surface we mean a connected,
compact, $2$-dimensional real manifold $S$,
orientable or non-orientable,
with or without boundary.
Then $S$ admits a structure of a finite
$2$-dimensional triangular cell complex,
and hence a triangulation.
We say that 
$(S,\vec{T})$ is a directed triangulated surface if $S$ is a surface,
$T$ is a triangulation of $S$ with at least $3$
pairwise different edges,
and $\vec{T}$ is an arbitrary choice
of orientations of the triangles in $T$.
To such $(S,\vec{T})$
we associate 
a triangulation quiver $(Q(S,\vec{T}),f)$,
where $Q(S,\vec{T})$ is a $2$-regular quiver, that is
every vertex is a source and target of exactly
two arrows. The 
vertices of this quiver are the edges of $T$,
and $f$ is a permutation of the arrows in $Q(S,\vec{T})$
reflecting the orientation $\vec{T}$ of triangles in $T$.
Since $Q(S,\vec{T})$ is 2-regular there is a
 second permutation, denoted by $g$, of the arrows of $Q(S,\vec{T})$.
If  $\cO(g)$ is the set of $g$-orbits of arrows in $Q(S,\vec{T})$, we 
will define 
two functions
$m_{\bullet} : \cO(g) \to \bN^*$
and
$c_{\bullet} : \cO(g) \to K^*$,
called weight and parameter functions.
Then the weighted surface algebra
$\Lambda(S,\vec{T},m_{\bullet},c_{\bullet})$
will be defined as  a quotient algebra
$K Q(S,\vec{T}) / I(S,\vec{T},m_{\bullet},c_{\bullet})$
of the path algebra $K Q(S,\vec{T})$ of $Q(S,\vec{T})$
over $K$  by an admissible ideal
$I(S,\vec{T},m_{\bullet},c_{\bullet})$
of $K Q(S,\vec{T})$.
Certain algebras of this form which are defined  via
 the tetrahedral triangulation of the sphere,
play a special role, we call these  tetrahedral algebras.

The following two theorems describe  basic properties
of the weighted surface algebras.

\begin{main-theorem}
\label{th:main1}
Let $\Lambda = \Lambda(S,\vec{T},m_{\bullet},c_{\bullet})$
be a weighted surface algebra over an algebraically
closed field $K$.
Then the following statements hold:
\begin{enumerate}[(i)]
 \item
  $\Lambda$ is a representation-infinite tame symmetric algebra.
 \item
  $\Lambda$ is of polynomial growth if and only if
  $\Lambda$ is a non-singular tetrahedral algebra.
\end{enumerate}
\end{main-theorem}

\begin{main-theorem}
\label{th:main2}
Let $\Lambda = \Lambda(S,\vec{T},m_{\bullet},c_{\bullet})$
be a weighted surface algebra over an algebraically
closed field $K$.
Then the following statements are equivalent:
\begin{enumerate}[(i)]
 \item
  All simple modules in $\mod \Lambda$ are periodic of period $4$.
 \item
  $\Lambda$ is a periodic algebra of period $4$.
 \item
  $\Lambda$ is not a singular tetrahedral algebra.
\end{enumerate}
\end{main-theorem}

We would like to mention that 
the periodicity of module categories $\mod \Lambda$
of weighted surface (triangulation) algebras different 
from a singular tetrahedral algebra, was established
independently using cluster theory methods
(see \cite{L4,VD}).
But we want stress that the periodicity of algebras
established in the above theorem is a much
stronger property that the periodicity of its
module category.
For example, the periodicity of an algebra
implies the periodicity of its Hochschild cohomology.
Moreover, we provide a self-contained proof of the
above theorem, presenting explicit constructions
of periodic bimodule resolutions of the considered
algebras (see Section~\ref{sec:periodicity}).

We obtain the following direct consequence of the above theorems
and the main result of \cite{E0} (see also Theorem~\ref{th:2.5}).

\begin{corollary}
Let $\Lambda = \Lambda(S,\vec{T},m_{\bullet},c_{\bullet})$
be a weighted surface algebra over an algebraically
closed field $K$, with the Grothendieck group $K_0(\Lambda)$
of rank at least $4$.
Then the Cartan matrix $C_{\Lambda}$ of $\Lambda$ is singular.
\end{corollary}

Let $(S,\vec{T})$ be a directed triangulated surface, and
$m_{\bullet}$, $c_{\bullet}$ weight and parameter functions of
$(Q(S,\vec{T}),f)$.
Assume that the boundary $\partial S$ of $S$ is not empty.
Then we may consider
a border function $b_{\bullet} : \partial(Q(S,\vec{T}),f) \to K$
on the set $\partial(Q(S,\vec{T}),f)$
of vertices of $Q(S,\vec{T})$ corresponding
to the boundary edges of the triangulation $T$ of $S$,
and the associated socle deformed weighted surface algebra
$
  \Lambda(S,\vec{T},m_{\bullet},c_{\bullet},b_{\bullet})
   = K Q(S,\vec{T}) / I (S,\vec{T},m_{\bullet},c_{\bullet},b_{\bullet}),
$
where $I (S,\vec{T},m_{\bullet},c_{\bullet},b_{\bullet})$
is an admissible ideal of $K Q(S,\vec{T})$
such that
$\Lambda(S,\vec{T},m_{\bullet},c_{\bullet},b_{\bullet})$
is socle equivalent to
$\Lambda(S,\vec{T},m_{\bullet},c_{\bullet})$.

The following theorem is the third main result of the paper.

\begin{main-theorem}
\label{th:main4}
Let $A$ be a basic, indecomposable, symmetric algebra
over an algebraically closed field $K$.
Assume that $A$ is socle equivalent but not isomorphic to
a weighted surface algebra
$\Lambda(S,\vec{T},m_{\bullet},c_{\bullet})$.
Then the following statements hold:
\begin{enumerate}[(i)]
 \item
  The surface $S$ has non-empty boundary.
 \item
  $K$ is of characteristic $2$.
 \item
  $A$ is isomorphic to a socle deformed weighted surface algebra
  $\Lambda(S,\vec{T},m_{\bullet},c_{\bullet},b_{\bullet})$.
 \item
  The Cartan matrix $C_A$ of $A$ is singular.
 \item
  $A$ is a tame algebra of non-polynomial growth.
 \item
  $A$ is a periodic algebra of period $4$.
\end{enumerate}
\end{main-theorem}

In Section~\ref{sec:socldeform} we will provide
explicit constructions of periodic bimodule resolutions
of the socle deformed weighted surface algebras.

The above theorems
are the key new results towards classifications
of distinguished classes of tame symmetric algebras.
As the continuation \cite{ESk5} of this paper
we classify basic, indecomposable, representation-infinite,
tame symmetric algebras $A$ with
$2$-regular Gabriel quiver
having at least $3$ vertices and where all simple modules are
periodic of period $4$ (called algebras of generalized quaternion type). 
These are the algebras  socle equivalent
to the weighted surface algebras
$\Lambda(S,\vec{T},m_{\bullet},c_{\bullet})$,
different from the singular tetrahedral algebra,
and the higher tetrahedral algebras investigated in \cite{ESk-HTA}.

Further, the orbit closures of the weighted surface algebras
(and their socle deformations)
in the affine varieties of associative $K$-algebra structures
contain new wide classes of tame symmetric algebras related
to algebras  of dihedral and semidihedral types,
which occurred in the study of blocks of group algebras
with dihedral and semidihedral defect groups.
We refer to \cite{ESk6,ES7} for a classification of algebras of
generalized dihedral type and a characterization of Brauer
graph algebras, using biserial weighted surface algebras.

This paper is organized as follows.
Section~\ref{sec:pre} contains  some known preliminary results
on algebras and modules.
In Section~\ref{sec:bimodule} we describe our general approach and
results for constructing a minimal projective bimodule
resolution of an algebra with periodic simple modules.
Section~\ref{sec:triangulation} introduces triangulation quivers
and shows that they arise naturally from orientations of triangles
of triangulated surfaces.
In Section~\ref{sec:weightsurfalg} we define weighted surface algebras
of directed triangulated surfaces and prove that they are
tame symmetric algebras.
Section~\ref{sec:tetrahedral} is devoted to distinguished properties
of a family of algebras given by the tetrahedral triangulation
of the sphere.
In Section~\ref{sec:periodicity} we discuss the periodicity
of arbitrary weighted surface algebras.
Section~\ref{sec:socldeform} deals with socle deformations
of weighted surface algebras
of directed triangulated surfaces with boundary and
their properties.
In Section~\ref{sec:deformperidic} we prove that
all these algebras are periodic algebras of period $4$.
In Section~\ref{sec:reptype} we
discuss the representation type of the weighted surface algebras
and their socle deformations.

For general background on the relevant representation theory
we refer to the books
\cite{ASS,E3,SS,SY}.

\section{Preliminary results}\label{sec:pre}

A \emph{quiver} is a quadruple $Q = (Q_0, Q_1, s, t)$
consisting of a finite set $Q_0$ of vertices,
a finite set $Q_1$ of arrows,
and two maps $s,t : Q_1 \to Q_0$ which associate
to each arrow $\alpha \in Q_1$ its source $s(\alpha) \in Q_0$
and  its target $t(\alpha) \in Q_0$.
We denote by $K Q$ the path algebra of $Q$ over $K$
whose underlying $K$-vector space has as its basis
the set of all paths in $Q$ of length $\geq 0$, and
by $R_Q$ the arrow ideal of $K Q$ generated by all paths $Q$
of length $\geq 1$.
An ideal $I$ in $K Q$ is said to be admissible
if there exists $m \geq 2$ such that
$R_Q^m \subseteq I \subseteq R_Q^2$.
If $I$ is an admissible ideal in $K Q$, then
the quotient algebra $K Q/I$ is called
a bound quiver algebra, and is a finite-dimensional
basic $K$-algebra.
Moreover, $K Q/I$ is indecomposable if and only if
$Q$ is connected.
Every basic, indecomposable, finite-dimensional
$K$-algebra $A$ has a bound quiver presentation
$A \cong K Q/I$, where $Q = Q_A$ is the \emph{Gabriel
quiver} of $A$ and $I$ is an admissible ideal in $K Q$.
For a bound quiver algebra $A = KQ/I$, we denote by $e_i$,
$i \in Q_0$, the associated complete set of pairwise
orthogonal primitive idempotents of $A$, and by
$S_i = e_i A/e_i \rad A$ (respectively, $P_i = e_i A$),
$i \in Q_0$, the associated complete family of pairwise
non-isomorphic simple modules (respectively, indecomposable
projective modules) in $\mod A$.

Following \cite{SW}, an algebra $A$ is said to be
\emph{special biserial} if $A$ is isomorphic
to a bound quiver algebra $K Q/I$, where the bound
quiver $(Q,I)$ satisfies the following conditions:
\begin{enumerate}[(a)]
 \item
  each vertex of $Q$ is a source and target of at most two arrows,
 \item
  for any arrow $\alpha$ in $Q$ there are at most
  one arrow $\beta$ and at most one arrow $\gamma$
  with $\alpha \beta \notin I$ and $\gamma \alpha \notin I$.
\end{enumerate}
Moreover, if in addition $I$ is generated by paths of $Q$, then
$A = K Q/I$ is said to be a \emph{string algebra} \cite{BR}.
It was  proved in \cite{PS} that the class
of special biserial algebras coincides with the class
of biserial algebras (indecomposable projective modules
have biserial structure) which admit simply connected
Galois coverings.
Furthermore, by \cite[Theorem~1.4]{WW}
we know that  every special biserial agebra is a quotient algebra
of a symmetric special biserial algebra.
We also mention that, if $A$ is a self-injective
special biserial algebra, then $A/\soc(A)$ is a  string algebra.

The following  has been proved by Wald and Waschb\"usch
in \cite{WW} (see also \cite{BR,DS} for alternative
proofs).

\begin{proposition}
\label{prop:2.1}
Every special biserial algebra is tame.
\end{proposition}

For a positive integer $d$, we denote by $\alg_d(K)$ the affine
variety of associative $K$-algebra structures with identity on
the affine space $K^d$.
Then the general linear group $\GL_d(K)$ acts on $\alg_d(K)$
by transport of the structures, and the $\GL_d(K)$-orbits in
$\alg_d(K)$ correspond to the isomorphism classes of $d$-dimensional
algebras (see \cite{Kr} for details). We identify a $d$-dimensional
algebra $A$ with the point of $\alg_d(K)$ corresponding to it.
For two $d$-dimensional algebras $A$ and $B$, we say that $B$
is a \emph{degeneration} of $A$ ($A$ is a \emph{deformation} of $B$)
if $B$ belongs to the closure of the $\GL_d(K)$-orbit
of $A$ in the Zariski topology of $\alg_d(K)$.

Geiss' Theorem \cite{Ge} shows that if $A$ and $B$ are two
$d$-dimensional algebras, $A$ degenerates to $B$ and $B$ is a tame
algebra, then $A$ is also a tame algebra (see also \cite{CB2}).
We will apply this theorem in the following special situation.

\begin{proposition}
\label{prop:2.2}
Let $d$ be a positive integer, and $A(t)$, $t \in K$,
be an algebraic family in $\alg_d(K)$ such that $A(t) \cong A(1)$
for all $t \in K \setminus \{0\}$.
Then $A(1)$ degenerates to $A(0)$.
In particular, if $A(0)$ is tame, then $A(1)$ is tame.
\end{proposition}

A family of algebras $A(t)$, $t \in K$, in $\alg_d(K)$
is said to be \emph{algebraic} if the induced map
$A(-) : K \to \alg_d(K)$ is a regular map of affine varieties.

An important combinatorial and homological invariant
of the module category $\mod A$ of an algebra $A$
is its Auslander-Reiten quiver $\Gamma_A$.
Recall that $\Gamma_A$ is the translation quiver whose
vertices are the isomorphism classes of indecomposable
modules in $\mod A$, the arrows correspond
to irreducible homomorphisms, and the translation
is the Auslander-Reiten translation $\tau_A = D \Tr$.
For $A$ self-injective, we denote by $\Gamma_A^s$
the stable Auslander-Reiten quiver of $A$, obtained
from $\Gamma_A$ by removing the isomorphism classes
of projective modules and the arrows attached to them.
By a stable tube we mean a translation quiver $\Gamma$
of the form $\mathbb{Z} \mathbb{A}_{\infty}/(\tau^r)$,
for some $r \geq 1$, and we call $r$ the rank of $\Gamma$.
We note that, for a symmetric algebra $A$, we have
$\tau_A = \Omega_A^2$ (see \cite[Corollary~IV.8.6]{SY}).
In particular, we have the following equivalence.

\begin{proposition}
\label{prop:2.3}
Let $A$ be an indecomposable, representation-infinite
symmetric algebra.
The following statements are equivalent:
\begin{enumerate}[(i)]
 \item
  $\Gamma_A^s$ consists of stable tubes.
 \item
  All indecomposable non-projective modules in $\mod A$
  are periodic.
\end{enumerate}
\end{proposition}

Therefore, we conclude that, if $A$ is an indecomposable,
representation-infinite, symmetric, periodic algebra (of period $4$)
then $\Gamma_A^s$ consists of stable tubes (of ranks $1$ and $2$).
We also note that, if $A$ is a representation-infinite special
biserial symmetric algebra, then $\Gamma_A^s$ admits
an acyclic component (see \cite{ESk1}),
and consequently $A$ is not a periodic algebra.

Let $A$ be an algebra over $K$ and $\sigma$ a $K$-algebra
automorphism of $A$. Then for any $A$-$A$-bimodule $M$
we denote by ${}_1M_{\sigma}$ the $A$-$A$-bimodule with
the underlying $K$-vector space $M$ and action defined
as $a m b = a m \sigma(b)$ for all $a, b \in A$ and $m \in M$.

The following has been proved in \cite[Theorem~1.4]{GSS}.

\begin{theorem}
\label{th:2.4}
Let $A$ be an algebra over $K$ and $d$ a positive integer.
Then the following statements are equivalent:
\begin{enumerate}[(i)]
 \item
  $\Omega_A^d(S) \cong S$ in $\mod A$ for every simple
  module $S$ in $\mod A$.
 \item
  $\Omega_{A^e}^d(S) \cong {}_1A_{\sigma}$ in $\mod A^e$ for
  some $K$-algebra automorphism $\sigma$ of $A$ such that
  $\sigma(e) A \cong e A$ for any primitive idempotent
  $e$ of $A$.
\end{enumerate}
Moreover, if $A$ satisfies these conditions, then $A$ is self-injective.
\end{theorem}

The \emph{Cartan matrix} $C_A$ of an algebra $A$ is the matrix
$(\dim_K \Hom_A(P_i,P_j))_{1 \leq i,j \leq n}$
for a complete family $P_1,\dots,P_n$
of a pairwise non-isomorphic indecomposable
projective modules in $\mod A$.
The following main result from \cite{E0}
shows why the original class of algebras of quaternion type is very
restricted compared with the algebras which we will study in this paper.

\begin{theorem}
\label{th:2.5}
Let $A$ be an indecomposable, representation-infinite tame
symmetric algebra
with non-singular Cartan matrix
such that every non-projective
indecomposable module in $\mod A$
is periodic of period dividing $4$.
Then $\mod A$ has at most three
pairwise non-isomorphic simple modules.
\end{theorem}

\section{Bimodule resolutions of self-injective algebras}\label{sec:bimodule}

In this section we describe a general approach
for proving that an algebra $A$ with periodic simple modules
is a periodic algebra.

Let $A = K Q/I$ be a bound quiver algebra,
and $e_i$, $i \in Q$, be the primitive idempotents of $A$
associated to the vertices of $Q$.
Then $e_i \otimes e_j$, $i,j \in Q_0$, form a set
of pairwise orthogonal primitive idempotents of
the enveloping algebra $A^e = A^{\op} \otimes_K A$
whose sum is the identity of $A^e$.
Hence,
$P(i,j) = (e_i \otimes e_j) A^e = A e_i \otimes e_j A$,
for $i,j \in Q_0$, form a complete set of pairwise non-isomorphic
indecomposable projective modules in $\mod A^e$
(see \cite[Proposition~IV.11.3]{SY}).

The following result by Happel \cite[Lemma~1.5]{Ha2} describes
the terms of a minimal projective resolution of $A$ in $\mod A^e$.

\begin{proposition}
\label{prop:3.1}
Let $A = K Q/I$ be a bound quiver algebra.
Then there is in $\mod A^e$ a minimal projective resolution of $A$
of the form
\[
  \cdots \rightarrow
  \bP_n \xrightarrow{d_n}
  \bP_{n-1} \xrightarrow{ }
  \cdots \rightarrow
  \bP_1 \xrightarrow{d_1}
  \bP_0 \xrightarrow{d_0}
  A \rightarrow 0,
\]
where
\[
  \bP_n = \bigoplus_{i,j \in Q_0}
          P(i,j)^{\dim_K \Ext_A^n(S_i,S_j)}
\]
for any $n \in \bN$.
\end{proposition}

The syzygy modules have an important property, a proof for the next
Lemma may be found in \cite[Lemma~IV.11.16]{SY}.

\begin{lemma}
\label{lem:3.2}
Let $A$ be an algebra.
For any positive integer $n$, the module  $\Omega_{A^e}^n(A)$
is  projective as a left $A$-module and also as a right
$A$-module.
\end{lemma}

There is no general recipe for
the differentials $d_n$ in Proposition~\ref{prop:3.1}, except for the first three which we will now describe.
We have
\[
  \bP_0 = \bigoplus_{i \in Q_0} P(i,i)
        = \bigoplus_{i \in Q_0} A e_i \otimes e_i A .
\]
The homomorphism $d_0 : \bP_0 \to A$ in $\mod A^e$ defined by
$d_0 (e_i \otimes e_i) = e_i$ for all $i \in Q_0$
is a minimal projective cover of $A$ in $\mod A^e$.
Recall that, for two vertices $i$ and $j$ in $Q$,
the number of arrows from $i$ to $j$ in $Q$ is equal
to $\dim_K \Ext_A^1(S_i,S_j)$
(see \cite[Lemma~III.2.12]{ASS}).
Hence we have
\[
  \bP_1 = \bigoplus_{\alpha \in Q_1} P\big(s(\alpha),t(\alpha)\big)
        = \bigoplus_{\alpha \in Q_1} A e_{s(\alpha)} \otimes e_{t(\alpha)} A
        .
\]
Then we have the following known fact (see \cite[Lemma~3.3]{BES4}
for a proof).

\begin{lemma}
\label{lem:3.3}
Let $A = K Q/I$ be a bound quiver algebra, and
$d_1 : \bP_1 \to \bP_0$ the homomorphism in $\mod A^e$
defined by
\[
 d_1(e_{s(\alpha)} \otimes e_{t(\alpha)}) =
   \alpha \otimes e_{t(\alpha)} - e_{s(\alpha)} \otimes \alpha
\]
for any arrow $\alpha$ in $Q$.
Then $d_1$ induces a minimal projective cover
$d_1 : \bP_1 \to \Omega_{A^e}^1(A)$ of
$\Omega_{A^e}^1(A) = \Ker d_0$ in $\mod A^e$.
In particular, we have
$\Omega_{A^e}^2(A) \cong \Ker d_1$ in $\mod A^e$.
\end{lemma}

We will denote  the  homomorphism
$d_1 : \bP_1 \to \bP_0$ by $d$.
For the algebras $A$ we will consider, the kernel
$\Omega_{A^e}^2(A)$ of $d$ will be generated,
as an $A$-$A$-bimodule, by some elements of $\bP_1$
associated to a set of relations generating the
admissible ideal $I$.
Recall that a relation in the path algebra $KQ$
is an element of the form
\[
  \mu = \sum_{r=1}^n c_r \mu_r
  ,
\]
where $c_1, \dots, c_r$ are non-zero elements of $K$ and
$\mu_r = \alpha_1^{(r)} \alpha_2^{(r)} \dots \alpha_{m_r}^{(r)}$
are paths in $Q$ of length $m_r \geq 2$, $r \in \{1,\dots,n\}$,
having a common source and a common target.
The admissible ideal $I$ can be generated by a finite set
of relations in $K Q$ (see \cite[Corollary~II.2.9]{ASS}).
In particular, the bound quiver algebra $A = K Q/I$ is given
by the path algebra $K Q$ and a finite number of identities
$\sum_{r=1}^n c_r \mu_r = 0$ given by a finite set of generators of
the ideal $I$.
Consider the $K$-linear homomorphism $\varrho : K Q \to \bP_1$
which assigns to a path $\alpha_1 \alpha_2 \dots \alpha_m$ in $Q$
the element
\[
  \varrho(\alpha_1 \alpha_2 \dots \alpha_m)
   = \sum_{k=1}^m \alpha_1 \alpha_2 \dots \alpha_{k-1}
                  \otimes \alpha_{k+1} \dots \alpha_m
\]
in $\bP_1$, where $\alpha_0 = e_{s(\alpha_1)}$
and $\alpha_{m+1} = e_{t(\alpha_m)}$.
Observe that
$\varrho(\alpha_1 \alpha_2 \dots \alpha_m) \in e_{s(\alpha_1)} \bP_1 e_{t(\alpha_m)}$.
Then, for a relation $\mu = \sum_{r=1}^n c_r \mu_r$
in $K Q$ lying in $I$, we have an element
\[
  \varrho(\mu) = \sum_{r=1}^n c_r \varrho(\mu_r) \in e_i \bP_1 e_j ,
\]
where $i$ is the common source and $j$ is the common
target of the paths $\mu_1,\dots,\mu_r$.
The following lemma shows that relations always
produce elements in the kernel of $d_1$;  the proof 
is straightforward.

\begin{lemma}
\label{lem:3.4}
Let $A = K Q/I$ be a bound quiver algebra and
$d_1 : \bP_1 \to \bP_0$ the homomorphism in $\mod A^e$
defined in Lemma~\ref{lem:3.3}. 
Then for any relation $\mu$ in $K Q$ lying in $I$,
we have $d_1(\varrho(\mu)) = 0$.
\end{lemma}

For an algebra $A = K Q/I$ in our context, we will see that
there exists a family of relations $\mu^{(1)},\dots,\mu^{(q)}$
generating the ideal $I$ such that the associated elements
$\varrho(\mu^{(1)}), \dots, \varrho(\mu^{(q)})$ generate
the $A$-$A$-bimodule $\Omega_{A^e}^2(A) = \Ker d_1$.
In fact, using Lemma~\ref{lem:3.2},
we will be able to show that
\[
  \bP_2 = \bigoplus_{j = 1}^q P\big(s(\mu^{(j)}),t(\mu^{(j)})\big)
        = \bigoplus_{j = 1}^q A e_{s(\mu^{(j)})} \otimes e_{t(\mu^{(j)})} A
        ,
\]
and the homomorphism $d_2 : \bP_2 \to \bP_1$ in $\mod A^e$ such that
\[
  d_2 \big(e_{s(\mu^{(j)})} \otimes e_{t(\mu^{(j)})}\big) = \varrho(\mu^{(j)})
  ,
\]
for $j \in \{1,\dots,q\}$, defines a projective cover
of $\Omega_{A^e}^2(A)$ in $\mod A^e$.
In particular, we have $\Omega_{A^e}^3(A) \cong \Ker d_2$ in $\mod A^e$.
We will denote  this homomorphism $d_2$ by $R$.

For the next map $d_3 : \bP_3 \to \bP_2$, which we will call $S:=d_3$
later,
we do not have a general recipe. To define it, 
we need a set of minimal generators for 
$\Omega_{A^e}^3(A)$, and 
Proposition~\ref{prop:3.1}  tells us where we should look for them.

\section{Triangulation quivers of surfaces}\label{sec:triangulation}

The aim of this section is to introduce triangulation quivers 
of directed triangulated surfaces and present several 
examples illustrating possible shapes of such
quivers.

In this paper, by a \emph{surface}
we mean a connected, compact, $2$-dimensional real
manifold $S$, orientable or non-orientable,
with  or without boundary.
It is well known that every surface $S$ admits
an additional structure of a finite
$2$-dimensional triangular cell complex,
and hence a triangulation by the deep Triangulation Theorem
(see for example \cite[Section~2.3]{Ca}).

For a natural number $n$, we denote by $D^n$ the unit disk
in the $n$-dimensional Euclidean space $\bR^n$, which consists
of all points of distance $\leq 1$ from the origin.
Then the boundary $\partial D^n$ of $D^n$ is the unit sphere
$S^{n-1}$ in $\bR^n$, formed by all points of distance $1$
from the origin.
Further, by an $n$-cell we mean a topological space
homeomorphic to the open disk $\intt D^n = D^n \setminus \partial D^n$.
In particular,
$D^0$ and $e^0$ consist of a single point,
and $S^0 = \partial D^1$ consists of two points.
A finite $m$-dimensional cell complex is a topological space $X = X^m$
constructed by the following procedure (see \cite{H}):
\begin{enumerate}[(1)]
 \item
  Start with a finite discrete set $X^0$, whose points are
  regarded as $0$-cells.
 \item
  Inductively, for $n \in \{1,\dots,m\}$,
  form the $n$-skeleton $X^n$
  from $X^{n-1}$ by attaching a finite number
  of $n$-cells $e_i^n$ via maps
  $\varphi_i^{n} : S^{n-1} \to X^{n-1}$.
  This means that $X^n$ is the quotient
  space of the disjoint union $X^{n-1} \coprod_i D_i^n$
  of $X^{n-1}$ and a finite collection
  of $n$-disks $D_i^n$ under the identification
  $x \sim \varphi_i^n (x)$ for $x \in \partial D_i^n$.
  The cell $e_i^n$ is the homeomorphic image
  of $\intt D_i^n = D_i^n \setminus \partial D_i^n$ under the quotient map.
  Hence, as a set $X^n$ is a disjoint union of $X^{n-1}$
  and all attached $n$-cells $e_i^n$.
\end{enumerate}
For each $n$-cell $e_i^n$ of $X$, the composition
of continuous maps
$D_i^n \hookrightarrow X^{n-1} \coprod_i D_i^n \to X^n \to X$
is denoted by $\phi_i^n$ and called the characteristic map
of $e_i^n$.
We also note that a subset $A \subset X$
is open (or closed) if and only if
$A \cap X^n$ is open (or closed) for any $n \in \{0,\dots,m\}$.

The following consequence of \cite[Proposition~A2]{H}
provides a convenient description of finite $m$-dimensional
cell complexes.

\begin{proposition}
\label{prop:4.1}
Let $m$ be a positive integer and $X$
a Hausdorff space.
Then a finite family of continuous maps
$\varphi_i^n : D_i^n \to X$,
with $n \in \{0,\dots,m\}$ and $D_i^n = D^n$,
is the family of characteristic maps of a finite
$m$-dimensional cell complex structure on $X$
if and only if the following conditions are satisfied:
\begin{enumerate}[(i)]
 \item
  Each $\varphi_i^n$ restricts to a homeomorphism from
  $\intt D_i^n$ into its image, a cell $e_i^n \subset X$,
  and these cells are all disjoint and their union is $X$.
 \item
  For each cell $e_i^n$, $\varphi_i^n(\partial D_i^n)$
  is contained in the union of a finite number of cells
  of smaller dimension than $n$.
\end{enumerate}
\end{proposition}

We refer to \cite[Appendix]{H} for some basic topological
facts about cell complexes.

Let $S$ be a surface.
In this paper, by a \emph{finite $2$-dimensional triangular cell complex structure} on $S$
we mean a finite family of continuous maps
$\varphi_i^n : D_i^n \to S$, with $n \in \{0,1,2\}$
and $D_i^n = D^n$,
satisfying the following conditions:
\begin{enumerate}[(1)]
 \item
  Each $\varphi_i^n$ restricts to a homeomorphism from
  $\intt D_i^n$ to the $n$-cell $e_i^n = \varphi_i^n(\intt D_i^n)$,
  and these cells are disjoint and their union is $S$.
 \item
  For each $2$-cell $e_i^2$, $\varphi_i^2(\partial D_i^2)$
  is contained in the union of $k$ $1$-cells  and $k$ $0$-cells,
  with $k \in \{2,3\}$. 
\end{enumerate}
Then the closures  $\varphi_i^2(D_i^2)$ of all $2$-cells $e_i^2$
are called \emph{triangles} of $S$,
and the closures  $\varphi_i^1(D_i^1)$ of all $1$-cells $e_i^1$
are called \emph{edges} of $S$.
The collection $T$ of all triangles $\varphi_i^2(D_i^2)$
is said to be a \emph{triangulation} of $S$.
We assume that such a triangulation $T$ of $S$
has at least three pairwise different edges,
or equivalently, there are at least three pairwise different
$1$-cells in the considered cell complex structure on $S$.
Then $T$ is a finite collection $T_1,\dots,T_n$ of triangles
of the form
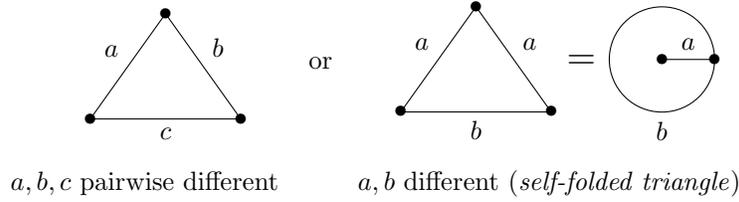
\begin{gather*}
\qquad
\begin{tikzpicture}[auto]
\coordinate (a) at (0,1.4);
\coordinate (b) at (-1,0);
\coordinate (c) at (1,0);
\draw (a) to node {$b$} (c)
(c) to node {$c$} (b);
\draw (b) to node {$a$} (a);
\node (a) at (0,1.4) {$\bullet$};
\node (b) at (-1,0) {$\bullet$};
\node (c) at (1,0) {$\bullet$};
\end{tikzpicture}
\qquad
\raisebox{7ex}{\mbox{or}}
\qquad
\begin{tikzpicture}[auto]
\coordinate (a) at (0,1.4);
\coordinate (b) at (-1,0);
\coordinate (c) at (1,0);
\draw (c) to node {$b$} (b)
(b) to node {$a$} (a);
\draw (a) to node {$a$} (c);
\node (a) at (0,1.4) {$\bullet$};
\node (b) at (-1,0) {$\bullet$};
\node (c) at (1,0) {$\bullet$};
\end{tikzpicture}
%
%
\raisebox{7ex}{\LARGE =}
\ \,
\begin{tikzpicture}[scale=.7,auto]
\coordinate (c) at (0,0);
\coordinate (a) at (1,0);
\coordinate (b) at (0,-1);
\draw (c) to node {$a$} (a);
\draw (b) arc (-90:270:1) node [below] {$b$};
\node (a) at (1,0) {$\bullet$};
\node (c) at (0,0) {$\bullet$};
\end{tikzpicture}
%
\\
\mbox{$a,b,c$ pairwise different}
\qquad
\quad
\mbox{$a,b$ different (\emph{self-folded triangle})}
\end{gather*}
such that every edge of such a triangle in $T$ is either
the edge of exactly two triangles, or is the self-folded
edge, or lies on the boundary.
We note that a given surface $S$ admits many
finite $2$-dimensional cell structures, and hence
triangulations.
We refer to \cite{Ca,KC,Ki} for
general background on surfaces and
constructions of surfaces from plane models.

Let $S$ be a surface and
$T$ a triangulation $S$.
To each triangle $\Delta$ in $T$ we may associate an orientation
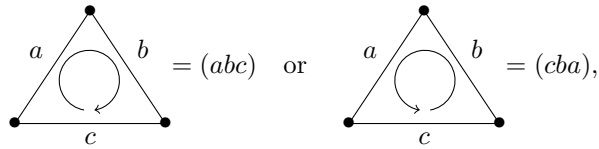
\[
\begin{tikzpicture}[auto]
\coordinate (a) at (0,1.5);
\coordinate (b) at (-1,0);
\coordinate (c) at (1,0);
\coordinate (d) at (-.08,.18);
\draw (a) to node {$b$} (c)
(c) to node {$c$} (b)
(b) to node {$a$} (a);
\draw[->] (d) arc (260:-80:.4);
\node (a) at (0,1.5) {$\bullet$};
\node (b) at (-1,0) {$\bullet$};
\node (c) at (1,0) {$\bullet$};
\end{tikzpicture}
\raisebox{7ex}{\!\!$=(abc)$}
\raisebox{7ex}{\quad or \ \ }
\begin{tikzpicture}[auto]
\coordinate (a) at (0,1.5);
\coordinate (b) at (-1,0);
\coordinate (c) at (1,0);
\coordinate (d) at (.08,.18);
\draw (a) to node {$b$} (c)
(c) to node {$c$} (b)
(b) to node {$a$} (a);
\draw[->] (d) arc (-80:260:.4);
\node (a) at (0,1.5) {$\bullet$};
\node (b) at (-1,0) {$\bullet$};
\node (c) at (1,0) {$\bullet$};
\end{tikzpicture}
\raisebox{7ex}{\!\!$=(cba)$,}
\]
if $\Delta$ has pairwise different edges $a,b,c$, and
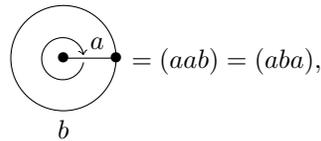
\[
\begin{tikzpicture}[scale=.7,auto]
\coordinate (c) at (0,0);
\coordinate (a) at (1,0);
\coordinate (b) at (0,-1);
\coordinate (d) at (.38,-.08);
\draw (c) to node[above right] {\!\!$a$} (a);
\draw (b) arc (-90:270:1) node [below] {$b$};
\node (a) at (1,0) {$\bullet$};
\node (c) at (0,0) {$\bullet$};
\draw[->] (d) arc (-10:-350:.4);
\end{tikzpicture}
\raisebox{7ex}{$=(aab)=(aba)$,}
\]
if $\Delta$ is self-folded, with the self-folded edge $a$,
and the other edge $b$.
Fix an orientation of each triangle $\Delta$ of $T$,
and denote this choice by $\vec{T}$.
Then
the pair $(S,\vec{T})$ is said to be a
\emph{directed triangulated surface}.
To each directed triangulated surface $(S,\vec{T})$
we associate the quiver $Q(S,\vec{T})$ whose vertices
are the edges of $T$ and the arrows are defined as
follows:
\begin{enumerate}[(1)]
 \item
  for any oriented triangle $\Delta = (a b c)$ in $\vec{T}$
  with pairwise different edges $a,b,c$, we have the cycle
  
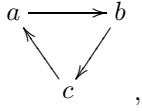
\[
    \xymatrix@C=.8pc@R=1.5pc
      {a \ar[rr] && b \ar[ld] \\ & c \ar[lu]}
    \raisebox{-7ex}{,}
  \]
 \item
  for any self-folded triangle $\Delta = (a a b)$ in $\vec{T}$,
  we have the quiver
  \[
    \xymatrix{ a \ar@(dl,ul)[] \ar@/^1.5ex/[r] & b \ar@/^1.5ex/[l]} ,
  \]
 \item
  for any boundary edge $a$ in ${T}$,
  we have the loop
  \[
    {\xymatrix{ a \ar@(dl,ul)[]}} .
  \]
\end{enumerate}
Then $Q = Q(S,\vec{T})$ is a triangulation quiver in
the following sense (introduced independently by Ladkani
in \cite[Definition~2.4]{La3}
(see also \cite[Definition~3.12]{L4})).

\begin{definition}\label{def:Q,f}
A \emph{triangulation quiver} is a pair $(Q,f)$,
where $Q = (Q_0,Q_1,s,t)$ is a finite connected quiver
and $f : Q_1 \to Q_1$ is a permutation
on the set $Q_1$ of arrows of $Q$ satisfying
the following conditions:
\begin{enumerate}[(a)]
 \item
  every vertex $i \in Q_0$ is the source and target of exactly two
  arrows in $Q_1$,
 \item
  for each arrow $\alpha \in Q_1$, we have $s(f(\alpha)) = t(\alpha)$,
 \item
  $f^3$ is the identity on $Q_1$.
\end{enumerate}
\end{definition}

\bigskip

Let  $Q = Q(S,\vec{T})$ be the quiver associated to the
directed triangulated surface $(S,\vec{T})$.
The permutation $f$ on its set of arrows
is defined as follows:
\begin{enumerate}[(1)]
 \item
  \raisebox{3ex}%
  {
    \xymatrix@C=.8pc@R=1.5pc
   {a \ar[rr]^{\alpha} && b \ar[ld]^{\beta} \\ & c \ar[lu]^{\gamma}}}%
  \quad
  $f(\alpha) = \beta$,
  $f(\beta) = \gamma$,
  $f(\gamma) = \alpha$,

  for an oriented triangle $\Delta = (a b c)$ in $\vec{T}$,
  with pairwise different edges $a,b,c$,
 \item
  \raisebox{0ex}%
  {\xymatrix{ a \ar@(dl,ul)[]^{\alpha} \ar@/^1.5ex/[r]^{\beta} & b \ar@/^1.5ex/[l]^{\gamma}}}
  \quad
  $f(\alpha) = \beta$,
  $f(\beta) = \gamma$,
  $f(\gamma) = \alpha$,

  for a self-folded triangle $\Delta = (a a b)$ in $\vec{T}$, and
 \item
  \raisebox{0ex}%
  {\xymatrix{ a \ar@(dl,ul)[]^{\alpha}}}
  \quad
  $f(\alpha) = \alpha$,

  for a boundary edge $a$ of ${T}$.
\end{enumerate}

We note that for such $(Q,f)$, 
$Q$ is  $2$-regular.
\emph{We will  consider only  triangulation
quivers with at least three vertices.}

We will see below that different directed triangulated
surfaces (even of different genus) may lead to the same
triangulation quiver (see Example~\ref{ex:4.4}).
 We also mention that a similar construction of the 
triangulation quiver associated to an orientable surface
was given in \cite[Proposition~2.5]{La3}
(see also \cite[Definition~4.1]{L4}).

Let $(Q,f)$ be a triangulation quiver.
Then we have the involution $\bar{}: Q_1 \to Q_1$
which assigns to an arrow $\alpha \in Q_1$
the arrow $\bar{\alpha}$ with $s(\alpha) = s(\bar{\alpha})$
and $\alpha \neq \bar{\alpha}$.
With this,  we obtain another permutation $g: Q_1 \to Q_1$
of the set $Q_1$ of arrows of $Q$ such that
$g(\alpha) = \overbar{f(\alpha)}$ for any $\alpha \in Q_1$.
We write $\cO(g)$ for the set of  $g$-orbits
in $Q_1$.

We will present now several examples of triangulation
quivers.
We will denote
by $\bS$ the sphere $S^2$,
by $\bT$ the torus and
by $\bP$ the projective plane.
For two surfaces $X$ and $Y$ we denote by $X \# Y$
the connected sum of $X$ and $Y$.

Recall that $X \# Y$ is the surface constructed by
the following steps:
\begin{enumerate}[(a)]
 \item
  Remove a small open $2$-disk from each of the spaces
  $X$ and $Y$, leaving the boundary $1$-disks
  on each of the surfaces.
 \item
  Glue together the boundary $1$-disks to form
  the connected sum.
\end{enumerate}

In the first three examples we describe all possible
triangulation quivers with exactly three vertices,
and related directed triangulated surfaces.

\begin{example}
\label{ex:4.3}
Let $S = T$ be the triangle
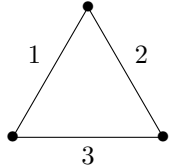
\[
\begin{tikzpicture}[auto]
\coordinate (a) at (0,1.73);
\coordinate (b) at (-1,0);
\coordinate (c) at (1,0);
\draw (a) to node {2} (c)
(c) to node {3} (b);
\draw (b) to node {1} (a);
\node (a) at (0,1.73) {$\bullet$};
\node (b) at (-1,0) {$\bullet$};
\node (c) at (1,0) {$\bullet$};
\end{tikzpicture}
\]
with the three pairwise different edges,
forming the boundary of $S$, and consider the clockwise orientation
$\vec{T}$ of $T$.
Then the triangulation quiver
$Q(S,\vec{T})$ is the quiver
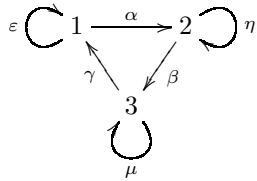
\[
  \xymatrix@C=.8pc@R=1.5pc{
     1 \ar@(dl,ul)[]^{\varepsilon} \ar[rr]^{\alpha} &&
     2 \ar@(ur,dr)[]^{\eta} \ar[ld]^{\beta} \\
     & 3 \ar@(dr,dl)[]^{\mu} \ar[lu]^{\gamma}}
\]
with  $f$-orbits
$(\alpha\ \beta\ \gamma)$,
$(\varepsilon)$,
$(\eta)$,
$(\mu)$.
Observe that we have only one $g$-orbit
$(\alpha\ \eta\ \beta\ \mu\ \gamma\ \varepsilon)$
of arrows in $Q(S,\vec{T})$.
\end{example}

\begin{example}
\label{ex:4.4}
Let $S$  be the sphere $\bS$ with  triangulation $T$
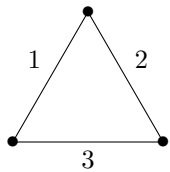
\[
\begin{tikzpicture}[auto]
\coordinate (a) at (0,1.73);
\coordinate (b) at (-1,0);
\coordinate (c) at (1,0);
\draw (a) to node {2} (c)
(c) to node {3} (b);
\draw (b) to node {1} (a);
\node (a) at (0,1.73) {$\bullet$};
\node (b) at (-1,0) {$\bullet$};
\node (c) at (1,0) {$\bullet$};
\end{tikzpicture}
\]
given by two unfolded triangles.
There are two possible orientations
$\vec{T}$ of the triangles of $T$ (up to duality)
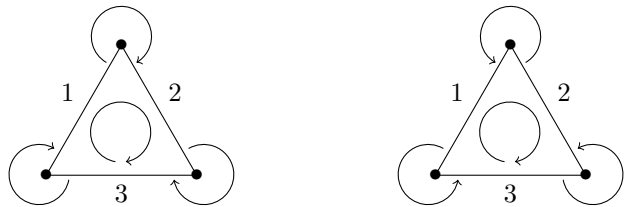
\[
\begin{tikzpicture}[auto]
\coordinate (a) at (0,1.73);
\coordinate (b) at (-1,0);
\coordinate (c) at (1,0);
\coordinate (d) at (-.08,.18);
\coordinate (aa) at (-.2,1.5);
\coordinate (bb) at (-.7,-.06);
\coordinate (cc) at (.7,-.06);
\draw (a) to node {2} (c)
(c) to node {3} (b)
(b) to node {1} (a);
\draw[->] (d) arc (260:-80:.4);
\node (a) at (0,1.73) {$\bullet$};
\node (b) at (-1,0) {$\bullet$};
\node (c) at (1,0) {$\bullet$};
\draw[->] (aa) arc (240:-60:.4);
\draw[->] (bb) arc (350:60:.4);
\draw[<-] (cc) arc (-170:120:.4);
\end{tikzpicture}
\qquad
\qquad
\qquad
\begin{tikzpicture}[auto]
\coordinate (a) at (0,1.73);
\coordinate (b) at (-1,0);
\coordinate (c) at (1,0);
\coordinate (d) at (-.08,.18);
\coordinate (aa) at (-.2,1.5);
\coordinate (bb) at (-.7,-.06);
\coordinate (cc) at (.7,-.06);
\draw (a) to node {2} (c)
(c) to node {3} (b)
(b) to node {1} (a);
\draw[->] (d) arc (260:-80:.4);
\node (a) at (0,1.73) {$\bullet$};
\node (b) at (-1,0) {$\bullet$};
\node (c) at (1,0) {$\bullet$};
\draw[<-] (aa) arc (240:-60:.4);
\draw[<-] (bb) arc (350:60:.4);
\draw[->] (cc) arc (-170:120:.4);
\end{tikzpicture}
\]
of $T$.
The associated triangulation quivers
$Q(S,\vec{T})$ are 
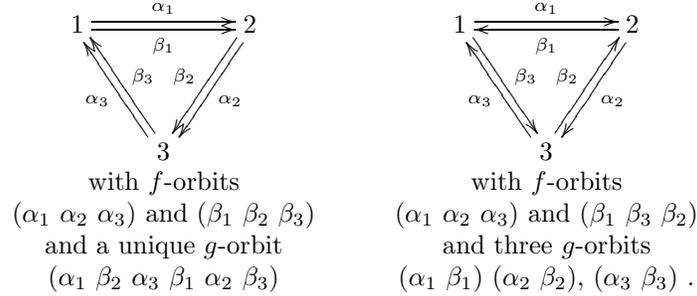
\[
 \begin{tabular}{c@{\qquad\quad}c}
  \xymatrix@R=3.pc@C=1.8pc{
    1
    \ar@<.35ex>[rr]^{\alpha_1}
    \ar@<-.35ex>[rr]_{\beta_1}
    && 2
    \ar@<.35ex>[ld]^{\alpha_2}
    \ar@<-.35ex>[ld]_{\beta_2}
    \\
    & 3
    \ar@<.35ex>[lu]^{\alpha_3}
    \ar@<-.35ex>[lu]_{\beta_3}
  }
 &
  \xymatrix@R=3.pc@C=1.8pc{
    1
    \ar@<.35ex>[rr]^{\alpha_1}
    \ar@<.35ex>[rd]^{\beta_3}
    && 2
    \ar@<.35ex>[ll]^{\beta_1}
    \ar@<.35ex>[ld]^{\alpha_2}
    \\
    & 3
    \ar@<.35ex>[lu]^{\alpha_3}
    \ar@<.35ex>[ru]^{\beta_2}
  }
  \\
  with $f$-orbits & with  $f$-orbits \\
  $(\alpha_1\ \alpha_2\ \alpha_3)$ and $(\beta_1\ \beta_2\ \beta_3)$ &
    $(\alpha_1\ \alpha_2\ \alpha_3)$ and $(\beta_1\ \beta_3\ \beta_2)$ \\
  and a unique $g$-orbit & and three $g$-orbits \\
  $(\alpha_1\ \beta_2\ \alpha_3\ \beta_1\ \alpha_2\ \beta_3)$ &
    $(\alpha_1\ \beta_1)$ $(\alpha_2\ \beta_2)$, $(\alpha_3\ \beta_3)$ .
 \end{tabular}
\]

Consider also the torus $\bT$ with the triangulation $T^*$
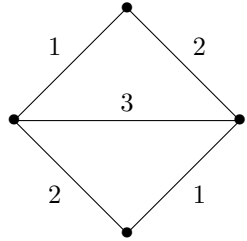
\[
\begin{tikzpicture}[auto]
\coordinate (a) at (0,1.5);
\coordinate (b) at (-1.5,0);
\coordinate (c) at (1.5,0);
\coordinate (d) at (0,-1.5);
\draw (a) to node {2} (c)
(c) to node {1} (d)
(d) to node {2} (b)
(b) to node {1} (a);
\draw (b) to node {3} (c);
\node (a) at (a) {$\bullet$};
\node (b) at (b) {$\bullet$};
\node (c) at (c) {$\bullet$};
\node (d) at (d) {$\bullet$};
\end{tikzpicture}
\]
and the two possible orientations
$\vec{T}^*$ of the triangles of $T^*$ (up to duality)
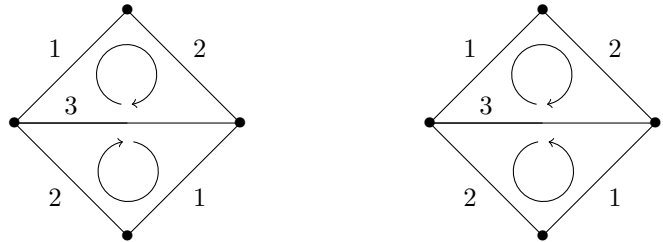
\[
\begin{tikzpicture}[auto]
\coordinate (o) at (0,0);
\coordinate (a) at (0,1.5);
\coordinate (b) at (-1.5,0);
\coordinate (c) at (1.5,0);
\coordinate (d) at (0,-1.5);
\coordinate (e) at (-.08,.25);
\coordinate (f) at (.08,-.25);
\draw (a) to node {2} (c)
(c) to node {1} (d)
(d) to node {2} (b)
(b) to node {1} (a);
\draw (b) to (c);
\draw (b) to node {3} (o);
\draw[->] (e) arc (260:-80:.4);
\draw[->] (f) arc (80:-260:.4);
\node (a) at (a) {$\bullet$};
\node (b) at (b) {$\bullet$};
\node (c) at (c) {$\bullet$};
\node (d) at (d) {$\bullet$};
\end{tikzpicture}
\qquad
\qquad
\qquad
\begin{tikzpicture}[auto]
\coordinate (o) at (0,0);
\coordinate (a) at (0,1.5);
\coordinate (b) at (-1.5,0);
\coordinate (c) at (1.5,0);
\coordinate (d) at (0,-1.5);
\coordinate (e) at (-.08,.25);
\coordinate (f) at (.08,-.25);
\draw (a) to node {2} (c)
(c) to node {1} (d)
(d) to node {2} (b)
(b) to node {1} (a);
\draw (b) to (c);
\draw (b) to node {3} (o);
\draw[->] (e) arc (260:-80:.4);
\draw[<-] (f) arc (80:-260:.4);
\node (a) at (a) {$\bullet$};
\node (b) at (b) {$\bullet$};
\node (c) at (c) {$\bullet$};
\node (d) at (d) {$\bullet$};
\end{tikzpicture}
\]
The associated triangulation quivers
$Q(\bT,\vec{T}^*)$ are
exactly the same as the triangulation quivers
$Q(S,\vec{T})$  above.
\end{example}

\begin{example}
\label{ex:4.5}
Let $S = \bP \# \bP$ be the connected sum
of two copies of the projective plane $\bP$.
Then $S$ admits the triangulation $T$ of the form
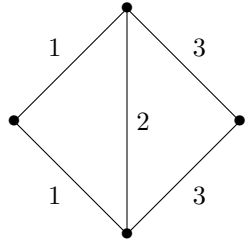
\[
\begin{tikzpicture}[auto]
\coordinate (a) at (0,1.5);
\coordinate (b) at (-1.5,0);
\coordinate (c) at (1.5,0);
\coordinate (d) at (0,-1.5);
\draw (a) to node {3} (c)
(c) to node {3} (d)
(d) to node {1} (b)
(b) to node {1} (a);
\draw (a) to node {2} (d);
\node (a) at (a) {$\bullet$};
\node (b) at (b) {$\bullet$};
\node (c) at (c) {$\bullet$};
\node (d) at (d) {$\bullet$};
\end{tikzpicture}
\]
given by two self-folded triangles sharing
a common edge.
Then we have
a
unique orientation $\vec{T}$
of these two triangles, and the associated
triangulation quiver
$Q(S,\vec{T})$ is of the form
\[
  \xymatrix{
    1
    \ar@(ld,ul)^{\alpha}[]
    \ar@<.5ex>[r]^{\beta}
    & 2
    \ar@<.5ex>[l]^{\gamma}
    \ar@<.5ex>[r]^{\delta}
    & 3
    \ar@<.5ex>[l]^{\sigma}
    \ar@(ru,dr)^{\varrho}[]
  } ,
\]
with the $f$-orbits
$(\alpha\ \beta\ \gamma)$
and
$(\varrho\ \sigma\ \delta)$.
Moreover, $\cO(g)$ consists
of the three $g$-orbits
$(\alpha)$,
$(\varrho)$,
$(\beta\ \delta\ \sigma\ \gamma)$.
We also mention that $\bP \# \bP$
is homeomorphic to the Klein bottle $\bK$
(see \cite[Example~3.8]{Ca}),
and consequently the above triangulation
quiver is also the quiver
$Q(\bK,\vec{T})$,
for the induced directed triangulated structure
on $\bK$.
\end{example}

In the next examples,
the shaded subquivers of a quiver $Q(S,\vec{T})$
define the $f$-orbits of arrows in $Q(S,\vec{T})$.
We note that for a loop, it is always clear from 
the context whether or not it is part of an $f$-orbit
of length $3$.


\begin{example}
Let $S = \bT \# \bP$, and let $T$  be the following triangulation of $S$
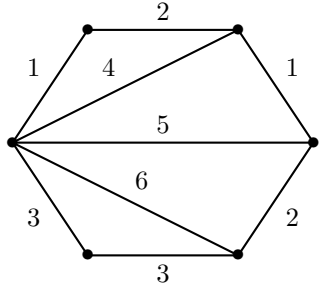
\[
\begin{tikzpicture}
[scale=1,auto]
\node (1) at (-2.0,0) {$\bullet$};
\node (2) at (-1,1.5) {$\bullet$};
\node (3) at (1,1.5) {$\bullet$};
\node (4) at (2.0,0) {$\bullet$};
\node (5) at (1,-1.5) {$\bullet$};
\node (6) at (-1,-1.5) {$\bullet$};
\coordinate (1) at (-2.0,0);
\coordinate (2) at (-1,1.5) ;
\coordinate (3) at (1,1.5) ;
\coordinate (4) at (2.0,0) ;
\coordinate (5) at (1,-1.5) ;
\coordinate (6) at (-1,-1.5) ;
\draw[thick]
(1) edge node {1} (2)
(1) edge node {4} (3)
(1) edge node {5} (4)
(1) edge node {6} (5)
(2) edge node {2} (3)
(3) edge node {1} (4)
(4) edge node {2} (5)
(5) edge node {3} (6)
(6) edge node {3} (1) ;
\end{tikzpicture}
\]
where  the edges $1,2$ correspond to $\bT$, and the edge $3$
corresponds to $\bP$.
Observe that $S$ has empty boundary.
We consider two orientations of the triangles of $T$ and the
associated quivers
\begin{gather*}
\begin{gathered}[b]
\begin{tikzpicture}
[->,scale=1.1,auto]
\coordinate (4) at (0,1.25);
\coordinate (4l) at (-0.15,1.25);
\coordinate (4p) at (0.15,1.25);
\coordinate (5) at (1.5,2);
\coordinate (1) at (0,2.75);
\coordinate (1l) at (-0.15,2.75);
\coordinate (1p) at (0.15,2.75);
\coordinate (2) at (-1.5,2);
\coordinate (6) at (0,0);
\fill[fill=gray!20] (.3,0) arc (-70:-2:2) -- (1.39,1.8) arc (-10:-170:1.4) -- (-1.6,1.8) arc (-178:-110:2) -- cycle;
\fill[fill=gray!20] (4p) -- (5) -- (1p) -- cycle;
\fill[fill=gray!20] (4l) -- (2) -- (1l) -- cycle;
\fill[fill=gray!20] (-.2,-.2) arc (135:225:0.7) -- (.2,-1.2) arc (-45:45:0.7) -- cycle;
\node [circle,minimum size=15](A) at (0,-1.4) { };
\node [circle,minimum size=32](B) at (0,-2.1) {};
\coordinate  (C) at (intersection 2 of A and B);
\coordinate  (D) at (intersection 1 of A and B);
 \tikzAngleOfLine(B)(D){\AngleStart}
 \tikzAngleOfLine(B)(C){\AngleEnd}
\fill[gray!20]%
   let \p1 = ($ (B) - (D) $), \n2 = {veclen(\x1,\y1)}
   in
     (D) arc (\AngleStart-360:\AngleEnd:\n2); 
\node [fill=white,circle,minimum size=1.5](A) at (0,-1.4) { };
\draw[thick,->]%
   let \p1 = ($ (B) - (D) $), \n2 = {veclen(\x1,\y1)}
   in
     (B) ++(60:\n2) node[right]{\footnotesize\ \ \ \raisebox{-7ex}{$\alpha = g \alpha$}}
     (D) arc (\AngleStart-360:\AngleEnd:\n2); 
\node (4) at (0,1.25) {4};
\node (4l) at (-0.15,1.25) {};
\node (4p) at (0.15,1.25) {};
\node (5) at (1.5,2) {5};
\node (1) at (0,2.75) {1};
\node (1l) at (-0.15,2.75) {};
\node (1p) at (0.15,2.75) {};
\node (2) at (-1.5,2) {2};
\node (6) at (0,0) {6};
\node (3) at (0,-1.4) {3};
\draw[thick,->] (.3,0) arc (-70:-36:2) node[right]{\footnotesize \raisebox{0ex}{$g^2 \beta$}} arc (-36:-2:2);
\draw[thick,->] (1.39,1.8) arc (-10:-90:1.4) node[below]{\footnotesize \raisebox{2ex}{$g^6 \beta$}} arc (-90:-170:1.4) ;
\draw[thick,->] (-1.6,1.8) arc (-178:-144:2) node[below]{\footnotesize \raisebox{-3ex}{$g^{10} \beta$}} arc (-144:-110:2);
\draw[thick,->] (1l) to node[above]{\footnotesize$g^9 \beta$} (2);
\draw[thick,->] (2) to node[below]{\footnotesize$g^7 \beta$} (4l);
\draw[thick,->] (4l) to node[left]{\footnotesize$g^4 \beta$\!} (1l);
\draw[thick,->] (1p) to node[above]{\footnotesize$g^5 \beta$} (5);
\draw[thick,->] (5) to node[below]{\footnotesize$g^3 \beta$} (4p);
\draw[thick,->] (4p) to node[right]{\footnotesize\!$g^8 \beta$} (1p);
\draw[thick,->] (-.2,-1.2) arc (225:180:0.7) node[left]{\footnotesize$g \beta$} arc (180:135:0.7);
\draw[thick,->] (.2,-.2) arc (45:0:0.7) node[right]{\footnotesize$\beta=g^{11}\beta$} arc (0:-45:0.7);
\end{tikzpicture}
\\
\mbox{$(S,\vec{T})$ for $\vec{T}$}
\\
\mbox{consisting of oriented triangles}
\\
\mbox{
   (1 2 4), (4 1 5), (5 2 6), (3 3 6)
}
\end{gathered}
\qquad
\quad
\begin{gathered}[b]
\begin{tikzpicture}
[->,scale=1.1,auto]
\coordinate (4) at (0,1.25);
\coordinate (4l) at (-0.15,1.25);
\coordinate (4p) at (0.15,1.25);
\coordinate (5) at (1.5,2);
\coordinate (1) at (0,2.75);
\coordinate (1l) at (-0.15,2.75);
\coordinate (1p) at (0.15,2.75);
\coordinate (2) at (-1.5,2);
\coordinate (6) at (0,0);
\fill[fill=gray!20] (.3,0) arc (-70:-2:2) -- (1.39,1.8) arc (-10:-170:1.4) -- (-1.6,1.8) arc (-178:-110:2) -- cycle;
\fill[fill=gray!20] (4p) -- (5) -- (1p) -- cycle;
\fill[fill=gray!20] (4l) -- (2) -- (1l) -- cycle;
\fill[fill=gray!20] (-.2,-.2) arc (135:225:0.7) -- (.2,-1.2) arc (-45:45:0.7) -- cycle;
\node [circle,minimum size=15](A) at (0,-1.4) { };
\node [circle,minimum size=32](B) at (0,-2.1) {};
\coordinate  (C) at (intersection 2 of A and B);
\coordinate  (D) at (intersection 1 of A and B);
 \tikzAngleOfLine(B)(D){\AngleStart}
 \tikzAngleOfLine(B)(C){\AngleEnd}
\fill[gray!20]%
   let \p1 = ($ (B) - (D) $), \n2 = {veclen(\x1,\y1)}
   in
     (D) arc (\AngleStart-360:\AngleEnd:\n2); 
\node [fill=white,circle,minimum size=1.5](A) at (0,-1.4) { };
\draw[thick,->]%
   let \p1 = ($ (B) - (D) $), \n2 = {veclen(\x1,\y1)}
   in
     (B) ++(60:\n2) node[right]{\footnotesize\ \ \ \raisebox{-7ex}{$\alpha = g \alpha$}}
     (D) arc (\AngleStart-360:\AngleEnd:\n2); 
\node (4) at (0,1.25) {4};
\node (4l) at (-0.15,1.25) {};
\node (4p) at (0.15,1.25) {};
\node (5) at (1.5,2) {5};
\node (1) at (0,2.75) {1};
\node (1l) at (-0.15,2.75) {};
\node (1p) at (0.15,2.75) {};
\node (2) at (-1.5,2) {2};
\node (6) at (0,0) {6};
\node (3) at (0,-1.4) {3};
\draw[thick,->] (.3,0) arc (-70:-36:2) node[below]{\footnotesize\ \raisebox{0ex}{$g^2 \beta$}} arc (-36:-2:2);
\draw[thick,->] (1.39,1.8) arc (-10:-90:1.4) node[below]{\footnotesize \raisebox{2ex}{$\gamma = g^3 \gamma$}} arc (-90:-170:1.4) ;
\draw[thick,->] (-1.6,1.8) arc (-178:-144:2) node[left]{\footnotesize\raisebox{0ex}{$g^{5} \beta$}} arc (-144:-110:2);
\draw[thick,->] (1l) to node[above]{\footnotesize$g^4 \beta$} (2);
\draw[thick,->] (2) to node[below]{\footnotesize$g \gamma$} (4l);
\draw[thick,->] (4l) to node[left]{\footnotesize$g \delta$\!} (1l);
\draw[thick,<-] (1p) to node[above]{\footnotesize$g^3 \beta$} (5);
\draw[thick,<-] (5) to node[below]{\footnotesize$g^2 \gamma$} (4p);
\draw[thick,<-] (4p) to node[right]{\footnotesize\!$\delta = g^2 \delta$} (1p);
\draw[thick,->] (-.2,-1.2) arc (225:180:0.7) node[left]{\footnotesize$g \beta$} arc (180:135:0.7);
\draw[thick,->] (.2,-.2) arc (45:0:0.7) node[right]{\footnotesize$\beta=g^{6}\beta$} arc (0:-45:0.7);
\end{tikzpicture}
\\
\mbox{$(S,\vec{T})$ for $\vec{T}$}
\\
\mbox{consisting of oriented triangles}
\\
\mbox{
  (1 2 4), (1 4 5), (5 2 6), (3 3 6)
}
\end{gathered}
\end{gather*}
Observe that  the first orientation
gives two $g$-orbits of arrows in $Q(S,\vec{T})$
(of lengths $1$ and $11$), while for the second orientation
there are four $g$-orbits
of arrows in $Q(S,\vec{T})$ (of lengths $1, 2, 3, 6$).
\end{example}

\begin{example}
Let $S$ be a once punctured triangle,
and let $T$ be the triangulation of $S$
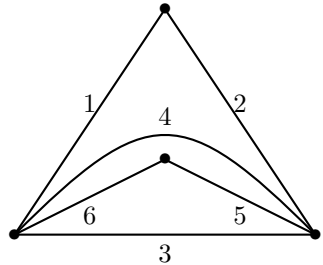
\[
\begin{tikzpicture}
[scale=1]
\node (A) at (-2,0) {$\bullet$};
\node (B) at (2,0) {$\bullet$};
\node (C) at (0,1) {$\bullet$};
\node (D) at (0,3) {$\bullet$};
\coordinate (A) at (-2,0) ;
\coordinate (B) at (2,0) ;
\coordinate (C) at (0,1) ;
\coordinate (D) at (0,3) ;
\draw[thick]
(A) edge node [above] {1} (D)
(D) edge node [above] {2} (B)
(A) edge node [below] {6} (C)
(C) edge node [below] {5} (B)
(A) edge [bend left=45,distance=2.5cm] node [above] {4} (B)
(A) edge node [below] {3} (B) ;
\end{tikzpicture}
\]
such that the edges $1,2,3$ are on the boundary.
We consider two orientations  $\vec{T}$ of the triangles of $T$ and the
associated quivers
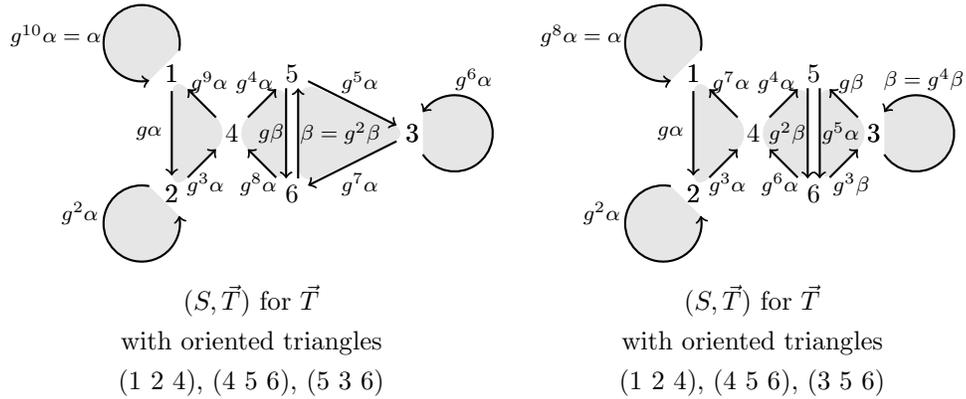
\begin{gather*}
\begin{gathered}[b]
\begin{tikzpicture}
[->,scale=.8]
\node (1) at (-2,1) {1};
\node (2) at (-2,-1) {2};
\node (3) at (2,0) {3};
\node (4) at (-1,0) {4};
\node (5) at (0,1) {5};
\node (6) at (0,-1) {6};
\node [circle,minimum size=.5cm](A) at (-2,1) {1};
\node [circle,minimum size=1.cm](B) at (-2.5,1.5) {};
\coordinate  (C1) at (intersection 2 of A and B);
\coordinate  (D1) at (intersection 1 of A and B);
 \tikzAngleOfLine(B)(D1){\AngleStart}
 \tikzAngleOfLine(B)(C1){\AngleEnd}
\fill[gray!20]%
   let \p1 = ($ (B) - (D1) $), \n2 = {veclen(\x1,\y1)}
   in
     (D1) arc (\AngleStart-360:\AngleEnd:\n2); 
\draw[thick,->]%
   let \p1 = ($ (B) - (D1) $), \n2 = {veclen(\x1,\y1)}
   in
     (B) ++(60:\n2) node[below]{\footnotesize\raisebox{-2.5ex}{$g^{10}\alpha = \alpha\qquad\qquad\qquad\qquad\quad\ $}}
     (D1) arc (\AngleStart-360:\AngleEnd:\n2); 
\node [circle,minimum size=.5cm](A) at (-2,-1) {2};
\node [circle,minimum size=1.cm](B) at (-2.5,-1.5) {};
\coordinate  (C) at (intersection 2 of A and B);
\coordinate  (D) at (intersection 1 of A and B);
 \tikzAngleOfLine(B)(D){\AngleStart}
 \tikzAngleOfLine(B)(C){\AngleEnd}
\fill[gray!20]%
   let \p1 = ($ (B) - (D) $), \n2 = {veclen(\x1,\y1)}
   in
     (D) arc (\AngleStart-360:\AngleEnd:\n2); 
\draw[thick,->]%
   let \p1 = ($ (B) - (D) $), \n2 = {veclen(\x1,\y1)}
   in
     (B) ++(60:\n2) node[left]{\footnotesize\raisebox{-7ex}{$g^2 \alpha\qquad\ $}}
     (D) arc (\AngleStart-360:\AngleEnd:\n2); 
\node [circle,minimum size=.5cm](A) at (2,0) {3};
\node [circle,minimum size=1.cm](B) at (2.7,0) {};
\coordinate  (C) at (intersection 2 of A and B);
\coordinate  (D) at (intersection 1 of A and B);
 \tikzAngleOfLine(B)(D){\AngleStart}
 \tikzAngleOfLine(B)(C){\AngleEnd}
\fill[gray!20]%
   let \p1 = ($ (B) - (D) $), \n2 = {veclen(\x1,\y1)}
   in
     (D) arc (\AngleStart-360:\AngleEnd:\n2); 
\draw[thick,->]%
   let \p1 = ($ (B) - (D) $), \n2 = {veclen(\x1,\y1)}
   in
     (B) ++(60:\n2) node[above]{\footnotesize\raisebox{-7ex}{$g^6 \alpha$}}
     (D) arc (\AngleStart-360:\AngleEnd:\n2); 
\fill[rounded corners=3mm,fill=gray!20] (-2,1) -- (-2,-1) -- (-1,0) -- cycle;
\fill[rounded corners=3mm,fill=gray!20] (-.1,.9) -- (-.1,-.9) -- (-1,0) -- cycle;
\fill[rounded corners=3mm,fill=gray!20] (.1,.9) -- (.1,-.9) -- (2,0) -- cycle;
\draw[thick,->] (-.1,.75) -- node[left]{\footnotesize$g \beta\!\!$} (-.1,-.75);
\draw[thick,->] (.1,-.75) -- node[right]{\footnotesize$\!\!\beta = g^2 \beta$} (.1,.75);
\draw[thick,->]
(1) edge node[left]{\footnotesize$g \alpha$} (2)
(2) edge node[below]{\footnotesize$\ g^3 \alpha$} (4)
(3) edge node[below]{\footnotesize$\ \ g^7 \alpha$} (6)
(4) edge node[above]{\footnotesize$\ \ g^9 \alpha$} (1)
(4) edge node[above]{\footnotesize$g^4 \alpha\ \ $} (5)
(5) edge node[above]{\footnotesize$\ \ g^5 \alpha$} (3)
(6) edge node[below]{\footnotesize$g^8 \alpha\ $} (4);
\end{tikzpicture}
\\
\mbox{$(S,\vec{T})$ for $\vec{T}$}
\\
\mbox{with oriented triangles}
\\
\mbox{
  (1 2 4), (4 5 6), (5 3 6)
}
\end{gathered}
\quad
\begin{gathered}[b]
\begin{tikzpicture}
[->,scale=.8]
\node (1) at (-2,1) {1};
\node (2) at (-2,-1) {2};
\node (3) at (1,0) {3};
\node (4) at (-1,0) {4};
\node (5) at (0,1) {5};
\node (6) at (0,-1) {6};
\node [circle,minimum size=.5cm](A) at (-2,1) {1};
\node [circle,minimum size=1.cm](B) at (-2.5,1.5) {};
\coordinate  (C1) at (intersection 2 of A and B);
\coordinate  (D1) at (intersection 1 of A and B);
 \tikzAngleOfLine(B)(D1){\AngleStart}
 \tikzAngleOfLine(B)(C1){\AngleEnd}
\fill[gray!20]%
   let \p1 = ($ (B) - (D1) $), \n2 = {veclen(\x1,\y1)}
   in
     (D1) arc (\AngleStart-360:\AngleEnd:\n2); 
\draw[thick,->]%
   let \p1 = ($ (B) - (D1) $), \n2 = {veclen(\x1,\y1)}
   in
     (B) ++(60:\n2) node[below]{\footnotesize\raisebox{-2.5ex}{$g^{8}\alpha = \alpha\qquad\qquad\qquad\qquad\quad$}}
     (D1) arc (\AngleStart-360:\AngleEnd:\n2); 
\node [circle,minimum size=.5cm](A) at (-2,-1) {2};
\node [circle,minimum size=1.cm](B) at (-2.5,-1.5) {};
\coordinate  (C) at (intersection 2 of A and B);
\coordinate  (D) at (intersection 1 of A and B);
 \tikzAngleOfLine(B)(D){\AngleStart}
 \tikzAngleOfLine(B)(C){\AngleEnd}
\fill[gray!20]%
   let \p1 = ($ (B) - (D) $), \n2 = {veclen(\x1,\y1)}
   in
     (D) arc (\AngleStart-360:\AngleEnd:\n2); 
\draw[thick,->]%
   let \p1 = ($ (B) - (D) $), \n2 = {veclen(\x1,\y1)}
   in
     (B) ++(60:\n2) node[left]{\footnotesize\raisebox{-7ex}{$g^2 \alpha\qquad\ $}}
     (D) arc (\AngleStart-360:\AngleEnd:\n2); 
\node [circle,minimum size=.5cm](A) at (1,0) {3};
\node [circle,minimum size=1.cm](B) at (1.7,0) {};
\coordinate  (C) at (intersection 2 of A and B);
\coordinate  (D) at (intersection 1 of A and B);
 \tikzAngleOfLine(B)(D){\AngleStart}
 \tikzAngleOfLine(B)(C){\AngleEnd}
\fill[gray!20]%
   let \p1 = ($ (B) - (D) $), \n2 = {veclen(\x1,\y1)}
   in
     (D) arc (\AngleStart-360:\AngleEnd:\n2); 
\draw[thick,->]%
   let \p1 = ($ (B) - (D) $), \n2 = {veclen(\x1,\y1)}
   in
     (B) ++(60:\n2) node[above]{\footnotesize\raisebox{-7ex}{$\!\!\!\!\!\!\beta = g^4 \beta$}}
     (D) arc (\AngleStart-360:\AngleEnd:\n2); 
\fill[rounded corners=3mm,fill=gray!20] (-2,1) -- (-2,-1) -- (-1,0) -- cycle;
\fill[rounded corners=3mm,fill=gray!20] (-.1,.9) -- (-.1,-.9) -- (-1,0) -- cycle;
\fill[rounded corners=3mm,fill=gray!20] (.1,.9) -- (.1,-.9) -- (1,0) -- cycle;
\draw[thick,->] (-.1,.75) -- node[left]{\footnotesize$g^2 \beta\!\!$} (-.1,-.75);
\draw[thick,->] (.1,.75) -- node[right]{\footnotesize$\!\!g^5 \alpha$} (.1,-.75);
\draw[thick,->]
(1) edge node[left]{\footnotesize$g \alpha$} (2)
(2) edge node[below]{\footnotesize$\ g^3 \alpha$} (4)
(6) edge node[below]{\footnotesize$\ \ g^3 \beta$} (3)
(4) edge node[above]{\footnotesize$\ \ g^7 \alpha$} (1)
(4) edge node[above]{\footnotesize$g^4 \alpha\ \ $} (5)
(3) edge node[above]{\footnotesize$\ \ g \beta$} (5)
(6) edge node[below]{\footnotesize$g^6 \alpha\ $} (4);
\end{tikzpicture}
\\
\mbox{$(S,\vec{T})$ for $\vec{T}$}
\\
\mbox{with oriented triangles}
\\
\mbox{
  (1 2 4), (4 5 6), (3 5 6)
}
\end{gathered}
\end{gather*}
In  both orientations $\vec{T}$, there are
 two $g$-orbits of arrows in $Q(S,\vec{T})$
but they have different length.
\end{example}

\begin{example}
Let $S = \bT \# \bT$,
and $T$ be the following triangulation of $S$
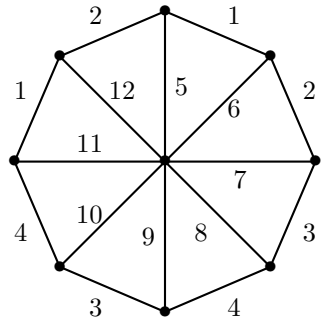
\[
\begin{tikzpicture}
[scale=2,auto]
\node (0) at (0,0) {$\bullet$};
\node (1) at (0,1) {$\bullet$};
\node (2) at (.7,.7) {$\bullet$};
\node (3) at (1,0) {$\bullet$};
\node (4) at (.7,-.7) {$\bullet$};
\node (5) at (0,-1) {$\bullet$};
\node (6) at (-.7,-.7) {$\bullet$};
\node (7) at (-1,0) {$\bullet$};
\node (8) at (-.7,.7) {$\bullet$};
\coordinate (0) at (0,0);
\coordinate (1) at (0,1) ;
\coordinate (2) at (.7,.7) ;
\coordinate (3) at (1,0) ;
\coordinate (4) at (.7,-.7) ;
\coordinate (5) at (0,-1) ;
\coordinate (6) at (-.7,-.7) ;
\coordinate (7) at (-1,0) ;
\coordinate (8) at (-.7,.7) ;
\draw[thick]
(1) edge node {1} (2)
(2) edge node {2} (3)
(3) edge node {3} (4)
(4) edge node {4} (5)
(5) edge node {3} (6)
(6) edge node {4} (7)
(7) edge node {1} (8)
(8) edge node {2} (1)
(1) edge node {5} (0)
(2) edge node [right] {6} (0)
(3) edge node {7} (0)
(4) edge node {8} (0)
(5) edge node {9} (0)
(6) edge node [left] {10} (0)
(7) edge node {11} (0)
(8) edge node {\!\!\!12} (0) ;
\end{tikzpicture}
\]
Consider the
orientation $\vec{T}$ of triangles in $T$
\[
\mbox{
   (1 6 5), (2 7 6), (7 3 8), (8 4 9),
   (9 3 10), (10 4 11), (11 1 12), (2 5 12).
}
\]
Then the quiver $Q(S,\vec{T})$ is of the form
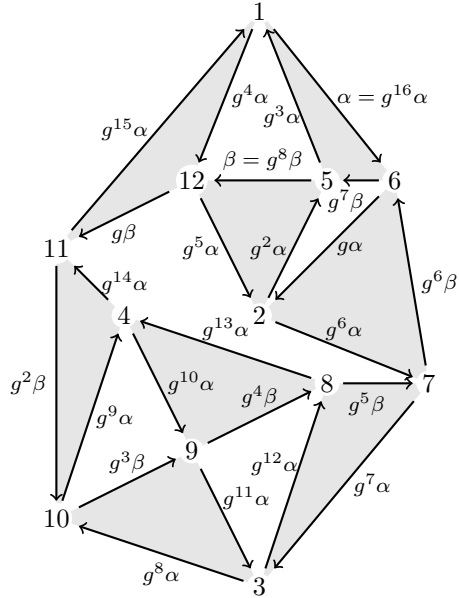
\[
\begin{tikzpicture}
[->,scale=.9]
\coordinate (1) at (0,4.5);
\coordinate (2) at (0,0);
\coordinate (3) at (0,-4);
\coordinate (4) at (-2,0);
\coordinate (5) at (1,2);
\coordinate (6) at (2,2);
\coordinate (7) at (2.5,-1);
\coordinate (8) at (1,-1);
\coordinate (9) at (-1,-2);
\coordinate (10) at (-3,-3);
\coordinate (11) at (-3,1);
\coordinate (12) at (-1,2);
\fill[fill=gray!20] (1) -- (5) -- (6) -- cycle;
\fill[fill=gray!20] (1) -- (11) -- (12) -- cycle;
\fill[fill=gray!20] (2) -- (5) -- (12) -- cycle;
\fill[fill=gray!20] (2) -- (6) -- (7) -- cycle;
\fill[fill=gray!20] (3) -- (7) -- (8) -- cycle;
\fill[fill=gray!20] (3) -- (9) -- (10) -- cycle;
\fill[fill=gray!20] (4) -- (8) -- (9) -- cycle;
\fill[fill=gray!20] (4) -- (10) -- (11) -- cycle;
\node [fill=white,circle,minimum size=1.5] (1) at (0,4.5) { };
\node [fill=white,circle,minimum size=1.5] (2) at (0,0) { };
\node [fill=white,circle,minimum size=1.5] (3) at (0,-4) { };
\node [fill=white,circle,minimum size=1.5] (4) at (-2,0) { };
\node [fill=white,circle,minimum size=1.5] (5) at (1,2) { };
\node [fill=white,circle,minimum size=1.5] (6) at (2,2) { };
\node [fill=white,circle,minimum size=1.5] (7) at (2.5,-1) { };
\node [fill=white,circle,minimum size=1.5] (8) at (1,-1) { };
\node [fill=white,circle,minimum size=1.5] (9) at (-1,-2) { };
\node [fill=white,circle,minimum size=4.5] (10) at (-3,-3) {\ \quad};
\node [fill=white,circle,minimum size=4.5] (11) at (-3,1) {\ \quad};
\node [fill=white,circle,minimum size=4.5] (12) at (-1,2) {\ \quad};
\node (1) at (0,4.5) {1};
\node (2) at (0,0) {2};
\node (3) at (0,-4) {3};
\node (4) at (-2,0) {4};
\node (5) at (1,2) {5};
\node (6) at (2,2) {6};
\node (7) at (2.5,-1) {7};
\node (8) at (1,-1) {8};
\node (9) at (-1,-2) {9};
\node (10) at (-3,-3) {10};
\node (11) at (-3,1) {11};
\node (12) at (-1,2) {12};
\draw[thick,->]
(1) edge node[right]{\footnotesize$\alpha = g^{16} \alpha$} (6)
(1) edge node[right]{\footnotesize$\!g^4 \alpha$} (12)
(2) edge node[left]{\footnotesize$g^{2} \alpha\!$} (5)
(2) edge node[above]{\footnotesize$g^{6} \alpha$} (7)
(3) edge node[above]{\footnotesize$\!\!\!\!\!\!\!\!\!\!g^{12} \alpha$} (8)
(3) edge node[below]{\footnotesize$g^8 \alpha$} (10)
(4) edge node[right]{\footnotesize$g^{10} \alpha$} (9)
(4) edge node[right]{\footnotesize$g^{14} \alpha$} (11)
(5) edge node[below]{\footnotesize$\!\!\!\!\!\!g^3 \alpha$} (1)
(5) edge node[above]{\footnotesize\raisebox{0ex}{$\beta = g^8 \beta$}} (12)
(6) edge node[right]{\footnotesize$g \alpha$} (2)
(6) edge node[below]{\footnotesize$\!\!\!\!\!\!\!\!g^7 \beta$} (5)
(7) edge node[right]{\footnotesize$g^7 \alpha$} (3)
(7) edge node[right]{\footnotesize$g^6 \beta$} (6)
(8) edge node[above]{\footnotesize$g^{13} \alpha$} (4)
(8) edge node[below]{\footnotesize$\!\!\!\!\!\!g^5 \beta$} (7)
(9) edge node[above]{\footnotesize$g^{11} \alpha\!\!\!\!\!\!\!\!\!\!\!$} (3)
(9) edge node[above]{\footnotesize$g^{4} \beta$} (8)
(10) edge node[right]{\footnotesize$\!g^9 \alpha$} (4)
(10) edge node[above]{\footnotesize$g^{3} \beta$} (9)
(11) edge node[left]{\footnotesize$g^{15} \alpha$} (1)
(11) edge node[left]{\footnotesize$g^2 \beta$} (10)
(12) edge node[left]{\footnotesize$g^{5} \alpha\!$} (2)
(12) edge node[below]{\footnotesize$g \beta$} (11) ;
\end{tikzpicture}
\]
and $g$ has  two orbits
(of lengths $8$ and $16$).
\end{example}

\begin{example}
Let $S = \bT \# \bT$,
and $T$ be the following triangulation of $S$
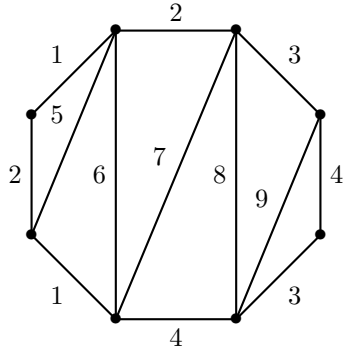
\[
\begin{tikzpicture}
[scale=.8,auto]
\node (1) at (-1,2.4) {$\bullet$};
\node (2) at (1,2.4) {$\bullet$};
\node (3) at (2.4,1) {$\bullet$};
\node (4) at (2.4,-1) {$\bullet$};
\node (5) at (1,-2.4) {$\bullet$};
\node (6) at (-1,-2.4) {$\bullet$};
\node (7) at (-2.4,-1) {$\bullet$};
\node (8) at (-2.4,1) {$\bullet$};
\coordinate (1) at (-1,2.4);
\coordinate (2) at (1,2.4);
\coordinate (3) at (2.4,1);
\coordinate (4) at (2.4,-1);
\coordinate (5) at (1,-2.4);
\coordinate (6) at (-1,-2.4);
\coordinate (7) at (-2.4,-1);
\coordinate (8) at (-2.4,1);
\draw[thick]
(1) edge node {2} (2)
(2) edge node {3} (3)
(3) edge node {4} (4)
(4) edge node {3} (5)
(5) edge node {4} (6)
(6) edge node {1} (7)
(7) edge node {2} (8)
(8) edge node {1} (1)
(5) edge node {8} (2)
(5) edge node {9} (3)
(6) edge node {6} (1)
(6) edge node {7} (2)
(7) edge node {5} (1);
\end{tikzpicture}
\]
We note that $S$ has empty boundary.
We consider the following orientation $\vec{T}$ of triangles in $T$
\[
\mbox{
  (1 5 2), (5 6 1), (2 6 7),
  (8 7 4), (3 9 8), (4 3 9).
}
\]
Then the quiver $Q(S,\vec{T})$ is of the form
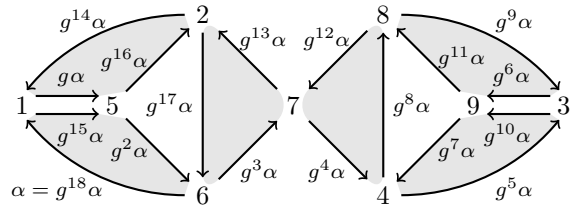
\[
\begin{tikzpicture}
[->,scale=.6]
\node (1) at (-6,0) {1};
\node (3) at (6,0) {3};
\node (5) at (-4,0) {5};
\node (7) at (0,0) {7};
\node (9) at (4,0) {9};
\node (2) at (-2,2) {2};
\node (6) at (-2,-2) {6};
\node (4) at (2,-2) {4};
\node (8) at (2,2) {8};
\fill[rounded corners=3mm,fill=gray!20] (-2,2) -- (-2,-2) -- (0,0) -- cycle;
\fill[rounded corners=3mm,fill=gray!20] (2,2) -- (2,-2) -- (0,0) -- cycle;
\fill[fill=gray!20] (-3.7,0.3) -- (-2.3,1.7) -- (-2.4,2.) arc (90:141:4.5) -- (-5.7,.2)  -- (-4.3,.2) -- cycle;
\draw[thick,->] (-2.4,2.) arc (90:120:4.5) node[above]{\footnotesize$g^{14} \alpha$} arc (120:141:4.5);
\draw[thick,->] (-5.7,.2)  -- node[above]{\footnotesize$g^{} \alpha\!\!$} (-4.3,.2) ;
\draw[thick,->] (-3.7,0.3) -- node[left]{\footnotesize$\!\!g^{16} \alpha$} (-2.3,1.7) ;
\fill[fill=gray!20] (-3.7,-0.3) -- (-2.3,-1.7) -- (-2.4,-2.) arc (270:219:4.5) -- (-5.7,-.2)  -- (-4.3,-.2) -- cycle;
\draw[thick,->] (-2.4,-2) arc (270:240:4.5) node[below]{\footnotesize$\!\!\!\!\!\!\alpha = g^{18} \alpha\ \ \ \ $} arc (240:219:4.5);
\draw[thick,->] (-5.7,-.2)  -- node[below]{\footnotesize$g^{15} \alpha\!\!\!\!\!\!\!$} (-4.3,-.2) ;
\draw[thick,->] (-3.7,-0.3) -- node[left]{\footnotesize$g^{2} \alpha$} (-2.3,-1.7) ;
\fill[fill=gray!20] (3.7,0.3) -- (2.3,1.7) -- (2.4,2.) arc (90:39:4.5) -- (5.7,.2)  -- (4.3,.2) -- cycle;
\draw[thick,->] (2.4,2.) arc (90:60:4.5) node[above]{\footnotesize$\ \ g^{9} \alpha\!\!$} arc (60:39:4.5);
\draw[thick,->] (5.7,.2)  --  node[above]{\footnotesize$\!\!\!\!g^{6} \alpha$} (4.3,.2) ;
\draw[thick,->] (3.7,0.3) -- node[right]{\footnotesize\raisebox{1ex}{$g^{11} \alpha$}} (2.3,1.7) ;
\fill[fill=gray!20] (3.7,-0.3) -- (2.3,-1.7) -- (2.4,-2.) arc (270:321:4.5) -- (5.7,-.2)  -- (4.3,-.2) -- cycle;
\draw[thick,->] (2.4,-2.) arc (270:300:4.5) node[below]{\footnotesize$\ \ g^{5} \alpha\!\!$} arc (300:321:4.5);
\draw[thick,->] (5.7,-.2)  -- node[below]{\footnotesize$\!\!\!\!\!\!g^{10} \alpha$} (4.3,-.2) ;
\draw[thick,->] (3.7,-0.3) -- node[right]{\footnotesize$g^{7} \alpha$\!\!} (2.3,-1.7) ;
\draw[thick,->]
(6) edge node[below]{\footnotesize$\ \ g^{3} \alpha\!\!$} (7)
(8) edge node[above]{\footnotesize$\!\!g^{12} \alpha\ \ $} (7)
(7) edge node[above]{\footnotesize$\ \ g^{13} \alpha\!\!\!\!$} (2)
(2) edge node[left]{\footnotesize$g^{17} \alpha$} (6)
(7) edge node[below]{\footnotesize$\!\!g^{4} \alpha\ \ $} (4)
(4) edge node[right]{\footnotesize$g^{8} \alpha$} (8);
\end{tikzpicture}
\]
There is only one $g$-orbit of arrows in $Q(S,\vec{T})$
(of length $18$).
\end{example}

\begin{example}
Let $S$ be obtained from $\bT \# \bP$ by creating one boundary component,
and $T$ the following triangulation of $S$
\[
\begin{tikzpicture}
    [scale=.8,auto]
\node (1) at (-1,2.4) {$\bullet$};
\node (2) at (1,2.4) {$\bullet$};
\node (3) at (2.4,1) {$\bullet$};
\node (4) at (2.4,-1) {$\bullet$};
\node (5) at (1,-2.4) {$\bullet$};
\node (6) at (-1,-2.4) {$\bullet$};
\node (7) at (-2.4,-1) {$\bullet$};
\node (8) at (-2.4,1) {$\bullet$};
\coordinate (1) at (-1,2.4) ;
\coordinate (2) at (1,2.4) ;
\coordinate (3) at (2.4,1) ;
\coordinate (4) at (2.4,-1) ;
\coordinate (5) at (1,-2.4) ;
\coordinate (6) at (-1,-2.4) ;
\coordinate (7) at (-2.4,-1) ;
\coordinate (8) at (-2.4,1) ;
\draw[thick]
(1) edge node {2} (2)
(2) edge node {1} (3)
(3) edge node {2} (4)
(4) edge node {4} (5)
(5) edge node {5} (6)
(6) edge node {3} (7)
(7) edge node {3} (8)
(8) edge node {1} (1)
(8) edge node {6} (2)
(8) edge node {7} (3)
(8) edge node {8} (4)
(8) edge node {9} (5)
(8) edge node {\!\!10} (6) ;
\end{tikzpicture}
\]
with two edges $4$ and $5$ on the boundary.
Consider the following orientation $\vec{T}$ of triangles in $T$
\[
\mbox{
  (1 2 6), (6 1 7), (7 2 8),
  (8 4 9), (9 5 10), (3 3 10).
}
\]
Then the quiver $Q(S,\vec{T})$ is of the form
\[
\begin{tikzpicture}
[->,scale=.9]
\coordinate (6) at (0,1.25);
\coordinate (6l) at (-0.15,1.25);
\coordinate (6p) at (0.15,1.25);
\coordinate (7) at (1.5,2);
\coordinate (1) at (0,2.75);
\coordinate (1l) at (-0.15,2.75);
\coordinate (1p) at (0.15,2.75);
\coordinate (2) at (-1.5,2);
\coordinate (8) at (0,0);
\coordinate (3) at (-3,-4) ;
\coordinate (4) at (1,-1.5) ;
\coordinate (5) at (0,-3) ;
\coordinate (9) at (-1,-1.5) ;
\coordinate (10) at (-2,-3) ;
\fill[fill=gray!20] (.3,0) arc (-70:-2:2) -- (1.39,1.8) arc (-10:-170:1.4) -- (-1.6,1.8) arc (-178:-110:2) -- cycle;
\fill[fill=gray!20] (6p) -- (7) -- (1p) -- cycle;
\fill[fill=gray!20] (6l) -- (2) -- (1l) -- cycle;
\fill[rounded corners=3mm,fill=gray!20] (4) -- (8) -- (9) -- cycle;
\fill[rounded corners=3mm,fill=gray!20] (5) -- (9) -- (10) -- cycle;
\fill[fill=gray!20] (-2.3,-3) arc (90:180:0.7) -- (-2.7,-4) arc (-90:0:0.7) -- cycle;
\node [circle,minimum size=1.5](A) at (-3,-4) { };
\node [circle,minimum size=1cm](B) at (-3.45,-4.45) {};
\coordinate  (C) at (intersection 2 of A and B);
\coordinate  (D) at (intersection 1 of A and B);
 \tikzAngleOfLine(B)(D){\AngleStart}
 \tikzAngleOfLine(B)(C){\AngleEnd}
\fill[gray!20]%
   let \p1 = ($ (B) - (D) $), \n2 = {veclen(\x1,\y1)}
   in
     (D) arc (\AngleStart-360:\AngleEnd:\n2); 
\node [fill=white,circle,minimum size=1.5](A) at (-3,-4) { };
\draw[thick,->]%
   let \p1 = ($ (B) - (D) $), \n2 = {veclen(\x1,\y1)}
   in
     (B) ++(60:\n2) node[right]{\footnotesize\ \ \ \raisebox{-10ex}{$\beta = g \beta$}}
     (D) arc (\AngleStart-360:\AngleEnd:\n2); 
\node [circle,minimum size=1.5](A) at (1,-1.5) { };
\node [circle,minimum size=1cm](B) at (1.45,-1.95) {};
\coordinate  (C) at (intersection 2 of A and B);
\coordinate  (D) at (intersection 1 of A and B);
 \tikzAngleOfLine(B)(D){\AngleStart}
 \tikzAngleOfLine(B)(C){\AngleEnd}
\fill[gray!20]%
   let \p1 = ($ (B) - (D) $), \n2 = {veclen(\x1,\y1)}
   in
     (D) arc (\AngleStart-360:\AngleEnd:\n2); 
\node [fill=white,circle,minimum size=1.5](A) at (1,-1.5) { };
\draw[thick,->]%
   let \p1 = ($ (B) - (D) $), \n2 = {veclen(\x1,\y1)}
   in
     (B) ++(60:\n2) node[right]{\footnotesize\ \ \ \raisebox{-7ex}{$g^3 \alpha$}}
     (D) arc (\AngleStart-360:\AngleEnd:\n2); 
\node [circle,minimum size=1.5](A) at (0,-3) { };
\node [circle,minimum size=1cm](B) at (0.45,-3.45) {};
\coordinate  (C) at (intersection 2 of A and B);
\coordinate  (D) at (intersection 1 of A and B);
 \tikzAngleOfLine(B)(D){\AngleStart}
 \tikzAngleOfLine(B)(C){\AngleEnd}
\fill[gray!20]%
   let \p1 = ($ (B) - (D) $), \n2 = {veclen(\x1,\y1)}
   in
     (D) arc (\AngleStart-360:\AngleEnd:\n2); 
\node [fill=white,circle,minimum size=1.5](A) at (0,-3) { };
\draw[thick,->]%
   let \p1 = ($ (B) - (D) $), \n2 = {veclen(\x1,\y1)}
   in
     (B) ++(60:\n2) node[right]{\footnotesize\ \ \ \raisebox{-7ex}{$g^6 \alpha$}}
     (D) arc (\AngleStart-360:\AngleEnd:\n2); 
\node (6) at (0,1.25) {6};
\node (6l) at (-0.15,1.25) {};
\node (6p) at (0.15,1.25) {};
\node[fill=white,circle,minimum size=2]  (7) at (1.5,2) { };
\node (7) at (1.5,2) {7};
\node (1) at (0,2.75) {1};
\node (1l) at (-0.15,2.75) {};
\node (1p) at (0.15,2.75) {};
\node[fill=white,circle,minimum size=2] (2) at (-1.5,2) { };
\node (2) at (-1.5,2) {2};
\node (8) at (0,0) {8};
\node (3) at (-3,-4) {3};
\node (4) at (1,-1.5) {4};
\node (5) at (0,-3) {5};
\node (9) at (-1,-1.5) {9};
\node (10) at (-2,-3) {10};
\draw[thick,->] (.3,0) arc (-70:-36:2) node[right]{\footnotesize \raisebox{0ex}{$g^{12} \alpha$}} arc (-36:-2:2);
\draw[thick,->] (1.39,1.8) arc (-10:-90:1.4) node[below]{\footnotesize \raisebox{2ex}{$g^{16} \alpha$}} arc (-90:-170:1.4) ;
\draw[thick,->] (-1.6,1.8) arc (-178:-144:2) node[left]{\footnotesize \raisebox{-3ex}{$g \alpha$}} arc (-144:-110:2);
\draw[thick,<-]
(6p) edge node[below]{\footnotesize$g^{13} \alpha$} (7)
(6l) edge node[below]{\footnotesize$g^{17} \alpha$} (2)
(7) edge node[above]{\footnotesize$g^{15} \alpha$} (1p)
(1l) edge node[left]{\footnotesize$g^{14} \alpha$\!} (6l)
(1p) edge node[right]{\footnotesize\!$g^{18} \alpha$} (6p)
(2) edge node[above]{\footnotesize$\!\!\!\!\!\!\!\!\!\!\!\!\!\!g^{19} \alpha = \alpha$} (1l)  ;
\draw[thick,->] (-2.3,-3) arc (90:135:0.7) node[left]{\footnotesize$g^{8} \alpha$} arc (135:180:0.7);
\draw[thick,->] (-2.7,-4) arc (-90:-45:0.7) node[right]{\footnotesize$g^{9} \alpha$} arc (-45:0:0.7);
\draw[thick,->]
(4) edge node[below]{\footnotesize$g^4 \alpha$} (9)
(9) edge node[left]{\footnotesize$g^{11} \alpha$} (8)
(8) edge node[right]{\footnotesize$g^2 \alpha$} (4)
(5) edge node[below]{\footnotesize$g^7 \alpha$} (10)
(10) edge node[left]{\footnotesize$g^{10} \alpha$} (9)
(9) edge node[right]{\footnotesize$g^5 \alpha$} (5)  ;
\end{tikzpicture}
\]
We note that there are two $g$-orbits of arrows in $Q(S,\vec{T})$,
of lengths $1$ and $19$.
\end{example}


We will now show that every triangulation quiver comes from a directed
triangulated surface.

\begin{theorem}
\label{th:4.11}
Let $(Q,f)$ be a triangulation quiver
with at least three vertices.
Then there exists a directed triangulated surface
$(S,\vec{T})$ such that $(Q,f) = Q(S,\vec{T})$.
\end{theorem}

\begin{proof}
Let $Q = (Q_0,Q_1,s,t)$.
We denote by $n(Q,f)$ the number of $f$-orbits
in $Q_1$ of length $3$.
We will prove the theorem by induction on $n(Q,f)$.
Observe that if $n(Q,f) = 1$ then $(Q,f)$
is the triangulation quiver described in Example~\ref{ex:4.3},
because $Q$ is a connected $2$-regular quiver with $|Q_0|\geq 3$.
Further, all possible triangulation quivers with three vertices
are described in Examples \ref{ex:4.3}, \ref{ex:4.4}, \ref{ex:4.5}.
Therefore, we may assume that $|Q_0|\geq 4$ and $n(Q,f) \geq 2$.
We shall consider two cases.

\smallskip

(1)
Assume that there is an $f$-orbit of length $3$
in $(Q,f)$ containing a loop.
Then $Q$ contains a subquiver
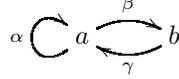
\[
  \xymatrix{ a \ar@(dl,ul)[]^{\alpha} \ar@/^1.5ex/[r]^{\beta} & b \ar@/^1.5ex/[l]^{\gamma}}
\]
with
$f(\alpha) = \beta$,
$f(\beta) = \gamma$,
$f(\gamma) = \alpha$.
Consider the quiver $Q' = (Q_0',Q_1',s',t')$
obtained from $Q$ by removing the vertex $a$,
the arrows $\alpha$, $\beta$, $\gamma$,
and adding a loop $\varepsilon$ at  vertex $b$.
Then we have the permutation $f' : Q_1' \to Q_1'$
such that $f'(\sigma) = f(\sigma)$ for any arrow
$\sigma \in Q_1 \setminus \{\alpha, \beta, \gamma\}$
and $f'(\varepsilon): = \varepsilon$.
Hence $(Q',f')$ is a triangulation quiver with
$|Q_0'| = |Q_0| - 1 \geq 3$.
By the inductive assumption,
there is a directed triangulated surface $(S',\vec{T}')$,
with $T'$ given by a finite $2$-dimensional cell complex
structure on $S'$, such that $(Q',f') = Q(S',\vec{T}')$.
Moreover, the loop $\varepsilon$ of $Q'$ is created
by a bordered edge $b$ of the triangulation $T'$ of $S'$.
Consider the surface $S''$ obtained from the projective
plane $\bP$ by creating one boundary component,
and its triangulation $T''$
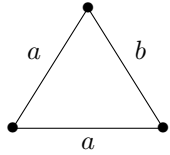
\[
\begin{tikzpicture}[auto]
\coordinate (a) at (0,1.6);
\coordinate (b) at (-1,0);
\coordinate (c) at (1,0);
\draw (c) to node {$a$} (b)
(b) to node {$a$} (a);
\draw (a) to node {$b$} (c);
\node (a) at (0,1.6) {$\bullet$};
\node (b) at (-1,0) {$\bullet$};
\node (c) at (1,0) {$\bullet$};
\end{tikzpicture}
\]
with  boundary edge $b$ and  self-folded edge $a$.
Moreover, let $\vec{T}''$ be the orientation $(a a b)$ of $T''$.
Let $\phi'_b : D^1 \to S'$ be the characteristic map
of the cell complex structure defining $(S',T')$
whose image is the edge $b$,
and $\phi''_b : D^1 \to S''$ the characteristic map
of the cell complex structure defining $(S'',T'')$
whose image is the edge $b$.
Denote by $S$ the quotient space of the disjoint union
$S' \sqcup S''$ under the identification
$\phi'_b(x) \sim \phi''_b(x)$ for all $x \in D^1$.
Then we have on $S$ the cell complex structure
induced by the cell complex structures of $S'$ and $S''$
and the characteristic map $\phi_b : D^1 \to S$,
whose image is the edge $b$,
obtained by gluing the two edges $b$ in $S'$ and $S''$,
and replacing the characteristic maps
$\phi'_b$ and $\phi''_b$.
In particular, applying Proposition~\ref{prop:4.1},
we infer that $S$ is a surface with the triangulation
$T = T' \sqcup T''$, and the orientation
$\vec{T}$ of triangles in $T$ given by the orientations
$\vec{T}'$ and $\vec{T}''$ of triangles in $T'$ and $T''$.
Moreover, we have $(Q,f) = Q(S,\vec{T})$.

\smallskip

(2)
Assume that there is no loop in any $f$-orbit of length $3$ in $Q_1$.
Then $Q$ contains a subquiver
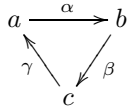
\[
  \xymatrix@C=.8pc@R=1.5pc
  {a \ar[rr]^{\alpha} && b \ar[ld]^{\beta} \\ & c \ar[lu]^{\gamma}}
\]
with
$f(\alpha) = \beta$,
$f(\beta) = \gamma$,
$f(\gamma) = \alpha$,
and $a,b,c$ are pairwise different vertices.
Consider the quiver $Q' = (Q_0',Q_1',s',t')$
obtained from $Q$ by removing
the arrows $\alpha$, $\beta$, $\gamma$,
and adding the loops
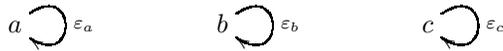
\[
  \xymatrix{ a \ar@(ur,dr)[]^{\varepsilon_a}}
  \qquad
  \qquad
  \xymatrix{ b \ar@(ur,dr)[]^{\varepsilon_b}}
  \qquad
  \qquad
  \xymatrix{ c \ar@(ur,dr)[]^{\varepsilon_c}}
\]
at the vertices $a,b,c$.
Then $Q'$ is a finite $2$-regular quiver with $Q_0' = Q_0$
and it has at most three connected components.
Moreover, there is the permutation $f' : Q_1' \to Q_1'$
such that $f'(\sigma) = f(\sigma)$ for any arrow
$\sigma \in Q_1 \setminus \{\alpha, \beta, \gamma\}$
and
$f'(\varepsilon_a) = \varepsilon_a$,
$f'(\varepsilon_b) = \varepsilon_b$,
$f'(\varepsilon_c) = \varepsilon_c$.
For each $i \in \{a,b,c\}$,
denote by $Q(i) = (Q(i)_0,Q(i)_1,s(i),t(i))$
the connected component of $Q'$ containing the vertex $i$,
and by $f_i : Q(i)_1 \to Q(i)_1$ the restriction
of $f'$ to $Q(i)_1$.
Observe that each $(Q(i), f_i)$ is a triangulation quiver
with $n(Q(i), f_i) \leq n(Q,f) - 1$.
Moreover, by the assumption imposed on the $f$-orbits
in $Q_1$,
we conclude that either $|Q(i)_0| \geq 3$ or $|Q(i)_0| = 1$.
Clearly, if $|Q(i)_0| = 1$ then $Q(i)$ is the loop
$\varepsilon_i$ at $i$.
Since $|Q_0| \geq 4$, we conclude that $|Q(i)_0| \geq 3$
for some $i \in \{a,b,c\}$.
We may assume that $|Q(a)_0| \geq 3$, and
$|Q(c)_0| = 1$, if $|Q(i)_0| = 1$ for some $i \in \{a,b,c\}$.
For each $i \in \{a,b,c\}$ with $|Q(i)_0| \geq 3$,
it follows from the inductive assumption that
$(Q(i),f_i) = Q(S(i),\vv{T(i)})$
for a directed triangulated surface $(S(i),\vv{T(i)})$.
Observe also that,
if $Q(i) = Q(j)$ for some $i \neq j$ in $\{a,b,c\}$,
then $|Q(i)_0| = |Q(j)_0| \geq 3$.
In such a case, we assume that
$(S(i),\vv{T(i)}) = (S(j),\vv{T(j)})$.
We may assume (without loss of generality)
that, if $Q'$ has at most two connected components,
then $Q(a) = Q(b)$.
We define the topological space $S'$ as follows:
\begin{itemize}
 \item
  $S' = S(a) \sqcup S(b) \sqcup S(c)$,
  if $Q(a)$, $Q(b)$, $Q(c)$ are pairwise different
  with
  $|Q(a)_0| \geq 3$,
  $|Q(b)_0| \geq 3$,
  $|Q(c)_0| \geq 3$;
 \item
  $S' = S(a) \sqcup S(b)$,
  if $Q(a)$, $Q(b)$, $Q(c)$ are pairwise different
  with
  $|Q(a)_0| \geq 3$,
  $|Q(b)_0| \geq 3$,
  $|Q(c)_0| = 1$;
 \item
  $S' = S(a)$,
  if $Q(a)$, $Q(b)$, $Q(c)$ are pairwise different
  with
  $|Q(a)_0| \geq 3$,
  $|Q(b)_0| = 1$,
  $|Q(c)_0| = 1$;
 \item
  $S' = S(a) \sqcup S(c)$,
  if $Q(a) = Q(b)$, different from $Q(c)$,
  and
  $|Q(c)_0| \geq 3$;
 \item
  $S' = S(a)$,
  if $Q(a) = Q(b)$, different from $Q(c)$,
  and
  $|Q(c)_0| = 1$;
 \item
  $S' = S(a)$,
  if $Q(a) = Q(b) = Q(c)$.
\end{itemize}
Observe that there is the finite $2$-dimensional
cell complex structure on $S'$,
given by the finite $2$-dimensional
cell complex structures on the surfaces $S(i)$,
defining the triangulations $T(i)$,
for $i \in \{a,b,c\}$ with $|Q(i)_0| \geq 3$,
and consequently the induced triangulation
$T'$ of $S'$.
We denote by $\vec{T}'$ the orientation of
triangles in $T'$ given by the orientations
$\vv{T(i)}$ of triangles of $T(i)$ in $S(i)$,
for all $i \in \{a,b,c\}$ with $|Q(i)_0| \geq 3$.
Moreover, for any $i \in \{a,b,c\}$ with $|Q(i)_0| \geq 3$,
we denote by $\phi'_i : D^1 \to S'$ the characteristic map
of the defined cell complex structure on $S'$
whose image is the edge $i$.

Consider now the triangle $S'' = T''$
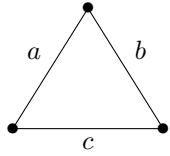
\[
\begin{tikzpicture}[auto]
\coordinate (a) at (0,1.6);
\coordinate (b) at (-1,0);
\coordinate (c) at (1,0);
\draw (a) to node {$b$} (c)
(c) to node {$c$} (b);
\draw (b) to node {$a$} (a);
\node (a) at (0,1.6) {$\bullet$};
\node (b) at (-1,0) {$\bullet$};
\node (c) at (1,0) {$\bullet$};
\end{tikzpicture}
\]
with the three pairwise different edges,
forming the boundary of $S''$,
and the orientation
$\vec{T}'' = (a b c)$
(see Example~\ref{ex:4.3}).
Let
$\phi''_a : D^1 \to S''$,
$\phi''_b : D^1 \to S''$,
$\phi''_c : D^1 \to S''$
be the characteristic maps
of the $2$-dimensional cell complex
structure on $S'' = T''$
whose images are respectively the edges $a$, $b$, $c$.

Let $S$ be the quotient space of $S' \sqcup S''$
under the identification
$\phi'_i(x) \sim \phi''_i(x)$ for all $x \in D^1$
and $i \in \{a,b,c\}$ with $|Q(i)_0| \geq 3$.
Then, applying Proposition~\ref{prop:4.1} again,
we conclude that $S$ is a surface with
a $2$-dimensional cell complex structure
defining the triangulation
$T = T' \sqcup T''$, and the orientation
$\vec{T}$ of triangles in $T$ given by the orientations
$\vec{T}'$ and $\vec{T}''$ of triangles in $T'$ and $T''$.
It follows from the above construction that
$(Q,f) = Q(S,\vec{T})$.
\end{proof}


\begin{corollary}
\label{cor:4.12}
Let $(Q,f)$ be a triangulation quiver
with at least three vertices.
Then $Q$ contains a loop $\alpha$ with $f(\alpha) = \alpha$ 
if and only if 
$(Q,f) = Q(S,\vec{T})$
for a directed triangulated surface $(S,\vec{T})$ where $S$
has non-empty boundary.
\end{corollary}

We end this section with the comment that the setting
of directed triangulated surfaces proposed in this paper
is natural for the purposes of a self-contained representation
theory of symmetric tame algebras of non-polynomial growth
which we are currently developing.
In particular, the realization Theorem~\ref{th:4.11} gives
the option of changing orientation of any triangle independently.


\section{Weighted surface algebras}\label{sec:weightsurfalg}

In this section we define  weighted surface algebras
of directed triangulated surfaces
and describe their basic properties.

Let $(Q,f)$ be a triangulation quiver.
Then we have two permutations
$f : Q_1 \to Q_1$
and
$g : Q_1 \to Q_1$
on the set $Q_1$ of arrows of $Q$ such that $f^3$
is the identity on $Q_1$ and $g = \bar{f}$,
where $\bar{ } : Q_1 \to Q_1$
is the involution which assigns to an arrow
$\alpha \in Q_1$ the arrow $\bar{\alpha}$
with $s({\alpha}) = s(\bar{\alpha})$
and ${\alpha} \neq \bar{\alpha}$.
For each arrow $\alpha \in Q_1$, we denote by
$\cO(\alpha)$ the $g$-orbit of $\alpha$
in $Q_1$, and set
$n_{\alpha} = n_{\cO(\alpha)} = |\cO(\alpha)|$.
Recall that  $\cO(g)$ is the set
of all $g$-orbits in $Q_1$.
A function
\[
  m_{\bullet} : \cO(g) \to \bN^* = \bN \setminus \{0\}
\]
is said to be a \emph{weight function} of $(Q,f)$,
and a function
\[
  c_{\bullet} : \cO(g) \to K^* = K \setminus \{0\}
\]
is said to be a \emph{parameter function} of $(Q,f)$.
We write briefly $m_{\alpha} = m_{\cO(\alpha)}$
and $c_{\alpha} = c_{\cO(\alpha)}$ for $\alpha \in Q_1$.
\emph{In this paper, we will assume that $m_{\alpha} n_{\alpha} \geq 3$
for any arrow $\alpha \in Q_1$.}

\medskip

For any arrow $\alpha \in Q_1$, we consider the path
\begin{align*}
  A_{\alpha} &= \Big( \alpha g(\alpha) \dots g^{n_{\alpha}-1}(\alpha)\Big)^{m_{\alpha}-1}
             \alpha g(\alpha) \dots g^{n_{\alpha}-2}(\alpha) , \mbox{ if } n_{\alpha} \geq 2, \\
  A_{\alpha} &= \alpha^{m_{\alpha}-1} , \mbox{ if } n_{\alpha} = 1 ,
\end{align*}
in $Q$ of length $m_{\alpha} n_{\alpha} - 1$ from
$s(\alpha)$ to $t(g^{n_{\alpha}-2}(\alpha))$.
Moreover, for any arrow $\alpha \in Q_1$, we have
the oriented cycle
\[
  B_{\alpha} = \Big( \alpha g(\alpha) \dots g^{n_{\alpha}-1}(\alpha)\Big)^{m_{\alpha}}
\]
of length $m_{\alpha} n_{\alpha}$.

\medskip
\begin{definition}\label{def:WTalgebra}\normalfont
Let $(Q,f)$ be a triangulation quiver with weight and parameter functions
$m_{\bullet}$ and $c_{\bullet}$.
We define  the bound quiver algebra
\[
  \Lambda(Q,f,m_{\bullet},c_{\bullet})
   = K Q / I (Q,f,m_{\bullet},c_{\bullet}),
\]
where $I (Q,f,m_{\bullet},c_{\bullet})$
is the admissible ideal in the path algebra $KQ$ of $Q$ over $K$
generated by:
\begin{enumerate}[(1)]
 \item
  ${\alpha} f({\alpha}) - c_{\bar{\alpha}} A_{\bar{\alpha}}$,
  for all arrows $\alpha \in Q_1$,
 \item
  $\beta f(\beta) g(f(\beta))$,
  for all arrows $\beta \in Q_1$.
\end{enumerate}
Then $\Lambda(Q,f,m_{\bullet},c_{\bullet})$ is called a
\emph{weighted triangulation algebra} of $(Q,f)$.
\end{definition}

We note that 
$Q$ is the quiver of the algebra $\Lambda(Q,f,m_{\bullet},c_{\bullet})$, 
and
the ideal $I(Q,f,m_{\bullet},c_{\bullet})$
is an admissible ideal of $K Q$, by the 
assumption that $m_{\alpha} n_{\alpha} \geq 3$
for all arrows $\alpha \in Q_1$.
We also note that a weighted triangulation algebra defined above
is a triangulation algebra defined in \cite{La3,L4}
as a quotient of complete path algebra of the quiver
by a closed ideal 
(see \cite[Definition~2.10 and Theorem~1.1(a)]{La3}
or \cite[Definition~5.16 and Proposition~7.4]{L4}).

\begin{definition}\label{def:degWSA}
\normalfont Consider the bound quiver algebra
\[
  B(Q,f,m_{\bullet},c_{\bullet})
   = K Q / J(Q,f,m_{\bullet},c_{\bullet}),
\]
where $J(Q,f,m_{\bullet},c_{\bullet})$
is the admissible ideal in the path algebra $KQ$ of $Q$ over $K$
generated by:
\begin{enumerate}[(1)]
 \item
  $c_{\alpha} B_{\alpha}
   - c_{\bar{\alpha}} B_{\bar{\alpha}}$,
  for all arrows $\alpha \in Q_1$,
 \item
  $\beta f(\beta)$,
  for all arrows $\beta \in Q_1$.\\
   We call this algebra a
  \emph{biserial weighted triangulation algebra}.
\end{enumerate}
\end{definition}

We note that a biserial weighted triangulation algebra
is a Brauer graph algebra.
In fact, it is shown in \cite{ES7}
that the class of Brauer graph algebras coincides with
the class of indecomposable idempotent algebras
of biserial weighted triangulation algebras
(we refer to \cite{ES7} for related references and results).

Let $\Lambda = \Lambda(Q,f,m_{\bullet},c_{\bullet})$ be a weighted triangulation algebra.
In order to study modules in $\mod \Lambda$ and properties of $\Lambda$,
we  specify a suitable basis of the algebra $\Lambda$,
defined in terms of the permutations $f$ and $g$.
We will identify an element of $K Q$ with its
residue class in $\Lambda = KQ/I(Q,f,m_{\bullet},c_{\bullet})$.
We will need also an extra notation.
For each arrow $\alpha$ in $Q_1$, we denote by
$A'_{\alpha}$ the subpath of
$A_{\alpha}$ from $t(\alpha)$ to $t(g^{n_{\alpha}-2})$
of length $m_{\alpha} n_{\alpha} - 2$ such that
$\alpha A'_{\alpha} = A_{\alpha}$.
We note that $A'_{\alpha}$ is a path of length $\geq 1$
since we assume that  
$m_{\alpha} n_{\alpha} \geq 3$.

\begin{lemma}
\label{lem:5.3}
Let $\alpha$ be an arrow in $Q$.
We have in $\Lambda$ the equalities:
\begin{enumerate}[(i)]
 \item
  $f^2(\alpha) = g^{n_{\bar{\alpha}} - 1}(\bar{\alpha})$.
 \item
  $A_{\bar{\alpha}}  f^2(\alpha) = B_{\bar{\alpha}}$.
 \item
  $\alpha A_{g(\alpha)}  = B_{{\alpha}}$.
 \item
  $c_{{\alpha}} B_{{\alpha}}
   = \alpha  f(\alpha) f^2(\alpha)
   = \bar{\alpha}  f(\bar{\alpha}) f^2(\bar{\alpha})
   = c_{\bar{\alpha}} B_{\bar{\alpha}}$.
 \item
  $A'_{\alpha} f^2(\bar{\alpha}) = A_{g(\alpha)}$.
\end{enumerate}
\end{lemma}

\begin{proof}
(i) \ 
The arrow $g(f^2(\alpha))$ starts at $t(f^2(\alpha))=s(\alpha)$ and we have
$g(f^2(\alpha))\neq f(f^2(\alpha))=\alpha$. Hence we have
$g(f^2(\alpha))=\bar{\alpha} = g^{n_{\bar{\alpha}}}(\bar{\alpha})$ and therefore
$f^2(\alpha) = g^{n_{\bar{\alpha}}-1}(\bar{\alpha})$. 

Part (ii)
follows from (i), and part (iii) holds by definition.

(iv)
From the relations in $\Lambda$, (iii),
and since $c_{\bullet}$ is constant on $g$-orbits, we obtain
\[
  \alpha  f(\alpha) f^2(\alpha)
    = \alpha \big( f(\alpha) f^2(\alpha) \big)
    = \alpha c_{\overbar{f(\alpha)}} A_{\overbar{f(\alpha)}}
    = c_{g(\alpha)} \alpha A_{g(\alpha)}
    = c_{{\alpha}} B_{{\alpha}}
    .
\]
Similarly, we have
$\bar{\alpha}  f(\bar{\alpha}) f^2(\bar{\alpha})
    = c_{\bar{\alpha}} B_{\bar{\alpha}}$.
Then, by (ii), we obtain
\[
  \alpha  f(\alpha) f^2(\alpha)
    = \big( \alpha  f(\alpha) \big) f^2(\alpha)
    = c_{\bar{\alpha}} A_{\bar{\alpha}} f^2(\alpha)
    = c_{\bar{\alpha}} B_{\bar{\alpha}}
    = \bar{\alpha} f(\bar{\alpha}) f^2(\bar{\alpha})
    .
\]

(v)
By (i) we have that $f^2(\bar{\alpha}) = g^{n_{\alpha}-1}(\alpha)$,
and hence the required equality holds.
\end{proof}

\begin{lemma}
\label{lem:5.4}
Let $\alpha$ be an arrow in $Q$.
Then the following hold:
\begin{enumerate}[(i)]
 \item
  $B_{\alpha} \rad \Lambda = 0$.
 \item
  $B_{\alpha}$ is non-zero.
\end{enumerate}
\end{lemma}

\begin{proof}
(i)
W must show that $B_{\alpha} \alpha = 0$
and $B_{{\alpha}} \bar{\alpha} = 0$
in $\Lambda$.
It follows from (i) and (iv) of Lemma~\ref{lem:5.3}
and the relations in $\Lambda$ that
\begin{align*}
 c_{{\alpha}} B_{{\alpha}} {\alpha}
  &= \bar{\alpha} f(\bar{\alpha}) f^2(\bar{\alpha}) {\alpha}
    = \bar{\alpha} \Big( f(\bar{\alpha}) f^2(\bar{\alpha}) g \big( f^2(\bar{\alpha}) \big) \Big)
    = 0 ,
  \\
 c_{{\alpha}} B_{{\alpha}} \bar{\alpha}
  &= \alpha f(\alpha) f^2(\alpha) \bar{\alpha}
    = \alpha \Big( f(\alpha) f^2(\alpha) g \big( f^2(\alpha) \big) \Big)
    = 0 ,
\end{align*}
and hence $B_{\alpha} \alpha = 0$
and $B_{{\alpha}} \bar{\alpha} = 0$,
because $c_{\alpha} \in K^{*}$.

(ii)
This follows from the relations defining $\Lambda$.
\end{proof}

It follows from
Lemmas \ref{lem:5.3} and \ref{lem:5.4}
that, for a vertex $i$ of $Q$ and the arrows
$\alpha$ and $\bar{\alpha}$ starting at $i$,
the element $c_{{\alpha}} B_{{\alpha}} = c_{\bar{\alpha}} B_{\bar{\alpha}}$
generates the socle of the projective module $e_i \Lambda$.
The next lemma shows that, for any vertex
$i$ of $Q$, the quotient
$e_i(\rad \Lambda)^2/\soc(e_i \Lambda)$ is a direct sum of
uniserial right $\Lambda$-modules, as well as gives
most of a basis for the indecomposable projective
module $e_i \Lambda$.

\begin{lemma}
\label{lem:5.5}
Let $\alpha$ be an arrow of $Q$.
Then the following hold:
\begin{enumerate}[(i)]
 \item
  $\alpha g(\alpha) f(g(\alpha)) = 0$ in $\Lambda$.
 \item
  $\alpha g(\alpha) \Lambda$ is a uniserial right
  $\Lambda$-module, with basis given by all
  initial subwords of $B_{\alpha}$ of length $\geq 2$.
  In particular,
  $\dim_K \alpha g(\alpha) \Lambda = m_{\alpha} n_{\alpha} - 1$.
\end{enumerate}
\end{lemma}

\begin{proof}
(i)
Since $\overbar{g(\alpha)} = f(\alpha)$
we obtain the equalities
\[
  \alpha g(\alpha) f\big(g(\alpha)\big)
   = \alpha c_{\overbar{g(\alpha)}} A_{\overbar{g(\alpha)}}
   = c_{f(\alpha)} \alpha A_{f(\alpha)}
   = c_{f(\alpha)} \alpha {f(\alpha)} g\big(f(\alpha)\big)
   \dots
   = 0 ,
\]
by the relations for the algebra $\Lambda$.

(ii)
If follows from (i) that the right $\Lambda$-module
$\alpha g(\alpha) \rad \Lambda$ is generated by
$\alpha g(\alpha) g^2(\alpha)$.
Then using (i) repeatedly we conclude
that $\alpha g(\alpha) \Lambda$ is a uniserial
right $\Lambda$-module with basis formed
by all initial subwords of $B_{\alpha}$
of length $\geq 2$.
Clearly, then
$\dim_K \alpha g(\alpha) \Lambda = m_{\alpha} n_{\alpha} - 1$.
\end{proof}

\begin{corollary}
\label{cor:5.6}
Let $i$ be a vertex of $Q$ and
$\alpha,\bar{\alpha}$ the two arrows in $Q$
with source $i$.
Then
$\dim_K e_i \Lambda =
  m_{\alpha} n_{\alpha}
  + m_{\bar{\alpha}} n_{\bar{\alpha}}$.
\end{corollary}

\begin{proof}
It follows from the previous lemma,
that a basis of $\alpha g(\alpha) \Lambda$
is given by the set of initial subwords of
$B_{\alpha}$ of length $\geq 2$.
Then we also see that
$\rad e_i \Lambda$ has basis consisting of all
initial subwords of $A_{\alpha}$ and
$A_{\bar{\alpha}}$ together with $B_{\alpha}$.
This shows that
$\dim_K e_i \Lambda =
  m_{\alpha} n_{\alpha}
  + m_{\bar{\alpha}} n_{\bar{\alpha}}$.
\end{proof}

We present now basic properties of the algebras
$B(Q,f,m_{\bullet},c_{\bullet})$ and $\Lambda(Q,f,m_{\bullet},c_{\bullet})$.

\begin{proposition}
\label{prop:5.7}
Let $(Q,f)$ be a triangulation quiver,
$m_{\bullet}$ and $c_{\bullet}$
weight and parameter functions of $(Q,f)$,
and $B = B(Q,f,m_{\bullet},c_{\bullet})$.
Then the following statements hold:
\begin{enumerate}[(i)]
 \item
  $B$ is a finite-dimensional
  special biserial algebra with
  $\dim_K B = \sum_{\cO \in \cO(g)} m_{\cO} n_{\cO}^2$.
 \item
  $B$ is a symmetric algebra.
 \item
  $B$ is a tame algebra.
\end{enumerate}
\end{proposition}

\begin{proof}
We write 
$J = J(Q,f,m_{\bullet},c_{\bullet})$.

(i)
Let $i$ be a vertex of $Q$ and let $\alpha, \bar{\alpha}$ be
the two arrows in $Q$ with source $i$.
Then the indecomposable projective right $B$-module $P_i = e_i B$
 has  dimension equal to
$\dim_K P_i = m_{\alpha} n_{\alpha} + m_{\bar{\alpha}} n_{\bar{\alpha}}$.
Indeed, $P_i$ has a basis given by $e_i$, all initial subwords of
$A_{\alpha}$ and $A_{\bar{\alpha}}$, and $B_{\alpha}$.
Then we deduce that
\[
  \dim_K B = \sum_{\cO \in \cO(g)} m_{\cO} n_{\cO}^2 .
\]

(ii)
It is well known (see for example \cite[Theorem~IV.2.2]{SY})
that $B$ is a symmetric algebra if and only if it has a symmetrizing form.
That is, there exists
a $K$-linear form $\varphi : B \to K$ such that
$\varphi(a b) = \varphi(b a)$ for all $a,b \in B$
and $\Ker \varphi$ does not contain non-zero one-sided
ideal of $B$.
Let $i$ be a vertex of the quiver $Q$ and $\alpha, \bar{\alpha}$
be the  arrows with source $i$.
Then the element
$c_{\alpha} B_{\alpha} + J =
 c_{\bar{\alpha}} B_{\bar{\alpha}} + J$
generates the one-dimensional socle of the indecomposable
projective right $B$-module $P_i$ at the vertex $i$.
Clearly, we have also that
$\topp (P_i) = S_i = \soc (P_i)$.
We define a required $K$-linear form
$\varphi : B \to K$ by assigning to the coset $u+J$
of a path $u$ in $Q$ the following element in $K$
\[
  \varphi (u+J) = \left\{
   \begin{array}{cl}
    c^{-1}_{\alpha} & \mbox{if $u=B_{\alpha}$ for an arrow $\alpha \in Q_1$},\\
    0 & \mbox{otherwise},\\
   \end{array}
  \right.
\]
and extending to a $K$-linear form.

(iii)
Since $B$ is special biserial, it is tame,
by Proposition~\ref{prop:2.1}.
\end{proof}

We refer to Section~\ref{sec:tetrahedral} 
for the tetrahedral algebras and their properties.

\begin{proposition}
\label{prop:5.8}
Let $(Q,f)$ be a triangulation quiver,
$m_{\bullet}$ and $c_{\bullet}$
weight and parameter functions of $(Q,f)$,
and $\Lambda = \Lambda(Q,f,m_{\bullet},c_{\bullet})$.
Then the following statements hold:
\begin{enumerate}[(i)]
 \item
  $\Lambda$ is a finite-dimensional algebra with
  $\dim_K \Lambda = \sum_{\cO \in \cO(g)} m_{\cO} n_{\cO}^2$.
 \item
  $\Lambda$ is a symmetric algebra.
 \item
  $\Lambda$ degenerates to the algebra
  $B(Q,f,m_{\bullet},c_{\bullet})$,
  provided $\Lambda$ is not a tetrahedral algebra.
 \item
  $\Lambda$ is a tame algebra.
\end{enumerate}
\end{proposition}

\begin{proof}
We abbreviate
$I = I(Q,f,m_{\bullet},c_{\bullet})$.

(i)
It follows from Corollary~\ref{cor:5.6} that,
for each vertex $i$ of $Q$,
the indecomposable projective right
$\Lambda$-module $P_i$ at the vertex $i$
has the dimension
$\dim_K P_i = m_{\alpha} n_{\alpha} + m_{\bar{\alpha}} n_{\bar{\alpha}}$,
where $\alpha, \bar{\alpha}$ are the two arrows in $Q$ with source $i$.
Then we get
\[
  \dim_K \Lambda = \sum_{\cO \in \cO(g)} m_{\cO} n_{\cO}^2 .
\]

(ii)
Similarly, as in the above Proposition, we define
a symmetrizing $K$-linear form $\varphi : \Lambda \to K$
by assigning to the coset $u+I$
of a path $u$ in $Q$ the following element in $K$
\[
  \varphi (u+I) = \left\{
   \begin{array}{cl}
    c^{-1}_{\alpha} & \mbox{if $u=B_{\alpha}$ for an arrow $\alpha \in Q_1$},\\
    0 & \mbox{otherwise},\\
   \end{array}
  \right.
\]
and extending to a $K$-linear form.

(iii)
Assume that $\Lambda$ is not a tetrahedral algebra.
For each $t \in K$, consider the bound quiver algebra
$\Lambda(t) = K Q / I^{(t)}$, where $I^{(t)}$
is the admissible ideal in the path algebra $K Q$
of $Q$ over $K$ generated by the elements:
\begin{enumerate}[(i)]
 \item
  ${\alpha} f({\alpha}) - t c_{\bar{\alpha}} A_{\bar{\alpha}}$,
  for all arrows $\alpha \in Q_1$,
 \item
  $\beta f(\beta) g(f(\beta))$,
  for all arrows $\beta \in Q_1$.
\end{enumerate}
Then 
a simple checking shows that
$\Lambda(t)$, $t \in K$, is an algebraic family
in the variety $\alg_d(K)$, with $d = \dim_K \Lambda$,
such that $\Lambda(t) \cong \Lambda(1) = \Lambda$
for all $t \in K\setminus\{0\}$ and
$\Lambda(0) = B = B(Q,f,m_{\bullet},c_{\bullet})$.
Then it follows from Proposition~\ref{prop:2.2}
that $\Lambda$ degenerates to $B$.
We refer also to \cite[Proposition~7.13]{L4} for
a different algebraic family of intermediate algebras
degenerating $\Lambda$ to $B$.

(iv)
If $\Lambda$ is not a tetrahedral algebra then
it follows from Propositions \ref{prop:2.2} and \ref{prop:5.7}
that $\Lambda$ is  tame.
Assume $\Lambda$ is a tetrahedral algebra.
If $\Lambda$ is non-singular then the tameness 
(even polynomial growth) of $\Lambda$ follows from the old
article \cite{NeS} where the representation theory  
of the trivial extensions of arbitrary tubular algebras
has been established.
If $\Lambda$ is singular, then the tameness of $\Lambda$
follows from \cite[Theorem]{DS0} and \cite[Theorem~A]{LS}.
\end{proof}

\begin{definition}\label{def:WSA}\normalfont
Let $(S,\vec{T})$ be a directed triangulated surface,
$(Q(S,\vec{T}),f)$ the associated triangulation
quiver, and let $m_{\bullet}$ and $c_{\bullet}$ be
weight and parameter functions of $(Q(S,\vec{T}),f)$.
Then the triangulation algebra
$\Lambda(Q(S,\vec{T}),f,m_{\bullet},c_{\bullet})$
will be
called a \emph{weighted surface algebra}.
\end{definition}

For further purposes, we would like to have two notions:
a weighted surface algebra
and
a weighted triangulation algebra on the grounds,
one of topological origin and the other purely algebraic.

We give now examples of weighted surface algebras,
using the triangulation quivers from
Examples \ref{ex:4.3}, \ref{ex:4.4}, \ref{ex:4.5}.

\begin{example}
\label{ex:5.10}
Let $(Q(S,\vec{T}),f)$ be
the triangulation quiver
\[
  \xymatrix@C=.8pc@R=1.5pc
  {
     1 \ar@(dl,ul)[]^{\varepsilon} \ar[rr]^{\alpha} &&
     2 \ar@(ur,dr)[]^{\eta} \ar[ld]^{\beta} \\
     & 3 \ar@(dr,dl)[]^{\mu} \ar[lu]^{\gamma}}
\]
with  $f$-orbits
$(\alpha\ \beta\ \gamma)$,
$(\varepsilon)$,
$(\eta)$,
$(\mu)$,
considered in Example \ref{ex:4.3}.
Then $g$ has only one orbit, 
$(\alpha\ \eta\ \beta\ \mu\ \gamma\ \varepsilon)$,
and hence a weight function
$m_{\bullet} : \cO(g) \to \bN^*$
and a parameter function
$c_{\bullet} : \cO(g) \to K^*$
are given by a positive integer $m$ and
a parameter $c \in K^*$.
The associated weighted surface algebra
$\Lambda = \Lambda(Q(S,\vec{T}),f,m_{\bullet},c_{\bullet})$
is given by the above quiver and the relations
\begin{align*}
 \alpha\beta &= c(\varepsilon\alpha\eta\beta\mu\gamma)^{m-1} \varepsilon\alpha\eta\beta\mu ,
 &
 \varepsilon^2 &= c(\alpha\eta\beta\mu\gamma\varepsilon)^{m-1} \alpha\eta\beta\mu\gamma,
 &
 \alpha\beta\mu &= 0,
 &
 \varepsilon^2 \alpha &= 0,
\\
 \beta\gamma &= c(\eta\beta\mu\gamma\varepsilon\alpha)^{m-1} \eta\beta\mu\gamma\varepsilon ,
 &
 \eta^2 &= c(\beta\mu\gamma\varepsilon\alpha\eta)^{m-1} \beta\mu\gamma\varepsilon\alpha,
 &
 \beta\gamma\varepsilon &= 0,
 &
 \eta^2 \beta &= 0,
\\
 \gamma\alpha &= c(\mu\gamma\varepsilon\alpha\eta\beta)^{m-1} \mu\gamma\varepsilon\alpha\eta,
 &
 \mu^2 &= c(\gamma\varepsilon\alpha\eta\beta\mu)^{m-1} \gamma\varepsilon\alpha\eta\beta,
 &
 \gamma\alpha\eta &= 0,
 &
 \mu^2 \gamma &= 0.
\end{align*}
Moreover, the Cartan matrix $C_{\Lambda}$ of $\Lambda$
is of the form
\[
  \begin{bmatrix}
   4m & 4m & 4m \\
   4m & 4m & 4m \\
   4m & 4m & 4m
  \end{bmatrix} ,
\]
and hence is singular.
\end{example}

\begin{example}
\label{ex:5.11}
Let $(Q(S,\vec{T}),f)$ be
the triangulation quiver
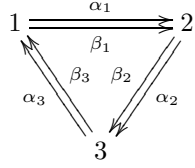
\[
  \xymatrix@R=3.pc@C=1.8pc{
    1
    \ar@<.35ex>[rr]^{\alpha_1}
    \ar@<-.35ex>[rr]_{\beta_1}
    && 2
    \ar@<.35ex>[ld]^{\alpha_2}
    \ar@<-.35ex>[ld]_{\beta_2}
    \\
    & 3
    \ar@<.35ex>[lu]^{\alpha_3}
    \ar@<-.35ex>[lu]_{\beta_3}
  }
\]
with  $f$-orbits
$(\alpha_1\ \alpha_2\ \alpha_3)$ and $(\beta_1\ \beta_2\ \beta_3)$,
considered in Example \ref{ex:4.4}.
Then $g$ has only one orbit, which is
$(\alpha_1\ \beta_2\ \alpha_3\ \beta_1\ \alpha_2\ \beta_3)$,
and hence
a weight function
$m_{\bullet} : \cO(g) \to \bN^*$
and
a parameter function
$c_{\bullet} : \cO(g) \to K^*$
are given by a positive integer $m$ and
a parameter $c \in K^*$.
The associated weighted surface algebra
$\Lambda = \Lambda(Q(S,\vec{T}),f,m_{\bullet},c_{\bullet})$
is given by the above quiver and the relations
\begin{align*}
 \alpha_1 \alpha_2 &= c(\beta_1 \alpha_2 \beta_3 \alpha_1 \beta_2 \alpha_3)^{m-1}  \beta_1 \alpha_2 \beta_3 \alpha_1 \beta_2,
 &
 \alpha_1 \alpha_2 \beta_3 &= 0,
\\
 \alpha_2 \alpha_3 &= c(\beta_2 \alpha_3 \beta_1 \alpha_2 \beta_3 \alpha_1)^{m-1}  \beta_2 \alpha_3 \beta_1 \alpha_2 \beta_3,
 &
 \alpha_2 \alpha_3 \beta_1 &= 0,
\\
 \alpha_3 \alpha_1 &= c(\beta_3 \alpha_1 \beta_2 \alpha_3 \beta_1 \alpha_2)^{m-1}  \beta_3 \alpha_1 \beta_2 \alpha_3 \beta_1,
 &
 \alpha_3 \alpha_1 \beta_2 &= 0,
\\
 \beta_1 \beta_2 &= c(\alpha_1 \beta_2 \alpha_3 \beta_1 \alpha_2 \beta_3)^{m-1}  \alpha_1 \beta_2 \alpha_3 \beta_1 \alpha_2,
 &
 \beta_1 \beta_2 \alpha_3 &= 0,
\\
 \beta_2 \beta_3 &= c(\alpha_2 \beta_3 \alpha_1 \beta_2 \alpha_3 \beta_1)^{m-1}  \alpha_2 \beta_3 \alpha_1 \beta_2 \alpha_3,
 &
 \beta_2 \beta_3 \alpha_1 &= 0,
\\
 \beta_3 \beta_1 &= c(\alpha_3 \beta_1 \alpha_2 \beta_3 \alpha_1 \beta_2)^{m-1}  \alpha_3 \beta_1 \alpha_2 \beta_3 \alpha_1,
 &
 \beta_3 \beta_1 \alpha_2 &= 0.
\end{align*}
Moreover, the Cartan matrix $C_{\Lambda}$ of $\Lambda$
is of the form
\[
  \begin{bmatrix}
   4m & 4m & 4m \\
   4m & 4m & 4m \\
   4m & 4m & 4m
  \end{bmatrix} ,
\]
and hence is singular.
\end{example}

\begin{example}
\label{ex:5.12}
Let $(Q(S,\vec{T}),f)$ be
the triangulation quiver
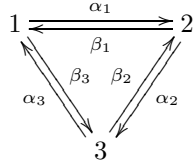
\[
  \xymatrix@R=3.pc@C=1.8pc{
    1
    \ar@<.35ex>[rr]^{\alpha_1}
    \ar@<.35ex>[rd]^{\beta_3}
    && 2
    \ar@<.35ex>[ll]^{\beta_1}
    \ar@<.35ex>[ld]^{\alpha_2}
    \\
    & 3
    \ar@<.35ex>[lu]^{\alpha_3}
    \ar@<.35ex>[ru]^{\beta_2}
  }
\]
with  $f$-orbits
$(\alpha_1\ \alpha_2\ \alpha_3)$ and $(\beta_1\ \beta_3\ \beta_2)$,
considered in Example \ref{ex:4.4}.
Then $\cO(g)$ consists of the three $g$-orbits
$(\alpha_1\ \beta_1)$ $(\alpha_2\ \beta_2)$, $(\alpha_3\ \beta_3)$
of length $2$.
Let $m_{\bullet} : \cO(g) \to \bN^*$
be a weight function
and
$m_1 = m_{\alpha_1}$,
$m_2 = m_{\alpha_2}$,
$m_3 = m_{\alpha_3}$.
By our assumption, we must take
$m_1 \geq 2$,
$m_2 \geq 2$,
$m_3 \geq 2$,
because
$|\cO(\alpha_1)| = 2$,
$|\cO(\alpha_2)| = 2$,
$|\cO(\alpha_3)| = 2$.
Let $c_{\bullet} : \cO(g) \to K^*$
be a parameter function
and
$c_1 = c_{\alpha_1}$,
$c_2 = c_{\alpha_2}$,
$c_3 = c_{\alpha_3}$.
Then the associated weighted surface algebra
$\Lambda = \Lambda(Q(S,\vec{T}),f,m_{\bullet},c_{\bullet})$
is given by the above quiver and the relations
\begin{align*}
 \alpha_1 \alpha_2 &= c_3(\beta_3 \alpha_3)^{m_3-1} \beta_3 ,
 &
 \alpha_2 \alpha_3 &= c_1(\beta_1 \alpha_1)^{m_1-1} \beta_1 ,
 &
 \alpha_3 \alpha_1 &= c_2(\beta_2 \alpha_2)^{m_2-1} \beta_2 ,
\\
 \beta_1 \beta_3 &= c_2(\alpha_2 \beta_2)^{m_2-1} \alpha_2 ,
 &
 \beta_3 \beta_2 &= c_1(\alpha_1 \beta_1)^{m_1-1} \alpha_1 ,
 &
 \beta_2 \beta_1 &= c_3(\alpha_3 \beta_3)^{m_3-1} \alpha_3 ,
\\
 \alpha_1 \alpha_2 \beta_2 &= 0,
 &
 \!\!\!\!\!\!\!\!\!\!\!\!\!\!\alpha_2 \alpha_3 \beta_3 &= 0,
 &
 \!\!\!\!\!\!\!\!\!\!\!\!\!\!\alpha_3 \alpha_1 \beta_1 &= 0,
\\
 \beta_1 \beta_3 \alpha_3 &= 0,
 &
 \!\!\!\!\!\!\!\!\!\!\!\!\!\!\beta_3 \beta_2 \alpha_2 &= 0,
 &
 \!\!\!\!\!\!\!\!\!\!\!\!\!\!\beta_2 \beta_1 \alpha_1 &= 0.
\end{align*}
Moreover, the Cartan matrix $C_{\Lambda}$ of $\Lambda$
is of the form
\[
  \begin{bmatrix}
   m_1+m_3 & m_1 & m_3 \\
   m_1 & m_1+m_2 & m_2 \\
   m_3 & m_2 & m_2+m_3
  \end{bmatrix} ,
\]
and $\det C_{\Lambda} = 4 m_1 m_2 m_3$.
Hence $C_{\Lambda}$ is non-singular.
\end{example}

\begin{example}
\label{ex:5.13}
Let $(Q(S,\vec{T}),f)$ be
the triangulation quiver
\[
  \xymatrix{
    1
    \ar@(ld,ul)^{\alpha}[]
    \ar@<.5ex>[r]^{\beta}
    & 2
    \ar@<.5ex>[l]^{\gamma}
    \ar@<.5ex>[r]^{\delta}
    & 3
    \ar@<.5ex>[l]^{\sigma}
    \ar@(ru,dr)^{\varrho}[]
  } ,
\]
with  $f$-orbits
$(\alpha\ \beta\ \gamma)$
and
$(\varrho\ \sigma\ \delta)$,
considered in Example \ref{ex:4.5}.
Then $\cO(g)$ consists of the $g$-orbits
$(\alpha)$,
$(\varrho)$,
$(\beta\ \delta\ \sigma\ \gamma)$.
Let $m_{\bullet} : \cO(g) \to \bN^*$
be a weight function
and
$m_{\alpha} = p$,
$m_{\varrho} = q$,
$m_{\beta} = r$.
By our assumption, we have
$p \geq 3$
and
$q \geq 3$,
because
$|\cO(\alpha)| = 1$ and $|\cO(\varrho)| = 1$.
Moreover, let $c_{\bullet} : \cO(g) \to K^*$
be a parameter function
and
$c_{\alpha} = a$,
$c_{\varrho} = b$,
$c_{\beta} = c$.
Then the associated weighted surface algebra
$\Lambda = \Lambda(Q(S,\vec{T}),f,m_{\bullet},c_{\bullet})$
is given by the above quiver and the relations
\begin{align*}
 \alpha\beta &= c(\beta\delta\sigma\gamma)^{r-1} \beta\delta\sigma,
 &
 \beta\gamma &= a \alpha^{p-1} ,
 &
 \gamma\alpha &= c(\delta\sigma\gamma\beta)^{r-1} \delta\sigma\gamma,
\\
 \varrho\sigma &= c(\sigma\gamma\beta\delta)^{r-1} \sigma\gamma\beta,
 &
 \sigma\delta &= b \varrho^{q-1} ,
 &
 \delta\varrho &= c(\gamma\beta\delta\sigma)^{r-1} \gamma\beta\delta,
\\
 \alpha\beta\delta &= 0,
\qquad
   \beta\gamma\beta = 0,
 &
   \!\!\!\!\gamma\alpha^2 &= 0,
\quad
 \varrho\sigma\gamma = 0,\!\!\!\!\!\!\!
 &
   \sigma\delta\sigma &= 0,
\quad
   \delta\varrho^2 = 0.
\end{align*}
Moreover, the Cartan matrix $C_{\Lambda}$ of $\Lambda$
is of the form
\[
  \begin{bmatrix}
   p+r & 2r & r \\
   2r & 4r & 2r \\
   r & 2r & q+r
  \end{bmatrix} ,
\]
and $\det C_{\Lambda} = 4 p q r$.
Hence $C_{\Lambda}$ is non-singular.
\end{example}

The class of
weighted surface algebras contains as a very
special subclass the class of Jacobian algebras
of surfaces with punctures.
Recall that a \emph{surface with punctures}
is a pair $(S,P)$, where $S$ is an
orientable surface with empty boundary, and $P$ is a finite
set of points in $S$, called
{punctures}.
Then an {ideal triangulation} (briefly, {triangulation})
of $(S,P)$ is any maximal collection $T$
of pairwise compatible arcs with the ends in $P$
whose relative interiors do not intersect each other
(see \cite[Section~2]{FoST}), and the triple $(S,P,T)$
is called a \emph{triangulated surface with punctures}.
Moreover, it is always assumed that
a triangulated surface with punctures
$(S,P,T)$ satisfies the following conditions:
\begin{itemize}
 \item
  if $S$ is a sphere then $|P| \geq 4$;
 \item
  there is no arc in $P$ starting and ending at the same puncture;
 \item
  for each puncture $p \in P$, there are at least $3$ arcs in $T$ incident to $p$.
\end{itemize}
A triangulated surface with punctures $(S,P,T)$
may be viewed as directed triangulated surface $(S,\vec{T})$,
where $\vec{T}$ is one of the two possible choices
of coherent orientations of triangles in $T$, using the
fact that $S$ is orientable. 
Then the quiver $Q(S,\vec{T})$ of $(S,\vec{T})$
is the adjacency quiver $Q(S,P,T)$ of $(S,P,T)$
defined by Fomin, Shapiro and Thurston \cite{FoST}.
Moreover, the quiver $Q(S,\vec{T}) = Q(S,P,T)$ has no loops
nor $2$-cycles
$\xymatrix@=1.5pc{ \bullet \ar@<.45ex>[r] & \bullet \ar@<.45ex>[l]}$.
Finally, the Jacobian algebra of $(S,P,T)$
with respects to the Labardini-Fragoso potential \cite{LF}
is the surface algebra $\Lambda(Q,f,m_{\bullet},c_{\bullet})$
of the directed triangulated surface $(S,\vec{T})$
given by $(S,P,T)$, and the weight function $m_{\bullet}$
taking only value $1$ (see  \cite{La1}).
For an arbitrary weight function $m_{\bullet}$ of $(S,\vec{T})$
we obtain a \emph{weighted Jacobian algebra} of $(S,P,T)$,
as investigated by Ladkani \cite{La2,La3}.

\section{Tetrahedral algebras}\label{sec:tetrahedral}

In this section we present a family of algebras given by
the tetrahedral triangulation of the sphere, which has exceptional
properties among all weighted surface algebras considered
in this paper.

\begin{example}
\label{ex:6.1}
Let $\bS=S^2$ be the sphere in $\bR^3$.
Consider the tetrahedral triangulation
$T$ of $\bS$
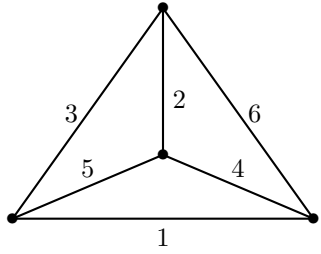
\[
\begin{tikzpicture}
[scale=1]
\node (A) at (-2,0) {$\bullet$};
\node (B) at (2,0) {$\bullet$};
\node (C) at (0,.85) {$\bullet$};
\node (D) at (0,2.8) {$\bullet$};
\coordinate (A) at (-2,0) ;
\coordinate (B) at (2,0) ;
\coordinate (C) at (0,.85) ;
\coordinate (D) at (0,2.8) ;
\draw[thick]
(A) edge node [left] {3} (D)
(D) edge node [right] {6} (B)
(D) edge node [below right] {2} (C)
(A) edge node [above] {5} (C)
(C) edge node [above] {4} (B)
(A) edge node [below] {1} (B) ;
\end{tikzpicture}
\]
and its coherent orientation $\vec{T}$
\[
\mbox{
  (1 5 4), (2 5 3), (2 6 4), (1 6 3).
}
\]
Then the associated quiver $Q(\bS,\vec{T})$
is of the form
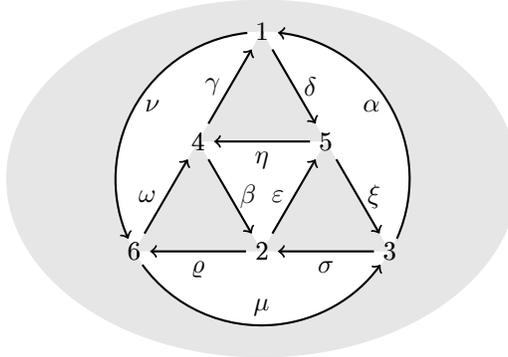
\[
\begin{tikzpicture}
[scale=.85]
\node (1) at (0,1.72) {$1$};
\node (2) at (0,-1.72) {$2$};
\node (3) at (2,-1.72) {$3$};
\node (4) at (-1,0) {$4$};
\node (5) at (1,0) {$5$};
\node (6) at (-2,-1.72) {$6$};
\coordinate (1) at (0,1.72);
\coordinate (2) at (0,-1.72);
\coordinate (3) at (2,-1.72);
\coordinate (4) at (-1,0);
\coordinate (5) at (1,0);
\coordinate (6) at (-2,-1.72);
\fill[fill=gray!20]
      (0,2.22cm) arc [start angle=90, delta angle=-360, x radius=4cm, y radius=2.8cm]
 --  (0,1.72cm) arc [start angle=90, delta angle=360, radius=2.3cm]
     -- cycle;
\fill[fill=gray!20]
    (1) -- (4) -- (5) -- cycle;
\fill[fill=gray!20]
    (2) -- (4) -- (6) -- cycle;
\fill[fill=gray!20]
    (2) -- (3) -- (5) -- cycle;

\node (1) at (0,1.72) {$1$};
\node (2) at (0,-1.72) {$2$};
\node (3) at (2,-1.72) {$3$};
\node (4) at (-1,0) {$4$};
\node (5) at (1,0) {$5$};
\node (6) at (-2,-1.72) {$6$};
\draw[->,thick] (-.23,1.7) arc [start angle=96, delta angle=108, radius=2.3cm] node[midway,right] {$\nu$};
\draw[->,thick] (-1.87,-1.93) arc [start angle=-144, delta angle=108, radius=2.3cm] node[midway,above] {$\mu$};
\draw[->,thick] (2.11,-1.52) arc [start angle=-24, delta angle=108, radius=2.3cm] node[midway,left] {$\alpha$};
\draw[->,thick]
(1) edge node [right] {$\delta$} (5)
(2) edge node [left] {$\varepsilon$} (5)
(2) edge node [below] {$\varrho$} (6)
(3) edge node [below] {$\sigma$} (2)
(4) edge node [left] {$\gamma$} (1)
(4) edge node [right] {$\beta$} (2)
(5) edge node [right] {$\xi$} (3)
(5) edge node [below] {$\eta$} (4)
(6) edge node [left] {$\omega$} (4)
;
\end{tikzpicture}
\]
where the shaded subquivers denote the $f$-orbits.

In $Q(\bS,\vec{T})$ we have the four $g$-orbits
which are, written in cycle notation,
$$(\beta \ \varepsilon \ \eta), \ \
(\varrho \ \mu \ \sigma), \ \ 
(\gamma \ \nu \ \omega), \ \ (\alpha \ \delta \ \xi).
$$
Let $m_{\bullet} : \cO(g) \to \bN^* = \bN \setminus \{0\}$
be the weight function taking the value $1$
on each $g$-orbit.
Consider a parameter function
$c_{\bullet} : \cO(g) \to K^* = K \setminus \{0\}$,
and let
$c_{\cO(\beta)} = a$,
$c_{\cO(\varrho)} = b$,
$c_{\cO(\gamma)} = c$
and
$c_{\cO(\alpha)} = d$,
for elements $a,b,c,d \in K^*$.
Then the algebra
$\Lambda(\bS,a,b,c,d) = \Lambda(\bS,\vec{T},m_{\bullet},c_{\bullet})$
is given by the above quiver
$Q(\bS,\vec{T})$ and the relations
\begin{align*}
 \delta \eta &= c\nu \omega,
  &
 \eta \gamma &= d \xi \alpha,
  &
 \gamma \delta &= a \beta \varepsilon,
  &
 \delta \eta \beta &= 0,
  &
 \eta \gamma \nu &= 0,
  &
 \gamma \delta \xi &= 0,
    \\
 \varrho \omega &= a \varepsilon \eta,
  &
 \omega \beta &= b \mu \sigma,
  &
 \beta \varrho &= c \gamma \nu,
  &
 \varrho \omega \gamma &= 0,
  &
 \omega \beta \varepsilon &= 0,
  &
 \beta \varrho \mu &= 0,
    \\
 \sigma \varepsilon &= d \alpha \delta,
  &
 \varepsilon \xi &= b \varrho \mu,
  &
 \xi \sigma &= a \eta \beta,
  &
 \sigma \varepsilon \eta &= 0,
  &
 \varepsilon \xi \alpha &= 0,
  &
 \xi \sigma \varrho &= 0,
    \\
 \alpha \nu &= b \sigma \varrho,
  &
 \nu \mu &= d \delta \xi,
  &
 \mu \alpha &= c \omega \gamma,
  &
 \alpha \nu \omega &= 0,
  &
 \nu \mu \sigma &= 0,
  &
 \mu \alpha \delta &= 0,
\end{align*}
corresponding to the four $f$-orbits in $Q(\bS,\vec{T})$ where an orbit
is given by the arrows around a shaded triangle.
Moreover, a minimal set of relations defining
$\Lambda(\bS,a,b,c,d)$
is given by the above twelve commutativity relations and the six zero
relations
\begin{align*}
 \delta \eta \beta &= 0,
  &
 \varrho \omega \gamma &= 0,
  &
 \sigma \varepsilon \eta &= 0,
  &
 \beta \varrho \mu &= 0,
  &
 \eta \gamma \nu &= 0,
  &
 \omega \beta \varepsilon &= 0,
\end{align*}
so the remaining six of the above zero relations are superfluous.

We note now that the algebra
$\Lambda(\bS,a,b,c,d)$
is isomorphic to the algebra\linebreak
$\Lambda(\bS,a b c d,1,1,1)$.
Indeed, there is an isomorphism of algebras
$\varphi  : \Lambda(\bS,a b c d,1,1,1) \to \Lambda(\bS,a,b,c,d)$
given by
\begin{align*}
&&
 \varphi(\alpha) &= a \alpha,
  &
 \varphi(\mu) &= b \mu,
  &
 \varphi(\nu) &= c \nu,
\\
&&
 \varphi(\delta) &= b c d \delta,
  &
 \varphi(\omega) &= b c d \omega,
  &
 \varphi(\sigma) &= b c d \sigma,
\\
 \varphi(\xi) &= \xi,
  &
 \varphi(\varrho) &= \varrho,
  &
 \varphi(\gamma) &= \gamma,
  &
 \varphi(\eta) &= \eta,
  &
 \varphi(\varepsilon) &= \varepsilon,
  &
 \varphi(\beta) &= \beta.
\end{align*}

An algebra $\Lambda(\bS,a,b,c,d)$,
with $a,b,c,d \in K^{*}$,
is said to be a \emph{tetrahedral algebra}.
Moreover, the triangulation quiver $Q(\bS,\vec{T})$
of $\Lambda(\bS,a,b,c,d)$ is said to be
the \emph{tetrahedral triangulation quiver}.

For each $\lambda \in K^{*} = K \setminus \{0\}$,
we abbreviate
$\Lambda(\bS,\lambda) = \Lambda(\bS,\lambda,1,1,1)$.
We shall discuss now distinguished properties of the tetrahedral algebras.
\end{example}

Recall that the trivial 
extension algebra
$\T(B) = B \ltimes D(B)$
of an algebra $B$ by the injective cogenerator
$D(B) = \Hom_K(B,K)$ has underlying $K$-vector space 
$\T(B) = B \oplus D(B)$,  and the multiplication
in $\T(B)$ is given by
\[
 (b_1,f_1)
 (b_2,f_2)
  =
 (b_1 b_2, b_1 f_2 + f_1 b_2)
\]
for $b_1,b_2 \in B$
and $f_1,f_2 \in D(B)$.
Then there is a canonical associative,
non-degenerate, symmetric $K$-bilinear form
$(-,-) : \T(B) \times \T(B) \to K$
defined by
\[
 \big( (b_1,f_1), (b_2,f_2) \big)
  = f_1(b_2) + f_2(b_1)
\]
for $b_1,b_2 \in B$
and $f_1,f_2 \in D(B)$.

A prominent role in the representation theory of tame
symmetric algebras of polynomial growth is played
by the trivial extensions of the tubular algebras
(in the sense of Ringel \cite{R}), whose representation
theory was described by Nehring and Skowro\'nski in \cite{NeS}.
Moreover, the derived equivalence classification of these algebras 
follows from results established in \cite{HR,Ric2}.
We refer also to the article \cite{PS2} for the invariance
of the trivial extensions of tubular algebras under
stable equivalences.
It follows also from \cite[Example~3.3]{Sk1} that there are exactly four
families of the trivial extensions of tubular algebras of tubular
type $(2,2,2,2)$, and the tetrahedral algebras $\Lambda(\bS,\lambda)$
with $\lambda \in K \setminus \{0,1\}$ form one of these families.
For the purposes of this section, we will describe now
the identification of a tetrahedral algebra $\Lambda(\bS,\lambda)$
with the trivial extension algebra $\T(B(\lambda))$ of an alebra
$B(\lambda)$ of global dimension $2$, being for $K \setminus \{0,1\}$
a tubular algebra of tubular type $(2,2,2,2)$.

For each $\lambda \in K^{*}$,
we denote by $B(\lambda)$ the $K$-algebra
given by the quiver
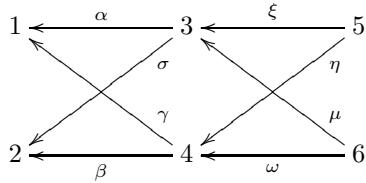
\[
  \xymatrix@C=4.5pc@R=3pc{
    1 &
    3 \ar[l]_{\alpha} \ar[ld]^(.2){\sigma} &
    5 \ar[l]_{\xi} \ar[ld]^(.2){\eta}
    \\
    2 &
    4 \ar[l]^{\beta} \ar[lu]_(.2){\gamma} &
    6 \ar[l]^{\omega} \ar[lu]_(.2){\mu}
  }
\]
and the relations
\begin{align*}
 \eta \gamma &= \xi \alpha,
 &
 \xi \sigma &= \lambda \eta \beta,
 &
 \mu \alpha &= \omega \gamma,
 &
 \omega \beta &= \mu \sigma.
\end{align*}
We note that $B(\lambda)$ is the double
one-point extension algebra of the path algebra
$H = K \Delta$ of the quiver $\Delta$
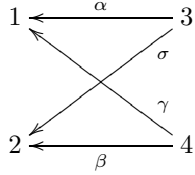
\[
  \xymatrix@C=4.5pc@R=3pc{
    1 &
    3 \ar[l]_{\alpha} \ar[ld]^(.2){\sigma}
    \\
    2 &
    4 \ar[l]^{\beta} \ar[lu]_(.2){\gamma}
  }
\]
of Euclidean type $\widetilde{\bA}_3$
by two indecomposable modules
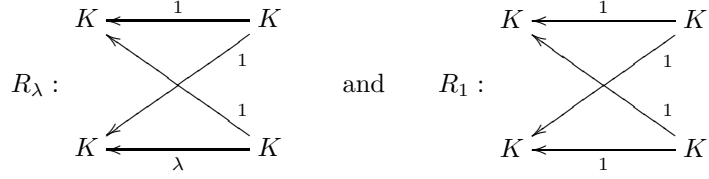
\[
  R_{\lambda} :
\vcenter{
  \xymatrix@C=4.5pc@R=3pc{
    K &
    K \ar[l]_{1} \ar[ld]^(.2){1}
    \\
    K &
    K \ar[l]^{\lambda} \ar[lu]_(.2){1}
  }
}
  \qquad
 \mbox{and}
 \qquad
  R_1 :
\vcenter{
  \xymatrix@C=4.5pc@R=3pc{
    K &
    K \ar[l]_{1} \ar[ld]^(.2){1}
    \\
    K &
    K \ar[l]^{1} \ar[lu]_(.2){1}
  }
}
\]
lying on the mouth of stable tubes of rank $1$ in $\Gamma_H$.
For $\lambda \in K \setminus \{0,1\}$,
the modules $R_{\lambda}$ and $R_1$ are not isomorphic,
and then $B(\lambda)$ is a tubular algebra of type
$(2,2,2,2)$ in the sense of \cite{R},
and consequently it is an algebra of polynomial growth.
On the other hand, $B(1)$ is the tame minimal
non-polynomial growth algebra $(30)$ from \cite{NoS}.
We also mention that all algebras
$B(\lambda)$, $\lambda \in K^{*}$,
are simply connected and of global dimension $2$.

\begin{lemma}
\label{lem:6.2}
For any $\lambda \in K^{*}$,
the algebras $\Lambda(\bS,\lambda)$
and $\T(B(\lambda))$ are isomorphic.
\end{lemma}

\begin{proof}
By general theory (see \cite{Sk2}),
the trivial extension algebra $\T(B(\lambda))$
is isomorphic to the orbit algebra
$\widehat{B(\lambda)} / (\nu_{\widehat{B(\lambda)}})$
of the repetitive category $\widehat{B(\lambda)}$
of $B(\lambda)$ with respect to the infinite
cyclic group $(\nu_{\widehat{B(\lambda)}})$
generated by the Nakayama automorphism
$\nu_{\widehat{B(\lambda)}}$ of $\widehat{B(\lambda)}$.
One checks directly  that $\widehat{B(\lambda)}$
contains the full convex subcategory
$B(\lambda)^{(2)}$
given by the quiver
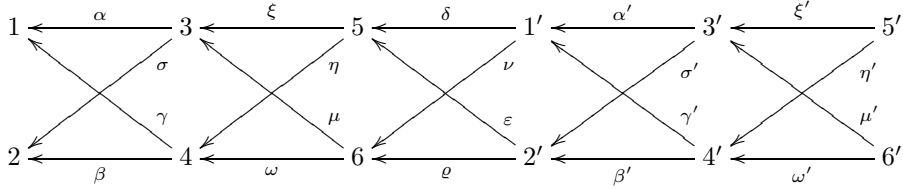
\[
  \xymatrix@C=4.5pc@R=3pc{
    1 &
    3 \ar[l]_{\alpha} \ar[ld]^(.2){\sigma} &
    5 \ar[l]_{\xi} \ar[ld]^(.2){\eta} &
    1' \ar[l]_{\delta} \ar[ld]^(.2){\nu} &
    3' \ar[l]_{\alpha'} \ar[ld]^(.2){\sigma'} &
    5' \ar[l]_{\xi'} \ar[ld]^(.2){\eta'}
    \\
    2 &
    4 \ar[l]^{\beta} \ar[lu]_(.2){\gamma} &
    6 \ar[l]^{\omega} \ar[lu]_(.2){\mu} &
    2' \ar[l]^{\varrho} \ar[lu]_(.2){\varepsilon} &
    4' \ar[l]^{\beta'} \ar[lu]_(.2){\gamma'} &
    6' \ar[l]^{\omega'} \ar[lu]_(.2){\mu'}
  }
\]
and the relations
\begin{align*}
 &&
 \eta \gamma &= \xi \alpha,
 &
 \xi \sigma &= \lambda \eta \beta,
 &
 \mu \alpha &= \omega \gamma,
 &
 \omega \beta &= \mu \sigma,
 \\
 &&
 \nu \mu &= \delta \xi,
 &
 \delta \eta &= \nu \omega,
 &
 \varepsilon \xi &= \varrho \mu,
 &
 \varrho \omega &= \lambda \varepsilon \eta,
 \\
 &&
 \sigma' \varepsilon &= \alpha' \delta,
 &
 \alpha' \nu &= \sigma' \varrho,
 &
 \gamma' \delta &= \lambda \beta' \varepsilon,
 &
 \beta' \varrho &= \gamma' \nu,
 \\
 &&
 \eta' \gamma' &= \xi' \alpha',
 &
 \xi' \sigma' &= \lambda \eta' \beta',
 &
 \mu' \alpha' &= \omega' \gamma',
 &
 \omega' \beta' &= \mu' \sigma',
 \\
 \delta \eta \beta &= 0,
 &
 \varrho \omega \gamma &= 0,
 &
 \sigma' \varepsilon \eta &= 0,
 &
 \beta' \varrho \mu &= 0,
 &
 \eta' \gamma' \nu &= 0,
 &
 \omega' \beta' \varepsilon &= 0,
\end{align*}
where  
$\nu_{\widehat{B(\lambda)}}(i) = i'$ for any vertex $i \in \{1,2,3,4,5,6\}$
and 
$\nu_{\widehat{B(\lambda)}}(\theta) = \theta'$ for any arrow 
$\theta \in \{\alpha,\beta,\gamma,\sigma,\xi,\omega,\eta,\mu\}$.

We conclude that $\T(B(\lambda))$ is isomorphic
to the algebra
$\Lambda(\bS,\lambda) = \Lambda(\bS,\lambda,1,1,1)$.
\end{proof}

We note that the algebra (category) $B(\lambda)^{(2)}$
is isomorphic to the duplicated algebra
\[
   \begin{bmatrix}
     B(\lambda) & 0 \\  D(B(\lambda)) & B(\lambda)
   \end{bmatrix}
   =
  \left\{
    \begin{bmatrix}
      b_1 & 0 \\ f & b_2
    \end{bmatrix}
     \ \Big|\
     b_1, b_2 \in B(\lambda),
     f \in D(B(\lambda))
    \right\}
\]
of $B(\lambda)$.

The next two propositions describe some distinguished
properties of the tetrahedral algebras $\Lambda(\bS,\lambda)$,
$\lambda \in K^{*}$.

\begin{proposition}
\label{prop:6.3}
For any $\lambda \in K \setminus \{ 0,1 \}$
the following statements hold:
\begin{enumerate}[(i)]
 \item
  $\Lambda(\bS,\lambda)$ is an algebra of polynomial growth.
 \item
  $\Lambda(\bS,\lambda)$ is a periodic algebra of period $4$.
 \item
  The simple modules in $\mod \Lambda(\bS,\lambda)$
  are periodic of period $4$.
 \item
  The simple modules in $\mod \Lambda(\bS,\lambda)$
  lie in six pairwise different stable tubes of rank $2$
  of $\Gamma_{\Lambda(\bS(\lambda))}$.
\end{enumerate}
\end{proposition}

\begin{proof}
It follows from Lemma~\ref{lem:6.2} that
$\Lambda(\bS,\lambda)$ is isomorphic to
the trivial extension algebra $\T(B(\lambda))$.
We identify $\Lambda(\bS,\lambda)$ and $\T(B(\lambda))$.
Since $\lambda \in K \setminus \{ 0,1 \}$,
the algebra $B(\lambda)$ is a tubular algebra
of tubular type $(2,2,2,2)$, and  $\T(B(\lambda))$ is the orbit algebra
$\widehat{B(\lambda)} / (\nu_{\widehat{B(\lambda)}})$.
Then, applying the results of \cite[Section~3]{NeS},
we conclude that $\T(B(\lambda))$
is an algebra of polynomial growth and the six
pairwise nonisomorphic indecomposable
projective-injective
$\T(B(\lambda))$-modules
$P_1, P_2, P_3, P_4, P_5, P_6$,
at the vertices
$1,2,3,4,5,6$,
lie in six pairwise different components
$\cC_1,\cC_2,\cC_3,\cC_4,\cC_5,\cC_6$
of $\Gamma_{\T(B(\lambda))}$
such that their stable parts
$\cC^s_1,\cC^s_2,\cC^s_3,\cC^s_4,\cC^s_5,\cC^s_6$
are stable tubes of rank $2$
and do not contain simple modules.
Further, since $\T(B(\lambda))$ is a symmetric
algebra, the six pairwise nonisomorphic
simple $\T(B(\lambda))$-modules
$S_1, S_2, S_3, S_4, S_5, S_6$,
at the vertices
$1,2,3,4,5,6$,
are the socles of the modules
$P_1, P_2, P_3, P_4, P_5, P_6$,
respectively.
Observe also that $P_i/S_i$ belongs
to $\cC_i$,
for any $i\in\{1,2,3,4,5,6\}$.
Then
$S_i = \Omega_{\T(B(\lambda))}(P_i/S_i)$
belongs to a component $\cT_i$ such that
$\cT^s_i = \Omega_{\T(B(\lambda))}(\cC^s_i)$,
for any $i\in\{1,2,3,4,5,6\}$.
Hence, we obtain that
$\cT_1,\cT_2,\cT_3,\cT_4,\cT_5,\cT_6$
are pairwise different stable tubes
of rank $2$ containing the simple modules
$S_1, S_2, S_3, S_4, S_5, S_6$,
respectively.
We also note that
$\tau_{\T(B(\lambda))} = \Omega^2_{\T(B(\lambda))}$,
because $\T(B(\lambda))$ is a symmetric algebra.
Therefore, the simple modules
$S_1, S_2, S_3, S_4, S_5, S_6$
are periodic modules of periodic $4$.

We will prove now that $T(B(\lambda))$
is  periodic as an algebra,  of period $4$.
Consider the cyclic group $H$ of automorphisms
of the algebra
$\T(B(\lambda)) = \Lambda(\bS,\lambda)$
generated by the automorphism $h$ given
by the following cyclic rotations
of the vertices and arrows of the quiver
$Q(\bS,\vec{T})$ from Example~\ref{ex:6.1}
\begin{align*}
 (1 \ 6 \ 3),
 &&
 (4 \ 2 \ 5),
 &&
 (\nu \ \mu \ \alpha),
 &&
 (\beta \ \varepsilon \ \eta),
 &&
 (\gamma \ \varrho \ \xi),
 &&
 (\delta \ \omega \ \sigma).
\end{align*}
Then $H$ is of order $3$ and acts freely on the set
of primitive idempotents of $\Lambda(\bS,\lambda)$
corresponding to the vertices of $Q(\bS,\vec{T})$.
Further, the orbit algebra
$\Lambda(\bS,\lambda)/H = \T(B(\lambda))/H$
is isomorphic to the algebra
$\Lambda'_3(\lambda)$
from \cite[Section~6]{BES4},
given by the quiver
\[
 \xymatrix{
  1 \ar@(dl,ul)[]^{\alpha} \ar@<+.5ex>[r]^{\sigma}
   & 2 \ar@<+.5ex>[l]^{\gamma} \ar@(ur,dr)[]^{\beta}
 }
\]
and the relations
\begin{align*}
\alpha^2 &= \sigma \gamma, &
\gamma \sigma &= \lambda \beta^2, &
\gamma \alpha &= \beta \gamma, &
\alpha \sigma &= \sigma \beta.
\end{align*}
We note that the above relations imply the zero relations
\begin{align*}
 \gamma \alpha^2 &= 0,
 &
 \alpha^2 \sigma &= 0,
 &
 \sigma \beta^2 &= 0,
 &
 \beta^2 \gamma &= 0,
 \\
 \alpha \sigma \beta &= 0,
 &
 \beta \gamma \alpha &= 0,
 &
 \sigma \gamma \sigma &= 0,
 &
 \gamma \sigma \gamma &= 0,
\end{align*}
because $\lambda \in K \setminus \{0,1\}$.
It has been proved in \cite[Proposition~7.1]{BES4}
that $\Lambda'_3(\lambda)$ is a periodic algebra
of period $4$.
Since the order of $H$ is coprime to $4$,
it follows from \cite[Theorem~3.7]{Du1} that
$\Lambda(\bS,\lambda) = \T(B(\lambda))$
is also a periodic algebra of period $4$.
\end{proof}

\begin{proposition}
\label{prop:6.4}
The algebra $\Lambda(\bS,1)$
is a tame algebra
of non-polynomial growth
and there exist three pairwise different components
$\cC_1,\cC_3,\cC_5$ in $\Gamma_{\Lambda(\bS,1)}$
having the following properties:
\begin{enumerate}[(i)]
 \item
  For each $r \in \{1,3,5\}$, $\cC_r$ is isomorphic to the stable
  translation quiver $\bZ \bD_{\infty}$.
 \item
  For each $r \in \{1,3,5\}$, the component $\cC_r$ contains a full
  translation subquiver of the form
  \[
     \xymatrix@R=1.2pc{
        \!\!\!\!\!\!\! \tau_{\Lambda(\bS,1)} S_r \ar[dr] && S_r \\
         & M_r \ar[ur] \ar[dr] \\
        \!\!\!\!\!\!\! \tau_{\Lambda(\bS,1)} S_{r+1} \ar[ur] && S_{r+1}
     }
  \]
  where $S_r$ and $S_{r+1}$ are the simple $\Lambda(\bS,1)$-modules
  at the vertices $r$ and $r+1$.
\end{enumerate}
In particular,  $\Lambda(\bS,1)$ does not have a simple periodic
module, 
and hence $\Lambda(\bS,1)$ is not a periodic algebra.
\end{proposition}

\begin{proof}
We identify
$\Lambda(\bS,1) = \T(B(1)) = \widehat{B(1)} / (\nu_{\widehat{B(1)}})$
using  Lemma~\ref{lem:6.2}.
Consider the Galois covering
$F : \widehat{B(1)} \to \widehat{B(1)} / (\nu_{\widehat{B(1)}}) = \T(B(1))$
and the push-down functor
$F_{\lambda} : \mod \widehat{B(1)} \to \mod \T(B(1))$
induced by $F$.
It follows from \cite[Theorem~3.6]{Ga} that $F_{\lambda}$
preserves the projective modules and almost split sequences.
Recall that $B(1)$ is the pg-critical algebra
$(30)$ from \cite[Theorem~3.2]{NoS},
and hence is a tame algebra of non-polynomial growth,
by \cite[Proposition~3.1]{NoS}.
Then the trivial extension algebra $\T(B(1))$ is of non-polynomial growth,
because $B(1)$ is a quotient algebra of $\T(B(1))$.
Then, applying Proposition~\ref{prop:5.8},
we conclude that $\Lambda(\bS,1) = \T(B(1))$
is a tame algebra of non-polynomial growth.

Consider now the full convex subcategory $B(1)^{(2)}$
of $\widehat{B(1)}$
presented in the proof of Lemma~\ref{lem:6.2}.
For each $i\in\{1,2,3,4,5,6\}$,
we denote
by $S_i^{*}$ the simple $\widehat{B(1)}$-module at the vertex $i$,
and
by $P'_i$ the indecomposable projective $\widehat{B(1)}$-module at the vertex $i'$.
Then, for each $i\in\{1,2,3,4,5,6\}$,
$S_i = F_{\lambda} (S_i^{*})$ is the simple $\T(B(1))$-module and
$P_i = F_{\lambda} (P'_i)$ the indecomposable projective $\T(B(1))$-module
the vertex $i$.
Applying \cite[Theorem~6.1]{NoS},
we conclude that the
Auslander-Reiten quiver $\Gamma_{\widehat{B(1)}}$ of $\widehat{B(1)}$
admits three pairwise different components $\cD_2,\cD_4,\cD_6$
having the following properties:
\begin{itemize}
 \item
  For each $r \in \{1,3,5\}$, the stable part $\cD_{r+1}^s$ of $\cD_{r+1}$
  is isomorphic to the translation quiver $\bZ \bD_{\infty}$.
 \item
  For each $r \in \{1,3,5\}$, the component $\cD_{r+1}$
  does not contain a simple module.
 \item
  For each $r \in \{1,3,5\}$, the component $\cD_{r+1}$ contains a full
  translation subquiver of the form
  
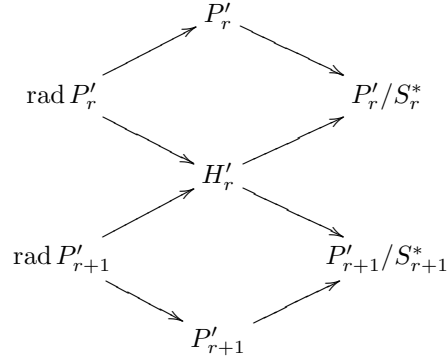
\[
     \xymatrix@R=1.2pc{
         & P'_r \ar[dr] \\
        \rad P'_r \ar[ur] \ar[dr] &&  P'_r/S_r^{*} \\
         & H'_r \ar[ur] \ar[dr] \\
        \rad P'_{r+1} \ar[ur] \ar[dr] &&  P'_{r+1}/S_{r+1}^{*} \\
         & P'_{r+1} \ar[ur]
     }
  \]
  where
  $\rad P'_r/S_r^{*} = H_r = \rad P'_{r+1}/S_{r+1}^{*}$.
\end{itemize}

Then, applying the push-down functor
$F_{\lambda} : \mod \widehat{B(1)} \to \mod \T(B(1))$,
we conclude that the
Auslander-Reiten quiver $\Gamma_{\T(B(1))}$ of $\T(B(1))$
admits three pairwise different components
$\cC_2=F_{\lambda} (\cD_2)$,
$\cC_4=F_{\lambda} (\cD_4)$,
$\cC_6=F_{\lambda} (\cD_6)$,
having the following properties:
\begin{itemize}
 \item
  For each $r \in \{1,3,5\}$, the stable part $\cC_{r+1}^s$ of $\cC_{r+1}$
  is isomorphic to the translation quiver $\bZ \bD_{\infty}$.
 \item
  For each $r \in \{1,3,5\}$, the component $\cC_{r+1}$
  does not contain a simple module.
 \item
  For each $r \in \{1,3,5\}$, the component $\cC_{r+1}$ contains a full
  translation subquiver of the form
  
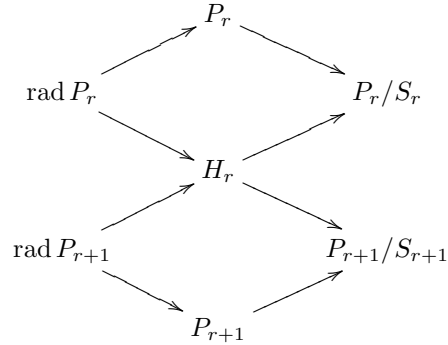
\[
     \xymatrix@R=1.2pc{
         & P_r \ar[dr] \\
        \rad P_r \ar[ur] \ar[dr] &&  P_r/S_r \\
         & H_r \ar[ur] \ar[dr] \\
        \rad P_{r+1} \ar[ur] \ar[dr] &&  P_{r+1}/S_{r+1} \\
         & P_{r+1} \ar[ur]
     }
  \]
  where
  $\rad P_r/S_r = H_r = \rad P_{r+1}/S_{r+1}$.
\end{itemize}
Observe that  $\cC_2,\cC_4,\cC_6$ are all components
of $\Gamma_{\T(B(1))}$ containing projective modules,
and do not contain simple modules.
For each $r \in \{1,3,5\}$, let $\cC_r$ be the component
of $\Gamma_{\T(B(1))}$
such that $\cC_r^s = \Omega_{\T(B(1))}(\cC_{r+1}^s)$.
Then, for each $r \in \{1,3,5\}$, $\cC_r^s$
is isomorphic to the translation quiver $\bZ \bD_{\infty}$,
and $\cC_{r}$
contains a full translation subquiver of the form
  \[
     \xymatrix@R=1.2pc{
        \!\!\!\!\!\!\! \tau_{\T(B(1))} S_r \ar[dr] && S_r \\
         & M_r \ar[ur] \ar[dr] \\
        \!\!\!\!\!\!\! \tau_{\T(B(1))} S_{r+1} \ar[ur] && S_{r+1}
     }
  \]
where $M_r = \Omega_{\T(B(1))}(H_r)$.
Clearly, $\cC_1,\cC_3,\cC_5$ are pairwise different components
of $\Gamma_{\T(B(1))}$,
and different from the components $\cC_2,\cC_4,\cC_6$.
In particular, we conclude that
$\cC_1=\cC_1^s$,
$\cC_3=\cC_3^s$,
$\cC_5=\cC_5^s$.
\end{proof}

We note also the following common property of all algebras
$\Lambda(\bS,\lambda)$, $\lambda \in K^{*}$.

\begin{proposition}
\label{prop:6.5}
Let $\lambda \in K^{*}$.
Then all uniserial modules of length $2$ in
$\mod \Lambda(\bS,\lambda)$ are periodic of period $4$
and form the mouth of six pairwise different
stable tubes of rank $2$ in $\Gamma_{\Lambda(\bS,\lambda)}$.
\end{proposition}

\begin{proof}
For each arrow $\theta$ in the quiver $Q(\bS,\vec{T})$
of $\Lambda(\bS,\lambda)$ we denote by $U_{\theta}$
the uniserial module of length $2$ in
$\mod \Lambda(\bS,\lambda)$ whose top is the simple
module $S_{s(\theta)}$ at $s(\theta)$ and the socle
is the simple module $S_{t(\theta)}$ at $t(\theta)$.
One checks directly that
$\Omega^2_{\Lambda(\bS,\lambda)}(U_{\theta}) = U_{f g f \theta}$.
Moreover, we have $f g f \theta \neq \theta$,
and $f g f^2 g f \theta = \theta$.
Hence, the uniserial modules
$U_{\theta}$ and $U_{f g f \theta}$ are periodic of period $4$
and form the mouth of a stable tube of rank $2$ in
$\Gamma_{\Lambda(\bS,\lambda)}$ (recall that 
$\tau_{\Lambda(\bS,\lambda)}=\Omega^2_{\Lambda(\bS,\lambda)}$).
Observe also that every uniserial module of length $2$ in
$\mod \Lambda(\bS,\lambda)$ is of the form $U_{\theta}$
for some arrow $\theta$ in $Q(\bS,\vec{T})$.
In fact, for each arrows $\theta$ in $Q(\bS,\vec{T})$
the uniserial module $U_{\theta}$ is isomorphic
to the module $\pi g(\pi) \Lambda(\bS,\lambda)$
with $\pi = g^{-2}(\theta)$,
described in Lemma~\ref{lem:5.5}.
\end{proof}

The Gabriel quiver
of a tetrahedral algebra has the following characterization:

\begin{lemma}
\label{lem:6.6}
Let $(Q,f)$ be a triangulation quiver with at least
three vertices.
The following statements are equivalent:
\begin{enumerate}[(i)]
 \item
  $(Q,f)$ is the tetrahedral triangulation quiver.
 \item
  For any arrow $\alpha$ in $Q_1$, we have $n_{\alpha} = 3$.
 \item
  $g^3$ is the identity on $Q_1$.
 \item
  There is an arrow $\beta$ in $Q_1$ such that
  $n_{\beta} = 3$,
  $n_{\bar{\beta}} = 3$,
  $n_{f(\beta)} = 3$,
  $n_{f(\bar{\beta})} = 3$.
\end{enumerate}
\end{lemma}

\begin{proof}
The implications
(i) $\Rightarrow$ (ii) $\Rightarrow$ (iii)
and (ii) $\Rightarrow$ (iv)
are obvious.
We will prove first that (iii) implies (ii).

Assume that $g^3$ is the identity on $Q_1$.
Suppose that $Q_1$ contains a loop
\[
   \xymatrix{ a \ar@(dl,ul)[]^{\alpha}}
   .
\]
Since $Q$ is a $2$-regular connected quiver with at least
three vertices, and since $\alpha$ belongs to a 3-cycle of
either $f$ or $g$,  it contains a subquiver
\[
   \xymatrix{ a \ar@(dl,ul)[]^{\alpha} \ar@/^1.5ex/[r]^{\beta} & b \ar@/^1.5ex/[l]^{\gamma}}
\]
and one of the two cases hold:
\begin{enumerate}[(1)]
 \item
  $f(\alpha) = \alpha$,
  $g(\alpha) = \beta$,
  $g(\beta) = \gamma$,
  $g(\gamma) = \alpha$;
 \item
  $f(\alpha) = \beta$,
  $f(\beta) = \gamma$,
  $f(\gamma) = \alpha$,
  $g(\alpha) = \alpha$.
\end{enumerate}
In  case (1), we obtain $f(\gamma) = \beta$, and hence
$f(\beta) = f^2(\gamma)$, so this is a loop at $b$ since
$f^3(\gamma) = \gamma$.
In  case (2), we obtain $g(\gamma) = \beta$, and hence
$g(\beta) = g^2(\gamma)$ which is again   a loop at $b$ since
$g^3(\gamma) = \gamma$.
Thus, in the both cases,
$Q$ is a quiver with two vertices, a contradiction.
Hence, $Q$ has no loops, and the statement (ii) holds.

It remains to show
that (iv) implies (i).
Assume that $\beta$ is an arrow in $Q_1$ such that
$n_{\beta} = 3$,
$n_{\bar{\beta}} = 3$,
$n_{f(\beta)} = 3$,
$n_{f(\bar{\beta})} = 3$.
We prove  statement (i) in several steps.

We first claim that
$\beta$, 
$\bar{\beta}$,
$f(\beta)$,
$f(\bar{\beta})$
are not loops.
Suppose that $\beta$ is a loop. 
Then $Q$ contains
a subquiver of the form
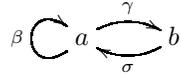
\[
   \xymatrix{ a \ar@(dl,ul)[]^{\beta} \ar@/^1.5ex/[r]^{\gamma} & b \ar@/^1.5ex/[l]^{\sigma}}
\]
with $a \neq b$,
$g(\beta) = \gamma$,
$g(\gamma) = \sigma$,
$g(\sigma) = \beta$.
Then
$f(\beta) = \beta$
and
$f(\sigma) = \gamma$. Since
$f^3(\sigma) = \sigma$
we have
$f(\gamma) = f^2(\sigma)$
and this is a loop at $b$, and consequently
$Q$ has only two vertices, a contradiction.
Similarly, we conclude that
$\bar{\beta}$,
$f(\beta)$,
$f(\bar{\beta})$
are not loops.

We claim now that
$\beta$, 
$\bar{\beta}$,
$f(\beta)$,
$f(\bar{\beta})$
do not belong to $2$-cycles.
Suppose that $\beta$ belongs to a $2$-cycle
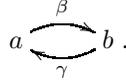
\[
   \xymatrix{ a \ar@/^1.5ex/[r]^{\beta} & b \ar@/^1.5ex/[l]^{\gamma}} .
\]
Then
$\gamma= f(\beta)$
or
$\gamma= g(\beta)$.
Since
$f^3(\beta) = \beta$
and
$g^3(\beta) = \beta$
we infer that 
$f(\gamma) = f^2(\beta)$
or
$g(\gamma) = g^2(\beta)$
 and hence  is a loop.
This is a contradiction 
because such a loop 
is equal to $\bar{\beta}$.
We note also that
$\overbar{f(\beta)} = g(\beta)$
and
$\overbar{f(\bar{\beta})} = g(\bar{\beta})$,
and hence 
$n_{\overbar{f(\beta)}} = n_{\beta} = 3$,
$n_{\overbar{f(\bar{\beta})}} = n_{\bar{\beta}} = 3$.
Then we conclude that
$\bar{\beta}$,
$f(\beta)$,
$f(\bar{\beta})$
do not belong to $2$-cycles.

In the next step, we prove that
$\beta$ ,
$\bar{\beta}$,
$f(\beta)$,
$f(\bar{\beta})$
are not part of double arrows.
Suppose that  $Q$ has  double arrow
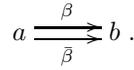
\[
  \xymatrix@C=.8pc{
    a
    \ar@<.5ex>[rr]^{\beta}
    \ar@<-.5ex>[rr]_{\bar{\beta}}
    && b
  } .
\]
Note that
$f(\beta)\neq f(\bar{\beta})$ and therefore they are the arrows starting at $b$, 
and similarly $f^2(\beta)$ and  $f^2(\bar{\beta})$ are the arrows
ending at $a$. 

Now $g(\beta)\neq g(\bar{\beta})$, and these also start at $b$. Since
 $g(\beta) \neq f(\beta)$, we must have $f(\beta)=g(\bar{\beta})$ and $f(\bar{\beta}) = g(\beta)$. 

Since $n_{\beta}=3$, the arrow $g^2(\beta)$ ends at $a$ and therefore
it must be one of $f^2(\beta)$ or $f^2(\bar{\beta})$. 
Similarly $g^2(\bar{\beta})$ ends at $a$. Now $g^2(\beta) \neq g^2(\bar{\beta})$ and therefore
$\{ f^2(\beta), f^2(\bar{\beta})\} = \{ g^2(\beta), g^2(\bar{\beta})\}.
$
If $g^2(\beta) = f^2(\beta)$ then 
$f(g^2(\beta)) = f^3(\beta) = \beta$.
But then 
$$g(g^2(\beta)) = \overline{f(g^2(\beta))} = \bar{\beta}
$$
and $n_{\beta} > 3$, a contradiction.
So we can only have $g^2(\beta) = f^2(\bar{\beta})$. This means that if 
 $\gamma: = f(\bar{\beta}) = g(\beta)$, 
then we have $g(\gamma) = f(\gamma)$ which also is a contradiction.
Similarly one shows that $f(\beta)$ and $f\bar{\beta})$ are not double arrows.

Summing up, we conclude that $Q$ contains
a subquiver of the form
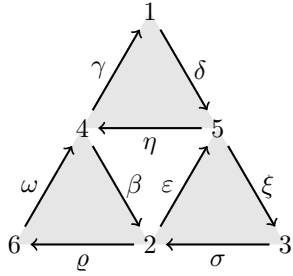
\[
\begin{tikzpicture}
[scale=.9]
\coordinate (1) at (0,1.72);
\coordinate (2) at (0,-1.72);
\coordinate (3) at (2,-1.72);
\coordinate (4) at (-1,0);
\coordinate (5) at (1,0);
\coordinate (6) at (-2,-1.72);

\fill[fill=gray!20]
    (1) -- (4) -- (5) -- cycle;
\fill[fill=gray!20]
    (2) -- (4) -- (6) -- cycle;
\fill[fill=gray!20]
    (2) -- (3) -- (5) -- cycle;

\node (1) at (0,1.72) {$1$};
\node (2) at (0,-1.72) {$2$};
\node (3) at (2,-1.72) {$3$};
\node (4) at (-1,0) {$4$};
\node (5) at (1,0) {$5$};
\node (6) at (-2,-1.72) {$6$};
\draw[->,thick]
(1) edge node [right] {$\delta$} (5)
(2) edge node [left] {$\varepsilon$} (5)
(2) edge node [below] {$\varrho$} (6)
(3) edge node [below] {$\sigma$} (2)
(4) edge node [left] {$\gamma$} (1)
(4) edge node [right] {$\beta$} (2)
(5) edge node [right] {$\xi$} (3)
(5) edge node [below] {$\eta$} (4)
(6) edge node [left] {$\omega$} (4)
;
\end{tikzpicture}
\]
where
$\varepsilon = g(\beta)$,
$\eta = g(\varepsilon)$,
$\beta = g(\eta)$,
and the shaded triangles denote the $f$-orbits
of the arrows $\beta, \varepsilon, \eta$.
Observe that
$\xi = g(\delta)$,
$\gamma = g(\omega)$,
$\varrho = g(\sigma)$.
Moreover,
we have
$\gamma = \bar{\beta}$,
$\varrho = f(\beta)$,
$\delta = f(\bar{\beta})$.
Hence,
by the  imposed assumption,  
there exist arrows
$\alpha, \nu, \mu$ in $Q_1$ with
$t(\alpha) = 1 = s(\nu)$,
$t(\nu) = 6 = s(\mu)$,
$t(\mu) = 3 = s(\alpha)$
such that
$g(\alpha) = \delta$,
$g(\nu) = \omega$,
$g(\mu) = \sigma$.
Obviously, then
$f(\alpha) = \nu$,
$f(\nu) = \mu$,
$f(\mu) = \alpha$.
Therefore, $(Q,f)$ is the required
tetrahedral triangulation quiver.
\end{proof}

An algebra
$\Lambda(\bS,a,b,c,d)$
for $a,b,c,d \in K^*$ with $a b c d = 1$
is said to be a \emph{singular tetrahedral algebra}.
It follows from
Lemma~\ref{lem:6.2}
and
Proposition~\ref{prop:6.4}
that the singular tetrahedral algebras
do not have periodic simple modules,
and hence are not periodic algebras.
We will prove in the next section
that all other weighted surface algebras
are periodic algebras.
We also mention that the tetrahedral algebras
$\Lambda(\bS,a,b,c,d)$
with
$a b c d \neq 1$
are all weighted surface algebras
of polynomial growth.

We would like to stress that, starting from the triangulation quiver
$Q(\bS,\vec{T})$ defined in Example~\ref{ex:6.1} and taking
weight functions with value different from $1$ on some $g$-orbits,
we may create infinitely many weighted surface algebras which are
not isomorphic to the tetrahedral algebras, discussed above.
Similarly, we may create infinitely many new weighted surface algebras
by changing the orientation of triangles in the tetrahedral triangulation
of the sphere.
The following example shows that we obtain new algebras even
if the weight function takes value $1$ on all $g$-orbits.

\begin{example}
\label{ex:6.7}
Let $T$ be the tetrahedral triangulation
of the sphere $\bS$
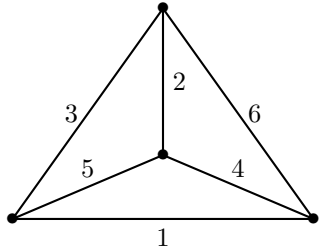
\[
\begin{tikzpicture}
[scale=1]
\node (A) at (-2,0) {$\bullet$};
\node (B) at (2,0) {$\bullet$};
\node (C) at (0,.85) {$\bullet$};
\node (D) at (0,2.8) {$\bullet$};
\coordinate (A) at (-2,0) ;
\coordinate (B) at (2,0) ;
\coordinate (C) at (0,.85) ;
\coordinate (D) at (0,2.8) ;
\draw[thick]
(A) edge node [left] {3} (D)
(D) edge node [right] {6} (B)
(D) edge node [right] {2} (C)
(A) edge node [above] {5} (C)
(C) edge node [above] {4} (B)
(A) edge node [below] {1} (B) ;
\end{tikzpicture}
\]
and $\vec{T}$ the orientation
\[
\mbox{
  (1 4 5), (2 5 3), (2 6 4), (1 6 3)
}
\]
of triangles in $T$, obtained from the coherent orientation
of triangles in $T$ considered in  Example~\ref{ex:6.1}
by changing the orientation of one triangle on the opposite
orientation, and keeping the orientations of all other
triangles unchanged.
Then the associated triangulation quiver
$Q(\bS,\vec{T})$ is 
of the form
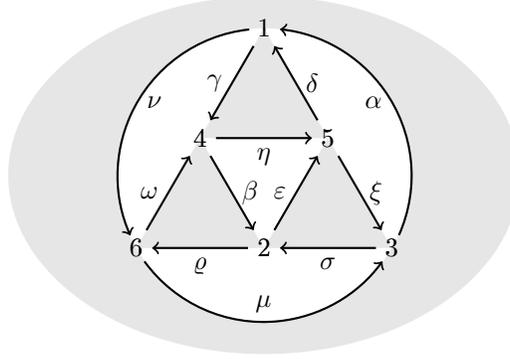
\[
\begin{tikzpicture}
[scale=.85]
\node (1) at (0,1.72) {$1$};
\node (2) at (0,-1.72) {$2$};
\node (3) at (2,-1.72) {$3$};
\node (4) at (-1,0) {$4$};
\node (5) at (1,0) {$5$};
\node (6) at (-2,-1.72) {$6$};
\coordinate (1) at (0,1.72);
\coordinate (2) at (0,-1.72);
\coordinate (3) at (2,-1.72);
\coordinate (4) at (-1,0);
\coordinate (5) at (1,0);
\coordinate (6) at (-2,-1.72);
\fill[fill=gray!20]
      (0,2.22cm) arc [start angle=90, delta angle=-360, x radius=4cm, y radius=2.8cm]
 --  (0,1.72cm) arc [start angle=90, delta angle=360, radius=2.3cm]
     -- cycle;
\fill[fill=gray!20]
    (1) -- (4) -- (5) -- cycle;
\fill[fill=gray!20]
    (2) -- (4) -- (6) -- cycle;
\fill[fill=gray!20]
    (2) -- (3) -- (5) -- cycle;

\node (1) at (0,1.72) {$1$};
\node (2) at (0,-1.72) {$2$};
\node (3) at (2,-1.72) {$3$};
\node (4) at (-1,0) {$4$};
\node (5) at (1,0) {$5$};
\node (6) at (-2,-1.72) {$6$};
\draw[->,thick] (-.23,1.7) arc [start angle=96, delta angle=108, radius=2.3cm] node[midway,right] {$\nu$};
\draw[->,thick] (-1.87,-1.93) arc [start angle=-144, delta angle=108, radius=2.3cm] node[midway,above] {$\mu$};
\draw[->,thick] (2.11,-1.52) arc [start angle=-24, delta angle=108, radius=2.3cm] node[midway,left] {$\alpha$};
\draw[->,thick]
 (5) edge node [right] {$\delta$} (1)
(2) edge node [left] {$\varepsilon$} (5)
(2) edge node [below] {$\varrho$} (6)
(3) edge node [below] {$\sigma$} (2)
 (1) edge node [left] {$\gamma$} (4)
(4) edge node [right] {$\beta$} (2)
(5) edge node [right] {$\xi$} (3)
 (4) edge node [below] {$\eta$} (5)
(6) edge node [left] {$\omega$} (4)
;
\end{tikzpicture}
\]
Then we have only two $g$-orbits of arrows in $Q(\bS,\vec{T})$
\begin{align*}
 \cO(\beta)  &= \{ \beta, \varepsilon = g \beta, \delta = g^2 \beta, \nu = g^3 \beta,
                  \omega = g^4 \beta, \eta = g^5 \beta, \xi = g^6 \beta, \alpha = g^7 \beta,
                  \gamma = g^8 \beta \}, \\
 \cO(\varrho)&= \{ \varrho, \mu = g \varrho, \sigma = g^2 \varrho \}.
\end{align*}
Moreover, let
$m_{\bullet} : \cO(g) \to \bN^*$
be the weight function
taking the value $1$
on each $g$-orbit in $\cO(g)$,
$c_{\bullet} : \cO(g) \to K^*$
a parameter function,
and
$a = c_{\cO(\beta)}$,
$b = c_{\cO(\varrho)}$.
Then the associated algebra
$\Lambda(\bS,\vec{T},m_{\bullet},c_{\bullet})$
is given by the above quiver
$Q(\bS,\vec{T})$ and
the relations
\begin{align*}
  \eta\delta &= a \beta\varepsilon\delta\nu\omega\eta\xi\alpha,
  &
  \delta\gamma &= a \xi\alpha\gamma\beta\varepsilon\delta\nu\omega,
  &
  \gamma\eta &= a \nu\omega\eta\xi\alpha\gamma\beta\varepsilon,
    \\
  \varrho\omega &= a \varepsilon\delta\nu\omega\eta\xi\alpha\gamma,
  &
  \omega\beta &= b  \mu\sigma,
  &
  \beta\varrho &= a \eta\xi\alpha\gamma\beta\varepsilon\delta\nu,
    \\
  \sigma\varepsilon &= a \alpha\gamma\beta\varepsilon\delta\nu\omega\eta,
  &
  \varepsilon\xi &= b \varrho\mu,
  &
  \xi\sigma &= a \delta\nu\omega\eta\xi\alpha\gamma\beta,
    \\
  \alpha\nu &= b \sigma\varrho,
  &
  \nu\mu &= a \gamma\beta\varepsilon\delta\nu\omega\eta\xi,
  &
  \mu\alpha &= a \omega\eta\xi\alpha\gamma\beta\varepsilon\delta,
    \\
  \eta\delta\nu &= 0,
  &
  \delta\gamma\beta &= 0,
  &
  \gamma\eta\xi &= 0,
    \\
  \varrho\omega\eta &= 0,
  &
  \omega\beta\varepsilon &= 0,
  &
  \beta\varrho\mu &= 0,
    \\
  \sigma\varepsilon\delta &= 0,
  &
  \varepsilon\xi\alpha &= 0,
  &
  \xi\sigma\varrho &= 0,
    \\
  \alpha\nu\omega &= 0,
  &
  \nu\mu\sigma &= 0,
  &
  \mu\alpha\gamma &= 0.
\end{align*}
Observe that this algebra
$\Lambda(\bS,\vec{T},m_{\bullet},c_{\bullet})$
is not isomorphic to a tetrahedral algebra.
We will prove in Section~\ref{sec:reptype} that
$\Lambda(\bS,\vec{T},m_{\bullet},c_{\bullet})$
is a tame algebra of non-polynomial growth.
It is also known that  derived equivalence of self-injective
algebras preserves the representation type (see \cite{K,KZ,Ric1}).
Hence it follows from Proposition~\ref{prop:6.3} that
$\Lambda(\bS,\vec{T},m_{\bullet},c_{\bullet})$
is not derived equivalent to a non-singular tetrahedral algebra.
We will show in Section~\ref{sec:periodicity} that
$\Lambda(\bS,\vec{T},m_{\bullet},c_{\bullet})$
is a periodic algebra.
Then, applying 
Proposition~\ref{prop:6.4},
we conclude that
$\Lambda(\bS,\vec{T},m_{\bullet},c_{\bullet})$
is not derived equivalent to a singular tetrahedral algebra,
because  periodicity of algebras is invariant
under derived equivalence
(see \cite{ESk3,Ric2}).
This shows that changing orientation of one
triangle in a directed triangulated surface may lead
to a non-derived equivalent weighted surface algebra.
\end{example}

\section{Periodicity of weighted surface algebras}\label{sec:periodicity}

In this section we will prove that every weighted surface algebra
with at least  three simple modules,
not isomorphic to a tetrahedral algebra, is a periodic algebra
of period $4$.
We note that, by
Propositions \ref{prop:6.3} and \ref{prop:6.4},
a tetrahedral algebra $\Lambda(\bS,\lambda)$, $\lambda \in K^*$,
is a periodic algebra if and only if $\Lambda(\bS,\lambda)$
is nonsingular ($\lambda \neq 1$).
Moreover, for $\lambda \in K \setminus \{ 0,1\}$ the algebra
has period $4$.

Throughout this section, we fix 
$\Lambda = \Lambda(Q,f,m_{\bullet},c_{\bullet})$
for a triangulation quiver $(Q,f)$
with at least  three vertices,
a weight function
$m_{\bullet} : \cO(g) \to \bN^*$
and
a parameter function
$c_{\bullet} : \cO(g) \to K^*$.
Moreover, we assume that 
$\Lambda$ is not a tetrahedral algebra.

We start by describing minimal
projective resolutions of simple modules in $\mod \Lambda$.

\begin{proposition}
\label{prop:7.1}
Let $i$ be a vertex of $Q$ and $\alpha$, $\bar{\alpha}$
the arrows of $Q$ starting at $i$.
Then there is an exact sequence in $\mod \Lambda$
\[
  0 \rightarrow
  S_i \rightarrow
  P_i \xrightarrow{\pi_3}
  P_{t(f(\alpha))} \oplus P_{t(f(\bar{\alpha}))} \xrightarrow{\pi_2}
  P_{t(\alpha)} \oplus P_{t(\bar{\alpha})} \xrightarrow{\pi_1}
  P_i \rightarrow
  S_i \rightarrow
  0,
\]
which give rise to a minimal projective resolution of $S_i$ in $\mod \Lambda$.
In particular, $S_i$ is a periodic module of period $4$.
\end{proposition}

\begin{proof}
We take for $S_i$ the simple quotient of $P_i = e_i \Lambda$,
and then $\Omega_{\Lambda}(S_i)$ can be identified with
$\rad P_i = \alpha \Lambda + \bar{\alpha} \Lambda$.
We define the homomorphism of right $\Lambda$-modules
\[
  \pi_1 : P_{t(\alpha)} \oplus P_{t(\bar{\alpha})} \to P_i
\]
by $\pi_1 (x,y) = \alpha x + \bar{\alpha} y$
for $x \in P_{t(\alpha)}$ and $y \in P_{t(\bar{\alpha})}$.
Clearly, $\pi_1$ induces a projective cover of
$\rad P_i= \Omega_{\Lambda} (S_i)$ and its
kernel is isomorphic to $\Omega_{\Lambda}^2(S_i)$.
We know the dimension of  $\Omega_{\Lambda}^2(S_i)$.
Namely, using the projective cover $\pi_1$ and
Corollary~\ref{cor:5.6}, we obtain the equalities
\begin{align*}
 \dim_K \Omega_{\Lambda}^2(S_i)
  &= \dim_K P_{t(\alpha)} + \dim_K  P_{t(\bar{\alpha})} - (\dim_K P_i - 1)
\\
  &= m_{f(\alpha)} n_{f(\alpha)}
    + m_{g(\alpha)} n_{g(\alpha)}
    + m_{f(\bar{\alpha})} n_{f(\bar{\alpha})}
    + m_{g(\bar{\alpha})} n_{g(\bar{\alpha})}
    - m_{\alpha} n_{\alpha}
\\
  & \quad 
    - m_{\bar{\alpha}} n_{\bar{\alpha}}
    + 1
\\
  &= m_{f(\alpha)} n_{f(\alpha)}
    + m_{f(\bar{\alpha})} n_{f(\bar{\alpha})}
    + 1
    ,
\end{align*}
because
$m_{g({\alpha})}=m_{\alpha}$,
$n_{g({\alpha})}=n_{\alpha}$ ,
$m_{g(\bar{\alpha})}=m_{\bar{\alpha}}$,
$n_{g(\bar{\alpha})}=n_{\bar{\alpha}}$.

Consider the elements in $P_{t(\alpha)} \oplus P_{t(\bar{\alpha})}$
\[
  \varphi = \big( f(\alpha) ,- c_{\bar{\alpha}} A'_{\bar{\alpha}} \big)
\qquad
  \mbox{and}
\qquad
  \psi = \big(- c_{{\alpha}} A'_{{\alpha}},  f(\bar{\alpha}) \big)
  .
\]
Observe that
\begin{align*}
 \pi_1(\varphi)
   &= \alpha f(\alpha) - c_{\bar{\alpha}} \bar{\alpha} A'_{\bar{\alpha}}
    = \alpha f(\alpha) - c_{\bar{\alpha}} A_{\bar{\alpha}}
    = 0,
 \\
 \pi_1(\psi)
   &= - c_{{\alpha}} {\alpha} A'_{{\alpha}} + \bar{\alpha} f(\bar{\alpha})
    = - c_{{\alpha}} A_{{\alpha}} + \bar{\alpha} f(\bar{\alpha})
    = 0,
\end{align*}
and hence $\varphi, \psi$ belong to
$\Ker \pi_1 = \Omega_{\Lambda}^2(S_i)$.
We note that $\varphi$ and $\psi$ are independent modulo the radical,
even in the case when $A'_{\bar{\alpha}}$ or $A'_{{\alpha}}$
is an arrow.
Indeed, if $A'_{\bar{\alpha}}$ (respectively, $A'_{{\alpha}}$)
is an arrow then $A'_{\bar{\alpha}} = g(\bar{\alpha})$
(respectively, $A'_{{\alpha}} = g({\alpha})$),
and is linearly independent from
$f(\bar{\alpha})$ (respectively, $f({\alpha})$).
We find the intersection of
$\varphi \Lambda$
and
$\psi \Lambda$.
Note that
\begin{align*}
 \varphi f^2(\alpha)
   &= \big( f(\alpha) f^2(\alpha) , - c_{\bar{\alpha}} A'_{\bar{\alpha}} f^2(\alpha) \big)
    = \big( f(\alpha) f^2(\alpha) , - c_{\bar{\alpha}} A_{g(\bar{\alpha})} \big)
  ,
 \\
 \psi f^2(\bar{\alpha})
   &= \big(- c_{{\alpha}} A'_{\alpha} f^2(\bar{\alpha}) ,  f(\bar{\alpha}) f^2(\bar{\alpha}) \big)
    = \big(- c_{{\alpha}} A_{g(\alpha)} ,  f(\bar{\alpha}) f^2(\bar{\alpha}) \big)
  ,
\end{align*}
by Lemma~\ref{lem:5.3}~(v).
Moreover, we have
$g(\alpha) = \overbar{f(\alpha)}$,
$g(\bar{\alpha}) = \overbar{f(\bar{\alpha})}$,
$c_{\alpha} = c_{g(\alpha)}$
$c_{\bar{\alpha}} = c_{g(\bar{\alpha})}$.
Hence we conclude that
$\varphi f^2(\alpha) = - \psi f^2(\bar{\alpha})$.
It follows from
Lemmas \ref{lem:5.3} and \ref{lem:5.4}
that
$f(\alpha) f^2(\alpha) f^3(\alpha) = c_{f(\alpha)} B_{f(\alpha)}$
is a non-zero element of the socle of
$P_{t(\alpha)} = P_{s(f(\alpha))}$ and
$f(\bar{\alpha}) f^2(\bar{\alpha}) f^3(\bar{\alpha}) = c_{f(\bar{\alpha})} B_{f(\bar{\alpha})}$
is a non-zero element of the socle of
$P_{t(\bar{\alpha})} = P_{s(f(\bar{\alpha}))}$.
On the other hand, we have
$- c_{{\alpha}} A_{g(\alpha)} g(f^2({\alpha}))
 = f(\alpha) f^2(\alpha) g(f^2(\alpha)) = 0$,
$- c_{\bar{\alpha}} A_{g(\bar{\alpha})} g(f^2(\bar{\alpha}))
  =
   f(\bar{\alpha}) f^2(\bar{\alpha}) g(f^2(\bar{\alpha})) = 0$,
and
$g(f^2(\alpha)) = f^3(\bar{\alpha})$,
$g(f^2(\bar{\alpha})) = f^3(\alpha)$.
Hence, the socle of
$P_{t(\alpha)} \oplus P_{t(\bar{\alpha})}$
is contained in $\varphi \Lambda \cap \psi \Lambda$.
In particular, we have that
$\dim_K(\varphi \Lambda \cap \psi \Lambda) \geq 3$,
because
$\varphi f^2(\alpha) = - \psi f^2(\bar{\alpha})$
is not in the socle of
$P_{t(\alpha)} \oplus P_{t(\bar{\alpha})}$.
We claim that
$\dim_K(\varphi \Lambda \cap \psi \Lambda) = 3$.
Suppose that
$\dim_K(\varphi \Lambda \cap \psi \Lambda) \geq 4$.
Observe that
if $A'_{{\alpha}}$ (respectively, $A'_{\bar{\alpha}}$)
is not an arrow, then
it follows from Lemma~\ref{lem:5.5}~(i)
that
$A'_{{\alpha}} g(f(\bar{\alpha})) = 0$,
(respectively, $A'_{\bar{\alpha}} g(f({\alpha})) = 0$),
and consequently
$\dim_K(\varphi \Lambda \cap \psi \Lambda) = 3$.
Suppose that
$\dim_K(\varphi \Lambda \cap \psi \Lambda) \geq 4$.
Then
$A'_{{\alpha}}$ and $A'_{\bar{\alpha}}$
are arrows,
and hence
$A'_{{\alpha}} = g({\alpha})$
and
$A'_{\bar{\alpha}} = g(\bar{\alpha})$.
Observe that then
$n_{\alpha} = 3$,
$n_{\bar{\alpha}} = 3$,
$f(g(\alpha)) = g(f(\bar{\alpha}))$,
$f(g(\bar{\alpha})) = g(f({\alpha}))$.
Moreover, there exists an element
$a \in K^*$ such that
$\varphi g(f(\alpha)) = a \psi g(f(\bar{\alpha}))$.
Then we obtain the equalities
\begin{align*}
 f(\alpha) g\big(f(\alpha) \big)
   &= -a  c_{{\alpha}} A'_{{\alpha}} g\big(f(\bar{\alpha})\big)
    = -a c_{{\alpha}} g({\alpha}) g\big(f(\bar{\alpha})\big)
    = -a c_{{\alpha}} g({\alpha}) f\big(g({\alpha})\big)
  ,
  \\
 a f(\bar{\alpha}) g\big(f(\bar{\alpha}) \big)
   &= - c_{\bar{\alpha}} A'_{\bar{\alpha}} g\big(f({\alpha})\big)
    = - c_{\bar{\alpha}} g(\bar{\alpha}) g\big(f({\alpha})\big)
    = - c_{\bar{\alpha}} g(\bar{\alpha}) f\big(g(\bar{\alpha})\big)
  .
\end{align*}
In particular, we conclude
$t(f(g(\alpha))) = t(g(f({\alpha})))$
and
$t(f(g(\bar{\alpha}))) = t(g(f(\bar{\alpha})))$,
and so $n_{f(\alpha)} = 3$ and $n_{f(\bar{\alpha})} = 3$.
Then we conclude that
$n_{\alpha} = 3$,
$n_{\bar{\alpha}} = 3$,
$n_{f(\alpha)} = 3$,
$n_{f(\bar{\alpha})} = 3$.
Hence, applying Lemma~\ref{lem:6.6},
we conclude that $(Q,f)$
is the tetrahedral triangulation quiver,
a contradiction.
Therefore, indeed
$\dim_K(\varphi \Lambda \cap \psi \Lambda) = 3$.
Further, we have the equalities
\begin{align*}
 \dim_K \varphi \Lambda
   &= \dim_K f(\alpha) \Lambda + \dim_K \soc (P_{t(\bar{\alpha})})
    = m_{f(\alpha)} n_{f(\alpha)} + 2 ,
  \\
 \dim_K \psi \Lambda
   &= \dim_K f(\bar{\alpha}) \Lambda + \dim_K \soc (P_{t({\alpha})})
    = m_{f(\bar{\alpha})} n_{f(\bar{\alpha})} + 2 .
\end{align*}
Then we conclude that
\begin{align*}
  \dim_K ( \varphi \Lambda + \psi \Lambda )
   &= \dim_K \varphi \Lambda + \dim_K \psi \Lambda
      - \dim_K(\varphi \Lambda \cap \psi \Lambda)
  \\
   &= m_{f(\alpha)} n_{f(\alpha)}
      + m_{f(\bar{\alpha})} n_{f(\bar{\alpha})} + 1 .
\end{align*}
Since $\varphi \Lambda + \psi \Lambda$ is contained
in $\Ker \pi_1 = \Omega_{\Lambda}^2(S_i)$,
comparing the dimensions, we conclude that
$\Omega_{\Lambda}^2(S_i) = \varphi \Lambda + \psi \Lambda$.
Hence we have found generators of
$\Omega_{\Lambda}^2(S_i)$.
In particular, we conclude that a projective cover of
$\Omega_{\Lambda}^2(S_i)$ in $\mod \Lambda$ is
induced by the homomorphism of right
$\Lambda$-modules
\[
  \pi_2 : P_{t(f(\alpha))} \oplus P_{t(f(\bar{\alpha}))}
  \to P_{t(\alpha)} \oplus P_{t(\bar{\alpha})}
\]
given by $\pi_2(u,v) = \varphi u + \psi v$
for $u \in P_{t(f(\alpha))}$ and $v \in P_{t(f(\bar{\alpha}))}$.
We have seen that
$\varphi f^2(\alpha) = - \psi f^2(\bar{\alpha})$.
This shows that the element in
$P_{t(f(\alpha))} \oplus P_{t(f(\bar{\alpha}))}
 = P_{s(f^2(\alpha))} \oplus P_{s(f^2(\bar{\alpha}))}$
\[
  \theta = \big( f^2(\alpha), f^2(\bar{\alpha}) \big)
\]
lies in $\Ker \pi_2 = \Omega_{\Lambda}^3(S_i)$.
We may calculate the dimension of $\Omega_{\Lambda}^3(S_i)$
as follows
\begin{align*}
  \dim_K \Omega_{\Lambda}^3(S_i)
   &= \dim_K P_{s(f^2(\alpha))} + \dim_K P_{s(f^2(\bar{\alpha}))}
      - \dim_K \Omega_{\Lambda}^2(S_i)
  \\
   &= m_{f^2(\alpha)} n_{f^2(\alpha)}
      + m_{g(f({\alpha}))} n_{g(f({\alpha}))}
      + m_{f^2(\bar{\alpha})} n_{f^2(\bar{\alpha})}
      + m_{g(f(\bar{\alpha}))} n_{g(f(\bar{\alpha}))}
 \\ & \,\quad
      - m_{f(\alpha)} n_{f(\alpha)}
      - m_{f(\bar{\alpha})} n_{f(\bar{\alpha})}
      -1
 \\
   &= m_{f^2(\alpha)} n_{f^2(\alpha)}
      + m_{f^2(\bar{\alpha})} n_{f^2(\bar{\alpha})}
      - 1 ,
\end{align*}
because
$m_{g(f({\alpha}))} = m_{f({\alpha})}$,
$n_{g(f({\alpha}))} = n_{f({\alpha})}$,
$m_{g(f(\bar{\alpha}))} = m_{f(\bar{\alpha})}$,
$n_{g(f(\bar{\alpha}))} = n_{f(\bar{\alpha})}$.
Applying Corollary~\ref{cor:5.6}
to the opposite algebra $\Lambda^{\op}$
we conclude that
$\dim_K \Lambda e_i
 = m_{f^2(\alpha)} n_{f^2(\alpha)} +
 m_{f^2(\bar{\alpha})} n_{f^2(\bar{\alpha})}$.
Since $\Lambda$ is a symmetric algebra, we have
$P_i \cong D (\Lambda e_i)$ in $\mod \Lambda$,
and hence
$\dim_K P_i
 = m_{f^2(\alpha)} n_{f^2(\alpha)} + m_{f^2(\bar{\alpha})} n_{f^2(\bar{\alpha})}$.
Hence we obtain that
$\dim_K \Omega_{\Lambda}^3(S_i) = \dim_K P_i - 1$.
Consider now the homomorphism of right $\Lambda$-modules
\[
  \pi_3 : P_i \to P_{t(f(\alpha))} \oplus P_{t(f(\bar{\alpha}))}
\]
given by $\pi_3(z) = \theta z$ for any $z \in P_i$.
Clearly, $\pi_3$ induces a projective cover of
$\Omega_{\Lambda}^3(S_i)$ in $\mod \Lambda$.
Moreover, $\Ker \pi_3 = S_i = \soc(P_i)$,
because
$\dim_K \Omega_{\Lambda}^3(S_i) = \dim_K (P_i/S_i)$.
In particular, we have
$\Omega_{\Lambda}^4(S_i) \cong S_i$ and
$\Omega_{\Lambda}^j(S_i) \ncong S_i$ for any $j \in \{ 1,2,3 \}$.
This finishes the proof.
\end{proof}

We would like to mention that
Proposition~\ref{prop:7.1}
holds also for any non-singular tetrahedral algebra
$\Lambda(\bS,a,b,c,d)$, which can be checked directly.
On the other hand, for a singular tetrahedral algebra
$\Lambda = \Lambda(\bS,a,b,c,d)$,
the proof given above is incorrect because
we have $\dim_K(\varphi \Lambda \cap \psi \Lambda) = 4$
(instead of $3$).
Clearly, it is also impossible by Proposition~\ref{prop:6.4}.

\medskip
The next aim is to construct the first steps of a minimal
projective bimodule resolution of $\Lambda$.
Then we will show that
$\Omega_{\Lambda}^4(\Lambda) \cong \Lambda$
in $\mod \Lambda^e$.
We shall use the notation introduced in Section~\ref{sec:bimodule}.
Recall the first few steps of a minimal projective resolution
of $\Lambda$ in $\mod \Lambda^e$,
\[
  \bP_3 \xrightarrow{S}
  \bP_2 \xrightarrow{R}
  \bP_1 \xrightarrow{d}
  \bP_0 \xrightarrow{d_0}
  \Lambda \to 0
\]
where
\begin{align*}
  \bP_0
     &= \bigoplus_{i \in Q_0} P(i,i)
      = \bigoplus_{i \in Q_0} \Lambda e_i \otimes e_i \Lambda ,
  \\
  \bP_1
     &= \bigoplus_{\alpha \in Q_1} P\big(s(\alpha),t(\alpha)\big)
      = \bigoplus_{\alpha \in Q_1} \Lambda e_{s(\alpha)} \otimes e_{t(\alpha)} \Lambda ,
\end{align*}
the homomorphism $d_0$
is defined by
$d_0 ( e_i \otimes e_i ) = e_i$ for all $i \in Q_0$,
and the homomorphism $d : \bP_1 \to \bP_0$
is defined by
\[
  d \big( e_{s(\alpha)} \otimes e_{t(\alpha)} \big)
    = \alpha \otimes e_{t(\alpha)} - e_{s(\alpha)} \otimes \alpha
\]
for any arrow $\alpha$ in $Q_1$
(see Lemma~\ref{lem:3.3}).
In particular, we have
$\Omega_{\Lambda}^1(\Lambda) = \Ker d_0$
and
$\Omega_{\Lambda}^2(\Lambda) = \Ker d$.
We define now the homomorphism $R : \bP_2 \to \bP_1$.
For each arrow $\alpha$, consider the  element in $K Q$
\[
  \mu_{\alpha}: = \alpha f(\alpha) - c_{\bar{\alpha}} A_{\bar{\alpha}} .
\]
Note that $\mu_{\alpha} = e_{s(\alpha)} \mu_{\alpha} e_{t(f(\alpha))}$.
It follows from Propositions \ref{prop:3.1} and \ref{prop:7.1}
that $\bP_2$ is of the form
\[
  \bP_2
      = \bigoplus_{\alpha \in Q_1} P\big(s(\alpha),t(f(\alpha))\big)
      = \bigoplus_{\alpha \in Q_1} \Lambda e_{s(\alpha)} \otimes e_{t(f(\alpha))} \Lambda .
\]
We define the homomorphism $R : \bP_2 \to \bP_1$
in $\mod \Lambda^e$ by
\[
  R\big( e_{s(\alpha)} \otimes e_{t(f(\alpha))}\big) = \varrho (\mu_{\alpha})
\]
for any arrow $\alpha$ in $Q_1$,
where $\varrho : K Q \to \bP_1$
is the $K$-linear homomorphism defined in Section~\ref{sec:bimodule}.
It follows from Lemma~\ref{lem:3.4} that
$\Im R \subseteq \Ker d$.

\begin{lemma}
\label{lem:7.2}
The homomorphism $R : \bP_2 \to \bP_1$
induces a projective cover
$\Omega_{\Lambda^e}^2(\Lambda)$
in $\mod \Lambda^e$.
In particular, we have $\Omega_{\Lambda^e}^3(\Lambda) = \Ker R$.
\end{lemma}

\begin{proof}
We know that
$\rad \Lambda^e
 = \rad \Lambda^{\op} \otimes \Lambda
   + \Lambda^{\op} \otimes \rad \Lambda$
(see \cite[Corollary~IV.11.4]{SY}).
It follows from the definition that the generators
$\varrho(\mu_{\alpha})$, $\alpha \in Q_1$,
of the image $R$ are elements of $\rad \bP_1$
which are linearly independent in
$\rad \bP_1 / \rad^2 \bP_1$.
Moreover, the form of $\bP_2$ tells us
where the generators of
$\Omega_{\Lambda^e}^2(\Lambda) = \Ker d$
must be.
Then we conclude that
$\varrho(\mu_{\alpha})$, $\alpha \in Q_1$,
form a minimal set of generators of the right
$\Lambda^e$-module
$\Omega_{\Lambda^e}^2(\Lambda)$.
Summing up, we obtain that $R : \bP_2 \to \Omega_{\Lambda^e}^2(\Lambda)$
is a projective cover of
$\Omega_{\Lambda^e}^2(\Lambda)$
in $\mod \Lambda^e$.
\end{proof}

By Propositions \ref{prop:3.1} and \ref{prop:7.1}
we have that $\bP_3$ is of the form
\[
  \bP_3
      = \bigoplus_{i \in Q_0} P(i,i)
      = \bigoplus_{i \in Q_0} \Lambda e_i \otimes e_i \Lambda .
\]
For each vertex $i \in Q_0$, consider the following element of $\bP_2$
\begin{align*}
  \psi_i
   &= \big( e_i \otimes e_{t(f(\alpha))} \big) f^2(\alpha)
     + \big( e_i \otimes e_{t(f(\bar{\alpha}))} \big) f^2(\bar{\alpha})
     - \alpha \big( e_{t(\alpha)} \otimes e_i \big)
     - \bar{\alpha} \big( e_{t(\bar{\alpha})} \otimes e_i \big)
  \\&
    = \big(e_{s(\alpha)} \otimes e_{t(f(\alpha))}\big) f^2(\alpha)
      + \big(e_{s(\bar{\alpha})} \otimes e_{t(f(\bar{\alpha}))}\big) f^2(\bar{\alpha})
      - \alpha \big(e_{s(f(\alpha))} \otimes e_{t(f^2(\alpha))}\big)
  \\& \quad\,
      - \bar{\alpha} \big(e_{s(f(\bar{\alpha}))} \otimes e_{t(f^2(\bar{\alpha}))}\big)
      ,
\end{align*}
where $\alpha$ and $\bar{\alpha}$ are the arrows starting at vertex $i$.
Then we define the homomorphism $S : \bP_3 \to \bP_2$
in $\mod \Lambda^e$ by
\[
  S( e_i \otimes e_i ) = \psi_i
\]
for any vertex $i \in Q_0$.

\begin{lemma}
\label{lem:7.3}
The homomorphism $S : \bP_3 \to \bP_2$
induces a projective cover of
$\Omega_{\Lambda^e}^3(\Lambda)$
in $\mod \Lambda^e$.
In particular, we have
$\Omega_{\Lambda^e}^4(\Lambda) = \Ker S$.
\end{lemma}

\begin{proof}
We will prove first that
$R(\psi_i) = 0$ for any $i \in Q_0$.
Fix a vertex $i \in Q_0$.
Then we have the equalities in $\bP_1$
\begin{align*}
  R(\psi_i)
   &= \varrho(\mu_{\alpha}) f^2(\alpha)
        + \varrho(\mu_{\bar{\alpha}}) f^2(\bar{\alpha})
        - \alpha \varrho \big(\mu_{f(\alpha)}\big)
        - \bar{\alpha} \varrho\big(\mu_{f(\bar{\alpha})}\big)
   \\&
      = \Big(\varrho\big(\alpha f(\alpha)\big)  - c_{\bar{\alpha}}\varrho(A_{\bar{\alpha}}) \Big) f^2(\alpha)
       + \Big(\varrho\big(\bar{\alpha} f(\bar{\alpha})\big)  - c_{\alpha}\varrho(A_{\alpha}) \Big) f^2(\bar{\alpha})
  \\& \quad\,
        - \alpha \Big(\varrho\big(f(\alpha) f^2(\alpha)\big)
                 - c_{\overbar{f(\alpha)}}\varrho\big(A_{\overbar{f(\alpha)}}\big) \Big)
        - \bar{\alpha} \Big(\varrho\big(f(\bar{\alpha}) f^2(\bar{\alpha})\big)
                 - c_{\overbar{f(\bar{\alpha})}}\varrho\big(A_{\overbar{f(\bar{\alpha})}}\big) \Big)
   \\&
      = \varrho\big(\alpha f(\alpha)\big) f^2(\alpha)
        + \varrho\big(\bar{\alpha} f(\bar{\alpha})\big)  f^2(\bar{\alpha})
        - \alpha \varrho\big(f(\alpha) f^2(\alpha)\big)
        - \bar{\alpha}\varrho\big(f(\bar{\alpha}) f^2(\bar{\alpha})\big)
  \\& \quad\,
        - c_{\bar{\alpha}}\varrho(A_{\bar{\alpha}}) f^2(\alpha)
        - c_{\alpha}\varrho(A_{\alpha}) f^2(\bar{\alpha})
        +  c_{\alpha} \alpha \varrho\big(A_{g(\alpha)}\big)
        +  c_{\bar{\alpha}} \bar{\alpha} \varrho\big(A_{g(\bar{\alpha})}\big)
   \\&
      =  e_i \otimes f(\alpha) f^2(\alpha) +\alpha \otimes  f^2(\alpha)
        +  e_i \otimes f(\bar{\alpha}) f^2(\bar{\alpha}) +\bar{\alpha} \otimes  f^2(\bar{\alpha})
  \\& \quad\,
        - \alpha \otimes  f^2(\alpha) - \alpha f(\alpha) \otimes e_i
        - \bar{\alpha} \otimes  f^2(\bar{\alpha})
        - \bar{\alpha} f(\bar{\alpha}) \otimes e_i
  \\& \quad\,
        - c_{\bar{\alpha}}\varrho(A_{\bar{\alpha}}) f^2(\alpha)
        - c_{\alpha}\varrho(A_{\alpha}) f^2(\bar{\alpha})
        +  c_{\alpha} \alpha \varrho\big(A_{g(\alpha)}\big)
        +  c_{\bar{\alpha}} \bar{\alpha} \varrho\big(A_{g(\bar{\alpha})}\big)
   \\&
      = c_{\alpha} \Big( e_i \otimes A_{g(\alpha)} + \alpha \varrho\big(A_{g(\alpha)}\big)
                    - \varrho(A_{\alpha}) f^2(\bar{\alpha}) - A_{\alpha} \otimes e_i \Big)
  \\& \quad\,
         + c_{\bar{\alpha}} \Big( e_i \otimes A_{g(\bar{\alpha})} + \bar{\alpha} \varrho\big(A_{g(\bar{\alpha})}\big)
                    - \varrho(A_{\bar{\alpha}}) f^2(\alpha) - A_{\bar{\alpha}} \otimes e_i \Big)
   \\&
      = 0
      ,
\end{align*}
because
$f^2(\bar{\alpha}) =g^{n_{\alpha} - 1}(\alpha)$
and
$f^2(\alpha) =g^{n_{\bar{\alpha}} - 1}(\bar{\alpha})$.
Hence  $\Im S \subseteq \Ker R$.
Further, it follows from the definition
that the generators $\psi_i$, $i \in Q_0$,
of the image of $S$ are elements of $\rad \bP_2$
which are linearly independent in
$\rad \bP_2 / \rad^2 \bP_2$.
Then we conclude from the form of $\bP_2$ that
these elements form a minimal set of generators of
$\Ker R = \Omega_{\Lambda^e}^3(\Lambda)$.
Hence $S : \bP_3 \to \Omega_{\Lambda^e}^3(\Lambda)$
is a projective cover of $\Omega_{\Lambda^e}^3(\Lambda)$
in $\mod \Lambda^e$.
\end{proof}

\begin{theorem}
\label{th:7.4}
There is an isomorphism
$\Omega_{\Lambda^e}^4(\Lambda) \cong \Lambda$
in $\mod \Lambda^e$.
In particular, $\Lambda$ is a periodic algebra
of period $4$.
\end{theorem}

\begin{proof}
This  is very similar to the proof of \cite[Theorem~5.9]{ESk2}.
For each vertex $i \in Q_0$, we denote by
$\cB_i$ the basis of $e_i \Lambda$ consisting of $e_i$,
all initial subwords of $A_{\alpha}$ and $A_{\bar{\alpha}}$,
and
$\omega_i = c_{\alpha} \cB_{\alpha} = c_{\bar{\alpha}} \cB_{\bar{\alpha}}$
(see Lemma~\ref{lem:5.3} and Corollary~\ref{cor:5.6}).
We note that $\omega_i$ generates the socle of $e_i \Lambda$.
Then $\cB = \bigcup_{i \in Q_0} \cB_i$ is a $K$-linear basis of $\Lambda$.
In the proof of
Proposition~\ref{prop:5.8},
we have defined the symmetrizing $K$-linear form
$\varphi : \Lambda \to K$ which assigns
to the coset $u + I$ of a path $u$ in $Q$ the element in $K$
\[
   \varphi(u+I) = \left\{ \begin{array}{cl}
      c_{\alpha}^{-1} & \mbox{if $u = \cB_{\alpha}$ for an arrow $\alpha \in Q_1$}, \\
      0 & \mbox{otherwise},
   \end{array} \right.
\]
where $I = I(Q,f,m_{\bullet},c_{\bullet})$.
Then, by general theory, we have
the 
symmetrizing form
$(-,-) : \Lambda \times \Lambda \to K$ such that
$(x,y) = \varphi(x y)$ for any $x,y \in \Lambda$.
Observe that, for any elements $x \in \cB_i$ and $y \in \cB$,
we have
\[
  (x,y) = \mbox{ the coefficient of $\omega_i$ in $x y$},
\]
when $x y$ is expressed
as a linear combination of the elements of $e_i \cB$ over $K$.
Consider also the dual basis
$\cB^* = \{ b^* \,|\, b \in \cB \}$ of $\Lambda$
such that
$(b,c^*) = \delta_{b c}$ for $b,c \in \cB$.
Observe that, for $x \in e_i \cB$ and $y \in \cB$, the element
$(x,y)$ can only be non-zero if $y = y e_i$.
In particular, if $b \in e_i \cB e_j$ then
$b^* \in e_j \cB e_i$.

For each vertex $i \in Q_0$, we define the element of $\bP_3$
\[
  \xi_i = \sum_{b \in \cB_i} b \otimes b^* .
\]
We note that $\xi_i$ is independent of the basis
of $\Lambda$
(see \cite[part (2a) on the page 119]{ESk2}).
It follows from \cite[part (2b) on the page 119]{ESk2}
that, for any element
$a \in e_i (\rad \Lambda) e_j \setminus e_i (\rad \Lambda)^2 e_j$,
we have
\[
  a \xi_i = \xi_j a .
\]
Consider now the homomorphism
\[
  \theta : \Lambda \to \bP_3
\]
in $\mod \Lambda^e$
such that
$\theta (e_i) = \xi_i$
for any $i \in Q_0$.
Then
$\theta (1_{\Lambda}) = \sum_{i \in Q_0} \xi_i$,
and consequently we have
\[
  a \Big( \sum_{i \in Q_0} \xi_i \Big)
  = \theta(a)
  = \Big( \sum_{i \in Q_0} \xi_i \Big) a
\]
for any element $a \in \Lambda$.
We claim that $\theta$ is a monomorphism.
It is enough to show that $\theta$ is a monomorphism
of right $\Lambda$-modules.
We know that $\Lambda = \bigoplus_{i \in Q_0} e_i \Lambda$
and each $e_i \Lambda$ has simple socle
generated by $\omega_i$.
For each $i \in Q_0$, we have
\begin{align*}
  \theta(\omega_i)
     &= \Big( \sum_{j \in Q_0} \xi_j \Big) \omega_i
     = \xi_i \omega_i
     = \sum_{b \in \cB_i} ( b \otimes b^* ) \omega_i
     = \sum_{b \in \cB_i} b \otimes b^* \omega_i
     = \omega_i \otimes \omega_i
     \neq 0
    .
\end{align*}
Hence the claim follows.
Our next aim is to show that
$S(\xi_i) = 0$ for any $i \in Q_0$,
or equivalently, that
$\Im \theta \subseteq \Ker S = \Omega_{\Lambda^e}^4(\Lambda)$.
Applying arguments from
\cite[part (3) on the pages 119 and 120]{ESk2},
we obtain that
\[
  \sum_{b \in \cB} b ( a^r \otimes a^s ) b^*
     = \sum_{b \in \cB} b \otimes a^{r+s} b^*
\]
for all integers $r,s \geq 0$ and any element
$a = e_p a e_q$ in $\rad \Lambda$, with $p,q \in Q_0$.
In particular, for each arrow $\alpha$ in $Q_1$, we have
\[
  \sum_{b \in \cB} b \alpha \otimes  b^*
     = \sum_{b \in \cB} b \otimes \alpha b^* ,
\]
and hence
\[
  \sum_{b \in \cB_i} b \alpha \otimes  b^*
     = \sum_{b \in \cB_i} b \otimes \alpha b^*
\]
for any $i \in Q_0$.
We note that every arrow $\beta$ in $Q$
occurs once as a left factor of some $\psi_j$
(with negative sign) and once a right factor
of some $\psi_k$ (with positive sign),
because $\beta = f^2 (\alpha)$ for a unique
arrow $\alpha$.
Then, for any $i \in Q_0$,
the following equalities hold
\begin{align*}
  S(\xi_i)
     &= \sum_{b \in \cB_i} S ( b \otimes b^* )
     = \sum_{b \in \cB_i} \sum_{j \in Q_0} S ( b e_j \otimes e_j b^* )
     = \sum_{b \in \cB_i} \sum_{j \in Q_0} b S ( e_j \otimes e_j ) b^*
  \\&
     = \sum_{b \in \cB_i} \sum_{j \in Q_0} b \psi_j b^*
     = \sum_{\alpha \in Q_1} \Bigg[
             \sum_{b \in \cB_i} - ( b \alpha \otimes  b^* )
            + \sum_{b \in \cB_i} b \otimes \alpha b^*
        \Bigg]
     = 0
     .
\end{align*}
Hence, indeed
$\Im \theta \subseteq \Ker S = \Omega_{\Lambda^e}^4(\Lambda)$,
and we obtain a monomorphism
$\theta : \Lambda \to \Omega_{\Lambda^e}^4(\Lambda)$
in $\mod \Lambda^e$.

Finally, it follows from
Theorem~\ref{th:2.4}
and
Proposition~\ref{prop:3.1}
that
$\Omega_{\Lambda^e}^4(\Lambda) \cong {}_1 \Lambda_{\sigma}$
in $\mod \Lambda^e$
for some $K$-algebra automorphism $\sigma$ of $\Lambda$.
Then
$\dim_K \Lambda = \dim_K \Omega_{\Lambda^e}^4(\Lambda)$,
and consequently $\theta$ is an isomorphism.
Therefore, we have
$\Omega_{\Lambda^e}^4(\Lambda) \cong \Lambda$
in $\mod \Lambda^e$.
Clearly, then $\Lambda$ is a periodic algebra
of period $4$.
\end{proof}

\begin{corollary}
\label{cor:7.5}
Let $(Q,f)$ be a triangulation quiver with at least four vertices,
let $m_{\bullet}$ and $c_{\bullet}$ be weight and parameter
functions of $(Q,f)$, 
and let $\Lambda = \Lambda(Q,f,m_{\bullet},c_{\bullet})$
be the associated weighted triangulation algebra.
Then the Cartan matrix $C_{\Lambda}$ of $\Lambda$
is singular.
\end{corollary}

\begin{proof}
This follows from
Theorems \ref{th:2.5}
and \ref{th:7.4}.
\end{proof}

\section{Socle deformed weighted surface algebras}\label{sec:socldeform}

In this section we introduce socle deformations of weighted surface algebras
of surfaces with boundary, and describe their basic properties.
We will show in the next section that these algebras
are periodic algebras of period $4$.

Let $(Q,f)$ be a triangulation quiver with at least
three vertices.
A vertex $i \in Q_0$ is said to be a \emph{border vertex}
of $(Q,f)$ if there is a loop $\alpha$ at $i$ with $f(\alpha) = \alpha$.
If so, then 
$\bar{\alpha} = g(\alpha)$,
$\alpha = f^2(\alpha) = g^{n_{\bar{\alpha}} - 1} (\bar{\alpha})$,
and
$f^2(\bar{\alpha}) = g^{-1} (\alpha)$.
In particular, we have
$n_{\alpha} = n_{\bar{\alpha}} \geq 3$,
because $|Q_0| \geq 3$.
Hence the loop $\alpha$ is uniquely determined by the vertex $i$,
and we call it a \emph{border loop} of $(Q,f)$.
We also note that
the following equalities hold
(see before Definition
\ref{def:WTalgebra}):
$
  \alpha A_{\bar{\alpha}}
    = B_{\alpha}
    = A_{\alpha} f^2(\bar{\alpha})$
and
$
  \bar{\alpha} A_{g(\bar{\alpha})}
   = B_{\bar{\alpha}}
    = A_{\bar{\alpha}} \alpha
$.
We denote by $\partial(Q,f)$
the set of all border vertices of $(Q,f)$,
and call it the \emph{border} of $(Q,f)$.
Observe that, if $(S,\vec{T})$
is a directed triangulated surface with
$(Q(S,\vec{T}), f) = (Q, f)$,
then the border vertices of $(Q,f)$
correspond bijectively to the boundary edges
of the triangulation $T$ of $S$.
Hence, the border $\partial(Q,f)$ of $(Q,f)$
is non-empty if and only if the boundary
$\partial S$ of $S$ is not empty.
A function
\[
  b_{\bullet} : \partial(Q,f) \to K
\]
is said to be a \emph{border function} of $(Q,f)$.
Assume that $\partial(Q,f)$ is not empty.
Then, for
a weight function
$m_{\bullet} : \cO(g) \to \bN^*$,
a parameter function
$c_{\bullet} : \cO(g) \to K^*$,
and
a border function
$b_{\bullet} : \partial(Q,f) \to K$,
we may consider the bound quiver algebra
\[
  \Lambda(Q,f,m_{\bullet},c_{\bullet},b_{\bullet})
   = K Q / I (Q,f,m_{\bullet},c_{\bullet},b_{\bullet}),
\]
where $I (Q,f,m_{\bullet},c_{\bullet},b_{\bullet})$
is the admissible ideal in the path algebra $KQ$ of $Q$ over $K$
generated by the elements:
\begin{enumerate}[(1)]
 \item
  ${\alpha} f({\alpha}) - c_{\bar{\alpha}} A_{\bar{\alpha}}$,
  for all arrows $\alpha \in Q_1$ which are not border loops,
 \item
  $\alpha^2 - c_{\bar{\alpha}} A_{\bar{\alpha}} - b_{s(\alpha)} B_{\bar{\alpha}}$,
  for all border loops $\alpha \in Q_1$,
 \item
  $\beta f(\beta) g(f(\beta))$,
  for all arrows $\beta \in Q_1$.
\end{enumerate}
Then $\Lambda(Q,f,m_{\bullet},c_{\bullet},b_{\bullet})$ is
said to be a
\emph{socle deformed weighted triangulation algebra}.
We note that if $b_{\bullet}$ is a zero border function
($b_i = 0$ for all $i \in \partial(Q,f)$) then
$\Lambda(Q,f,m_{\bullet},c_{\bullet},b_{\bullet})
  =
    \Lambda(Q,f,m_{\bullet},c_{\bullet})$.
Moreover, if $(Q,f)= Q(S,\vec{T})$
for a directed triangulated surface $(S,\vec{T})$
with non-empty boundary, then
$\Lambda(Q(S,\vec{T}),m_{\bullet},c_{\bullet},b_{\bullet})$
is said to be a
\emph{socle deformed weighted surface algebra}.

\begin{proposition}
\label{prop:8.1}
Let $(Q,f)$ be a triangulation quiver
with at least three vertices and $\partial(Q,f)$
not empty,
$m_{\bullet}$,
$c_{\bullet}$,
$b_{\bullet}$
weight,
parameter,
border functions of $(Q,f)$,
$\bar{\Lambda} = \Lambda(Q,f,m_{\bullet},c_{\bullet},b_{\bullet})$,
and
$\Lambda = \Lambda(Q,f,m_{\bullet},c_{\bullet})$.
Then the following hold:
\begin{enumerate}[(i)]
 \item
  $\bar{\Lambda}$ is a finite-dimensional algebra
  with $\dim_K \bar{\Lambda} = \sum_{\cO \in \cO(g)} m_{\cO} n_{\cO}^2$.
 \item
  $\bar{\Lambda}$ is socle equivalent to $\Lambda$.
 \item
  $\bar{\Lambda}$ degenerates to $\Lambda$.
 \item
  $\bar{\Lambda}$ is a tame algebra.
 \item
  $\bar{\Lambda}$ is a symmetric algebra.
 \item
  The Cartan matrix $C_{\bar{\Lambda}}$ of $\bar{\Lambda}$ is singular.
\end{enumerate}
\end{proposition}

\begin{proof}
We abbreviate $\bar{I} = I(Q,f,m_{\bullet},c_{\bullet},b_{\bullet})$.

(i)
Let $i$ be a vertex of $Q$, and let
$\alpha$, $\bar{\alpha}$ be
the arrows in $Q$
with source $i$.
Then the indecomposable projective right
$\bar{\Lambda}$-module
$P_i = e_i \bar{\Lambda}$
has  basis given by $e_i$,
all initial subwords of $A_{\alpha}$ and $A_{\bar{\alpha}}$,
and
$c_{\alpha} B_{\alpha} = c_{\bar{\alpha}} B_{\bar{\alpha}}$,
and hence
$\dim_K P_i = m_{\alpha} n_{\alpha}  +  m_{\bar{\alpha}} n_{\bar{\alpha}}$.
Then we obtain
\[
  \dim_K \bar{\Lambda} = \sum_{\cO \in \cO(g)} m_{\cO} n_{\cO}^2 .
\]

(ii)
We note that $\soc(\bar{\Lambda})$ and $\soc(\Lambda)$
are generated by the elements
$c_{\alpha} B_{\alpha} = c_{\bar{\alpha}} B_{\bar{\alpha}}$
for all arrows $\alpha$ in $Q_1$, and
$B_{\alpha} = B_{\bar{\alpha}}$,
$c_{\alpha} = c_{\bar{\alpha}} = c_{g(\alpha)}$
for all loops $\alpha$ in $Q_1$ with
$s(\alpha) \in \partial(Q,f)$.
Therefore the algebras
$\bar{\Lambda}/\soc(\bar{\Lambda})$
and
$\Lambda/\soc(\Lambda)$
are isomorphic.
Hence $\bar{\Lambda}$ is socle equivalent
to $\Lambda$.

(iii)
For each $t \in K$,
consider the bound quiver algebra
$\bar{\Lambda}(t) = KQ/\bar{I}^{(t)}$,
where $\bar{I}^{(t)}$ is the admissible ideal
in the path algebra $K Q$ of $Q$ over $K$
generated by the elements:
\begin{enumerate}[(1)]
 \item
  ${\alpha} f({\alpha}) - c_{\bar{\alpha}} A_{\bar{\alpha}}$,
  for all arrows $\alpha \in Q_1$ which are not border loops,
 \item
  $\alpha^2 - c_{\bar{\alpha}} A_{\bar{\alpha}} - t b_{s(\alpha)} B_{\bar{\alpha}}$,
  for all border loops $\alpha \in Q_1$,
 \item
  $\beta f(\beta) g(f(\beta))$,
  for all arrows $\beta \in Q_1$.
\end{enumerate}
Then $\bar{\Lambda}(t)$, $t \in K$,
is an algebraic family in the variety $\alg_{d}(K)$,
with $d = \dim_K \bar{\Lambda}$,
such that
$\bar{\Lambda}(t) \cong \bar{\Lambda}(1) = \bar{\Lambda}$
for all $t \in K^*$ and
$\bar{\Lambda}(0) \cong \Lambda = \Lambda(Q,f,m_{\bullet},c_{\bullet})$.
It follows from Proposition~\ref{prop:3.1}
that $\bar{\Lambda}$ degenerates to $\Lambda$.

(iv)
$\bar{\Lambda}$ is a tame algebra because
$\bar{\Lambda}/\soc(\bar{\Lambda}) \cong \Lambda/\soc(\Lambda)$
and
$\Lambda$ is  tame, by Proposition~\ref{prop:5.8}.
This also follows from
Propositions \ref{prop:2.2} and \ref{prop:5.8}.

(v)
We define a symmetrizing form
$\bar{\varphi} : \bar{\Lambda} \to K$
of
$\bar{\Lambda} = KQ/\bar{I}$
by assigning to the coset $u + \bar{I}$ of a path
$u$ in $Q$ the following element of $K$
\[
   \bar{\varphi}(u+\bar{I}) = \left\{ \begin{array}{cl}
      c_{\alpha}^{-1} & \mbox{if $u = B_{\alpha}$ for an arrow $\alpha \in Q_1$}, \\
      b_i c_{\alpha}^{-1} & \mbox{if $u = \alpha^2$ for some border loop $\alpha \in Q_1$}, \\
      0 & \mbox{otherwise}.
   \end{array} \right.
\]
We note that for a border loop $\alpha$ of $(Q,f)$ we have
$B_{\alpha} = B_{\bar{\alpha}}$
and
$c_{\alpha} = c_{\bar{\alpha}}$.
Moreover, for any arrow $\beta$ in $Q_1$,
we have $\beta f(\beta) g(f(\beta)) = 0$
and
$\beta g(\beta) f(g(\beta)) = 0$
in $\bar{\Lambda}$
(see Lemma~\ref{lem:5.5}).
Hence, if $\alpha$ is a border loop, then
$\alpha^2 \bar{\alpha} = \alpha^2 g(\alpha) = 0$
and
$f^2 (\bar{\alpha}) \alpha^2
 = g^{-1}(\alpha) \alpha^2
 =
g^{-1}(\alpha) g(g^{-1}(\alpha)) f(g(g^{-1}(\alpha))) = 0$.

(vi)
This follows from (ii), (v),
Corollary~\ref{cor:7.5},
and the fact that all
weighted triangulation algebras given by the triangulation quivers
with three vertices and non-empty border have
singular Cartan matrices
(see Examples
\ref{ex:4.3},
\ref{ex:4.4},
\ref{ex:4.5},
and
\ref{ex:5.10}).
\end{proof}

We note that in general a selfinjective algebra
which is socle equivalent to a tame symmetric algebra, need not be
symmetric
(see \cite[Theorems~6.4, 6.7, and Proposition~6.8]{BES4}).

\begin{proposition}
\label{prop:8.2}
Let $(Q,f)$ be a triangulation quiver
with at least three vertices and $\partial(Q,f)$
not empty, and
$m_{\bullet}$,
$c_{\bullet}$,
$b_{\bullet}$
weight,
parameter,
border functions of $(Q,f)$.
Assume that $K$ has characteristic different from $2$.
Then the algebras
$\Lambda(Q,f,m_{\bullet},c_{\bullet},b_{\bullet})$
and
$\Lambda(Q,f,m_{\bullet},c_{\bullet})$
are isomorphic.
\end{proposition}

\begin{proof}
Since $K$ has characteristic different from $2$,
for any vertex $i \in \partial(Q,f)$ there exists
a unique element $a_i \in K$ such that $b_i = 2 a_i$.
Then we have an isomorphism of $K$-algebras
$h :\Lambda(Q,f,m_{\bullet},c_{\bullet}) \to  \Lambda(Q,f,m_{\bullet},c_{\bullet},b_{\bullet})$
such that
\[
   h(\alpha) = \left\{ \begin{array}{cl}
      \alpha & \mbox{for any arrow $\alpha \in Q_1$ which is not a border loop}, \\
      \alpha - a_{s(\alpha)}\alpha^2 & \mbox{for any border loop $\alpha \in Q_1$}.
   \end{array} \right.
\]
We note that, if $\alpha$ is a border loop in $Q_1$, then
$\alpha^2 g(\alpha) = 0$ and $g^{-1}(\alpha) \alpha^2 = 0$.
\end{proof}

\begin{proposition}
\label{prop:8.3}
Let $A$ be a basic, indecomposable.
symmetric algebra with the Grothendieck group $K_0(A)$
of rank at least $3$ which is socle equivalent to
a weighted triangulated algebra
$\Lambda(Q,f,m_{\bullet},c_{\bullet})$.
Then $A$ is isomorphic to an algebra\linebreak
$\Lambda(Q,f,m_{\bullet},c_{\bullet},b_{\bullet})$
for some
border function $b_{\bullet}$ of $(Q,f)$.
\end{proposition}

\begin{proof}
Let
$\Lambda = \Lambda(Q,f,m_{\bullet},c_{\bullet})$,
$I = I(Q,f,m_{\bullet},c_{\bullet})$,
and so
${\Lambda} = K Q / {I}$.
Since $A$ is socle equivalent to $\Lambda$,
there is a $K$-algebra isomorphism
${\varphi} : A/\soc(A) \to \Lambda/\soc(\Lambda)$.
Then $A$ is isomorphic to a bound quiver algebra
$K Q/J$
for an admissible ideal $J$ of $K Q$,
because $A$ is a basic algebra.
Moreover, we may assume that
$\varphi(\alpha) = \alpha$ for any arrow
$\alpha$ in $Q_1$.
Because $A$ is a symmetric algebra, each
indecomposable
projective right $A$-module $e_i A$
has one-dimensional socle generated by an element
$\omega_i \in e_i A e_i$ such that
$\omega_i \rad A = 0$.
We have the following relations in $A$:
\begin{enumerate}[(1)]
 \item
  ${\alpha} f({\alpha}) + \soc(A) = c_{\bar{\alpha}} A_{\bar{\alpha}} + \soc(A)$,
  for all arrows $\alpha \in Q_1$,
 \item
  $\beta f(\beta) g(f(\beta)) \in \soc(A)$,
  for all arrows $\beta \in Q_1$.
\end{enumerate}
Let $\beta$ be an arrow in $Q_1$ and $i = s(\beta)$.
Then $i = t(f^2(\beta)) \neq t(g(f(\beta)))$.
Since $\beta f(\beta) g(f(\beta))
= e_i\beta f(\beta) g(f(\beta))$,
we conclude that $\beta f(\beta) g(f(\beta)) \in \soc(e_i A)$,
and hence
$\beta f(\beta) g(f(\beta)) = \lambda \omega_i$
for an element $\lambda \in K$.
But then $\beta f(\beta) g(f(\beta)) = 0$
because $\omega_i \in e_i A e_i$.

Take now an arrow $\alpha \in Q_1$, and let $i = s(\alpha)$.
We know that $\alpha f(\alpha)$ and $A_{\bar{\alpha}}$
are paths in $Q$ from $i$ to
$t(f(\alpha)) = s(f^2(\alpha)) = g^{n_{\bar{\alpha}} -1} (\bar{\alpha})$.
Hence, we deduce that
${\alpha} f({\alpha}) + \soc(e_i A) = c_{\bar{\alpha}} A_{\bar{\alpha}}
 + \soc(e_i A)$,
and consequently
$\alpha f(\alpha) - c_{\bar{\alpha}} A_{\bar{\alpha}} = b_i \omega_i$
for some element $b_i \in K$.
We also note that if $i \notin \partial(Q,f)$ then $i \neq t(f(\alpha))$,
and then we conclude as above that
${\alpha} f({\alpha}) = c_{\bar{\alpha}} A_{\bar{\alpha}}$.
Clearly, for $i \in \partial(Q,f)$,
we have
$\alpha f(\alpha) = c_{\bar{\alpha}} A_{\bar{\alpha}} + b_i \omega_i$.
Moreover, in this case have
$B_{{\alpha}} = B_{\bar{\alpha}}$
and
$c_{{\alpha}} = c_{\bar{\alpha}}$,
because $\bar{\alpha} = g(\alpha)$,
so we may take $\omega_i = B_{{\alpha}}$.
Hence, we have the
border function $b_{\bullet} :  \partial(Q,f) \to K$
such that $A$ is isomorphic to the algebra
$\Lambda(Q,f,m_{\bullet},c_{\bullet},b_{\bullet})$.
\end{proof}

The above results show that it is worthwile to distinguish the weighted
surface (triangulation) algebras from the socle deformed weighted
surface (triangulation) algebras, occuring only in characteristic $2$.

It may seem at the first sight that the notion of a socle deformed 
weighted surface (triangulation) algebra is a special case
of the notion of a triangulation algebra defined in 
\cite[Definition~5.16 and Proposition~7.4]{L4}.
But it follows from Theorem~\ref{th:4.11},
\cite[Main Theorem]{ESk5},
and \cite[Theorem~7.1 and Proposition~7.4]{L4} that
these two notions actually coincide.

We end this section with an example showing that
there exist socle deformed weighted surface algebras
which are not isomorphic to a weighted surface algebra.

\begin{example}
\label{ex:8.4}
Let $(Q(S,\vec{T}),f)$ be
the triangulation quiver
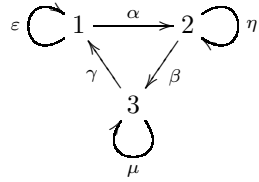
\[
  \xymatrix@C=.8pc@R=1.5pc
  {
     1 \ar@(dl,ul)[]^{\varepsilon} \ar[rr]^{\alpha} &&
     2 \ar@(ur,dr)[]^{\eta} \ar[ld]^{\beta} \\
     & 3 \ar@(dr,dl)[]^{\mu} \ar[lu]^{\gamma}}
\]
with the $f$-orbits
$(\alpha\ \beta\ \gamma)$,
$(\varepsilon)$,
$(\eta)$,
$(\mu)$,
considered in Examples \ref{ex:4.3} and \ref{ex:5.10}.
Then $\cO(g)$ consists of one $g$-orbit
$(\alpha\ \eta\ \beta\ \mu\ \gamma\ \varepsilon)$.
Let
$m_{\bullet} : \cO(g) \to \bN$
be the weight function
with
$m_{\cO(g)} = 1$
and
$c_{\bullet} : \cO(g) \to K^*$
the  parameter function
with
$c_{\cO(g)} = 1$.
Then the associated weighted surface algebra
$\Lambda = \Lambda(Q(S,\vec{T}),f,m_{\bullet},c_{\bullet})$
is given by the above quiver and the relations
\begin{align*}
 \alpha\beta &= \varepsilon\alpha\eta\beta\mu ,
 &
 \varepsilon^2 &= \alpha\eta\beta\mu\gamma,
 &
 \alpha\beta\mu &= 0,
 &
 \varepsilon^2 \alpha &= 0,
\\
 \beta\gamma &= \eta\beta\mu\gamma\varepsilon ,
 &
 \eta^2 &= \beta\mu\gamma\varepsilon\alpha,
 &
 \beta\gamma\varepsilon &= 0,
 &
 \eta^2 \beta &= 0,
\\
 \gamma\alpha &= \mu\gamma\varepsilon\alpha\eta,
 &
 \mu^2 &= \gamma\varepsilon\alpha\eta\beta,
 &
 \gamma\alpha\eta &= 0,
 &
 \mu^2 \gamma &= 0.
\end{align*}
Observe that the border
$\partial(Q(S,\vec{T}),f)$
of $(Q(S,\vec{T}),f)$
is the set $Q_0 = \{1,2,3\}$ of vertices of $Q$,
and $\varepsilon$, $\eta$, $\mu$
are the border loops.
Take now a border function
$b_{\bullet} : \partial(Q(S,\vec{T}),f) \to K$.
Then the associated socle  deformed
weighted surface algebra
$\bar{\Lambda} = \Lambda(Q(S,\vec{T}),f,m_{\bullet},c_{\bullet},b_{\bullet})$
is given by the above quiver and the relations
\begin{align*}
 \alpha\beta &= \varepsilon\alpha\eta\beta\mu ,
 &
 \varepsilon^2 &=  \alpha\eta\beta\mu\gamma + b_1 \alpha\eta\beta\mu\gamma\varepsilon,
 &
 \alpha\beta\mu &= 0,
 &
 \varepsilon^2 \alpha &= 0,
\\
 \beta\gamma &= \eta\beta\mu\gamma\varepsilon ,
 &
 \eta^2 &=\beta\mu\gamma\varepsilon\alpha + b_2 \beta\mu\gamma\varepsilon\alpha\eta ,
 &
 \beta\gamma\varepsilon &= 0,
 &
 \eta^2 \beta &= 0,
\\
 \gamma\alpha &= \mu\gamma\varepsilon\alpha\eta,
 &
 \mu^2 &= \gamma\varepsilon\alpha\eta\beta + b_3 \gamma\varepsilon\alpha\eta\beta\mu ,
 &
 \gamma\alpha\eta &= 0,
 &
 \mu^2 \gamma &= 0.
\end{align*}

Assume that $K$ has characteristic $2$ and $b_{\bullet}$
is non-zero, say $b_1 \neq 0$.
We claim that the algebras
${\Lambda}$ and $\bar{\Lambda}$ are not isomorphic.
Suppose that there is an isomorphism
$h : {\Lambda} \to \bar{\Lambda}$ of $K$-algebras.
Then there exist elements
$r_1, s_1, t_1, u_1, v_1, w_1 \in K^*$
and $r_i, s_i, t_i, u_j, v_j, w_j \in K$,
$i \in \{2,3,4\}$,
$j \in \{2,3\}$,
such that
\begin{align*}
 h(\alpha)
    &= r_1 \alpha
      + r_2 \varepsilon\alpha
      + r_3 \alpha\eta
      + r_4 \varepsilon\alpha\eta ,
 &
 h(\varepsilon)
    &= u_1 \varepsilon
      + u_2 \varepsilon^2
      + u_3 \varepsilon^3 ,
 \\
 h(\beta)
    &= s_1 \beta
      + s_2 \eta\beta
      + s_3 \beta\mu
      + s_4 \eta\beta\mu ,
 &
 h(\eta)
    &= v_1 \eta
      + v_2 \eta^2
      + v_3 \eta^3 ,
 \\
 h(\gamma)
    &= t_1 \gamma
      + t_2 \mu\gamma
      + t_3 \gamma\varepsilon
      + t_4 \mu\gamma\varepsilon ,
 &
 h(\mu)
    &= w_1 \mu
      + w_2 \mu^2
      + w_3 \mu^3 .
\end{align*}
Observe that we have in $\bar{\Lambda}$ the equalities
\begin{align*}
 \varepsilon^3
   &= \varepsilon (\alpha\eta\beta\mu\gamma + b_1 \alpha\eta\beta\mu\gamma \varepsilon)
   = \varepsilon \alpha\eta\beta\mu\gamma ,
 \\
 \varepsilon^3
   &= (\alpha\eta\beta\mu\gamma + b_1 \alpha\eta\beta\mu\gamma \varepsilon) \varepsilon
   = \alpha\eta\beta\mu\gamma \varepsilon .
\end{align*}
Since $K$ has characteristic $2$,
we conclude that the following equalities
hold in $\bar{\Lambda}$
\begin{align*}
 u_1^2  \alpha\eta\beta\mu\gamma
 &
   + u_1^2 b_1 \alpha\eta\beta\mu\gamma \varepsilon
  = u_1^2 \varepsilon^2
  = h(\varepsilon)^2
\\&
  = h(\varepsilon^2)
  = h(\alpha\eta\beta\mu\gamma)
  = h(\alpha) h(\eta) h(\beta) h(\mu) h(\gamma)
\\&
  = r_1 v_1 s_1 w_1 t_1 \alpha\eta\beta\mu\gamma
    + r_2 v_1 s_1 w_1 t_1 \varepsilon \alpha\eta\beta\mu\gamma
    + r_1 v_1 s_1 w_1 t_3 \alpha\eta\beta\mu\gamma \varepsilon
\\&
  = r_1 v_1 s_1 w_1 t_1 \alpha\eta\beta\mu\gamma
    + v_1 s_1 w_1 (r_2 t_1 + r_1  t_3) \alpha\eta\beta\mu\gamma \varepsilon
  ,
\end{align*}
and hence
$u_1^2 = r_1 v_1 s_1 w_1 t_1$
and
$u_1^2 b_1 = v_1 s_1 w_1 (r_2 t_1 + r_1  t_3)$.
In particular, we obtain that $r_2 t_1 + r_1 t_3 \neq  0$,
because
$u_1, b_1, v_1, s_1, w_1 \in K^*$,
On the other hand, we have the following equalities
in $\bar{\Lambda}/(\rad \bar{\Lambda})^4$
\begin{align*}
 0 + (\rad \bar{\Lambda})^4
   &= h (\mu\gamma\varepsilon\alpha\eta) + (\rad \bar{\Lambda})^4
    = h (\gamma\alpha) + (\rad \bar{\Lambda})^4
   \\&
    = h (\gamma) h (\alpha) + (\rad \bar{\Lambda})^4
    = (r_2 t_1 + r_1 t_3) \gamma\varepsilon\alpha + (\rad \bar{\Lambda})^4
   ,
\end{align*}
and hence $r_2 t_1 + r_1 t_3 = 0$, a contradiction.
This proves that  the algebras
${\Lambda}$ and $\bar{\Lambda}$ are not isomorphic.
We note that then, by
Proposition~\ref{prop:8.3},
the algebra $\bar{\Lambda}$ is not isomorphic
to any weighted surface algebra.
\end{example}

It would be interesting to know when,
for $K$ of characteristic $2$, a socle deformed weighted surface algebra
is isomorphic to a weighted surface algebra.

\section{Periodicity of socle deformed weighted surface algebras}\label{sec:deformperidic}

In this section we prove that all socle deformed weighted surface algebras
introduced in the previous section are periodic algebras of period $4$.

Assume that $K$ has characteristic $2$.
Let $(Q,f)$ be a triangulation quiver with at least
three vertices and non-empty border $\partial(Q,f)$.
Moreover, let
$m_{\bullet} : \cO(g) \to \bN^*$
be a weight function,
$c_{\bullet} : \cO(g) \to K^*$
a parameter function
and
$b_{\bullet} : \partial(Q(S,\vec{T}),f) \to K$
 a border function,
which we assume to be non-zero.
Moreover, let\linebreak
$\Lambda = \Lambda(Q(S,\vec{T}), f,m_{\bullet},c_{\bullet})$
be the associated weighted triangulation algebra
and
$\bar{\Lambda} = \Lambda(Q(S,\vec{T}),f,m_{\bullet},c_{\bullet},b_{\bullet})$
the associated socle  deformed
weighted triangulation algebra.
We note that $(Q,f)$ is not the tetrahedral
triangulation quiver, because $\partial(Q,f)$
is not empty.

We have the following analogue of Proposition~\ref{prop:7.1}.

\begin{proposition}
\label{prop:9.1}
Let $i$ be a vertex of $Q$ and $\alpha$, $\bar{\alpha}$
the arrows of $Q$ starting at $i$.
Then there is in $\mod \bar{\Lambda}$ a short exact sequence
\[
  0 \rightarrow
  S_i \rightarrow
  P_i \xrightarrow{\pi_3}
  P_{t(f(\alpha))} \oplus P_{t(f(\bar{\alpha}))} \xrightarrow{\pi_2}
  P_{t(\alpha)} \oplus P_{t(\bar{\alpha})} \xrightarrow{\pi_1}
  P_i \rightarrow
  S_i \rightarrow
  0,
\]
which give rise to a minimal projective resolution of $S_i$ in $\mod \bar{\Lambda}$.
In particular, $S_i$ is a periodic module of period $4$.
\end{proposition}

\begin{proof}
If $i$ is not a border vertex,
the claim follows by arguments as in the proof
of  Proposition~\ref{prop:7.1}.
Therefore, assume that $i \in \partial(Q,f)$.
In this case, we have
$\alpha = f(\alpha)$,
$\bar{\alpha} = g(\alpha)$,
and hence
$c_{\alpha} = c_{\bar{\alpha}}$,
$B_{\alpha} = B_{\bar{\alpha}}$.
We take for $S_i$ the simple quotient of
$P_i = e_i \bar{\Lambda}$,
and hence $\Omega_{\bar{\Lambda}}(S_i)$
is identified with
$\rad P_i = \alpha \bar{\Lambda} + \bar{\alpha} \bar{\Lambda}$.
We define, as in the proof of
Proposition~\ref{prop:7.1},
the homomorphism of right $\bar{\Lambda}$-modules
\[
  \pi_1 : P_{t(\alpha)} \oplus P_{t(\bar{\alpha})} \to P_i
\]
by $\pi_1 (x,y) = \alpha x + \bar{\alpha} y$
for $x \in P_{t(\alpha)}$, $y \in P_{t(\bar{\alpha})}$,
and show that it induces a projective cover of
$\rad P_i = \Omega_{\bar{\Lambda}}(S_i)$
in $\mod \bar{\Lambda}$.
In particular, we obtain that
$\Omega_{\bar{\Lambda}}^2(S_i) = \Ker \pi_1$.

Consider now the following elements in
$P_{t(\alpha)} \oplus P_{t(\bar{\alpha})} = P_i \oplus P_{t(\bar{\alpha})}$
\[
  \bar{\varphi} = \big( f(\alpha) ,- c_{\bar{\alpha}} A'_{\bar{\alpha}} - b_i B'_{\bar{\alpha}} \big)
\qquad
  \mbox{and}
\qquad
  \bar{\psi} = \big(- c_{{\alpha}} A'_{{\alpha}} - b_i A_{\alpha} ,  f(\bar{\alpha}) \big)
  ,
\]
where $B'_{\bar{\alpha}}$ is the subpath of $B_{\bar{\alpha}}$
from $t(\bar{\alpha})$ to $i$ of length
$m_{\bar{\alpha}} n_{\bar{\alpha}} - 1 = m_{\alpha} n_{\alpha} - 1$
such that $\bar{\alpha} B'_{\bar{\alpha}} = B_{\bar{\alpha}}$.
Then we have the equalities
\begin{align*}
 \pi_1(\bar{\varphi})
   &= \alpha^2 - c_{\bar{\alpha}} \bar{\alpha} A'_{\bar{\alpha}} - b_i \bar{\alpha} B'_{\bar{\alpha}}
    = \alpha^2 - c_{\bar{\alpha}} A_{\bar{\alpha}} - b_i B_{\bar{\alpha}}
    = 0,
 \\
 \pi_1(\bar{\psi})
   &= - c_{{\alpha}} {\alpha} A'_{{\alpha}} - b_i {\alpha} A_{{\alpha}} + \bar{\alpha} f(\bar{\alpha})
    = - c_{{\alpha}} A_{{\alpha}} + \bar{\alpha} f(\bar{\alpha})
    = 0,
\end{align*}
because $\alpha A_{\alpha} = \alpha^2 A'_{\alpha} = 0$
due to
$\alpha^2 g(\alpha) = \alpha f(\alpha) g(f(\alpha)) = 0$,
and hence $\bar{\varphi}$, $\bar{\psi}$  belong to
$\Ker \pi_1 = \Omega_{\bar{\Lambda}}^2(S_i)$.
We have also the equalities
\begin{align*}
 \bar{\varphi} f^2(\alpha)
   &= \bar{\varphi} \alpha
    = \big( \alpha^2, - c_{\bar{\alpha}} A'_{\bar{\alpha}} \alpha - b_i B'_{\bar{\alpha}} \alpha \big)
    = \big( \alpha^2, - c_{\bar{\alpha}} B'_{\bar{\alpha}} \big)
  ,
 \\
 \bar{\psi} f^2(\bar{\alpha})
   &= \big(- c_{{\alpha}} A'_{\alpha} f^2(\bar{\alpha}) - b_i A_{{\alpha}} f^2(\bar{\alpha}) ,
           f(\bar{\alpha}) f^2(\bar{\alpha}) \big)
 \\&          
    = \big(- c_{{\alpha}} A_{g(\alpha)} - b_i B_{{\alpha}} ,  f(\bar{\alpha}) f^2(\bar{\alpha}) \big)
  ,
\end{align*}
because
$B'_{\bar{\alpha}} \alpha
 = A_{\overbar{f(\bar{\alpha})}} \alpha
 = c_{f(\bar{\alpha})}^{-1} f(\bar{\alpha}) f^2(\bar{\alpha}) \alpha
 = 0$
due to the equality
$\alpha = g(f^2(\bar{\alpha}))$.
Moreover, we have the equalities
$c_{\alpha} = c_{\bar{\alpha}}$,
$A_{g(\alpha)} = A_{\bar{\alpha}}$,
$B_{\alpha} = B_{\bar{\alpha}}$,
$c_{\bar{\alpha}} = c_{g(\bar{\alpha})}$,
$B'_{\bar{\alpha}} = A_{g(\bar{\alpha})}$,
and
$g(\bar{\alpha}) = \overbar{f(\bar{\alpha})}$.
Hence we conclude that
$\bar{\varphi} f^2(\alpha) = - \bar{\psi} f^2(\bar{\alpha})$.
Recall also that $(Q,f)$
is not the tetrahedral triangulation quiver.
Then, as in the proof of
Proposition~\ref{prop:7.1},
we conclude that
$\dim_K(\bar{\varphi} \bar{\Lambda} \cap \bar{\psi} \bar{\Lambda}) = 3$,
$\bar{\varphi} \bar{\Lambda} + \bar{\psi} \bar{\Lambda} = \Omega_{\bar{\Lambda}}^2(S_i)$,
and the homomorphism of right
$\bar{\Lambda}$-modules
\[
  \pi_2 : P_{t(f(\alpha))} \oplus P_{t(f(\bar{\alpha}))}
  \to P_{t(\alpha)} \oplus P_{t(\bar{\alpha})}
\]
given by $\pi_2(u,v) = \bar{\varphi} u + \bar{\psi} v$
for $u \in P_{t(f(\alpha))}$, and $v \in P_{t(f(\bar{\alpha}))}$
induces a projective cover $\Omega_{\bar{\Lambda}}^2(S_i)$
in $\mod \bar{\Lambda}$.
In particular, we obtain that
$\Omega_{\bar{\Lambda}}^3(S_i) = \Ker \pi_2$.
Further, since
$\bar{\varphi} f^2(\alpha) = - \bar{\psi} f^2(\bar{\alpha})$,
the element
\[
  \bar{\theta} = \big( f^2(\alpha), f^2(\bar{\alpha}) \big)
\]
of $P_{t(f(\alpha))} \oplus P_{t(f(\bar{\alpha}))}
 = P_{s(f^2(\alpha))} \oplus P_{s(f^2(\bar{\alpha}))}$
lies in $\Ker \pi_2 = \Omega_{\bar{\Lambda}}^3(S_i)$.
We may then consider the homomorphism of right $\bar{\Lambda}$-modules
\[
  \pi_3 : P_i \to P_{t(f(\alpha))} \oplus P_{t(f(\bar{\alpha}))}
\]
given by $\pi_3(z) = \bar{\theta} z$ for $z \in P_i$.
Applying arguments as in the final part of the proof of
Proposition~\ref{prop:7.1},
we conclude that $\Ker \pi_3 = S_i$
and that $\pi_3$ induces a projective cover of
$\Omega_{\bar{\Lambda}}^3(S_i)$ in $\mod \bar{\Lambda}$.
Hence
$\Omega_{\bar{\Lambda}}^4(S_i) = \Ker \pi_3 =  S_i$.
Moreover, we have
$\Omega_{\bar{\Lambda}}^j(S_i) \ncong S_i$ for any $j \in \{ 1,2,3 \}$.
This finishes the proof.
\end{proof}

We recall now the notation for the first few steps of a minimal
projective resolution of $\bar{\Lambda}$
in $\mod \bar{\Lambda}^e$
\[
  \bP_3 \xrightarrow{S}
  \bP_2 \xrightarrow{R}
  \bP_1 \xrightarrow{d}
  \bP_0 \xrightarrow{d_0}
  \bar{\Lambda} \to 0
  ,
\]
where
\begin{align*}
  \bP_0
     &= \bigoplus_{i \in Q_0} P(i,i)
      = \bigoplus_{i \in Q_0} \bar{\Lambda} e_i \otimes e_i \bar{\Lambda} ,
  \\
  \bP_1
     &= \bigoplus_{\alpha \in Q_1} P\big(s(\alpha),t(\alpha)\big)
      = \bigoplus_{\alpha \in Q_1} \bar{\Lambda} e_{s(\alpha)} \otimes e_{t(\alpha)} \bar{\Lambda} ,
\end{align*}
the homomorphism $d_0 : \bP_0 \to \bar{\Lambda}$
in $\mod \bar{\Lambda}^e$
is defined by
$d_0 ( e_i \otimes e_i ) = e_i$ for all $i \in Q_0$,
and the homomorphism $d : \bP_1 \to \bP_0$
in $\mod \bar{\Lambda}^e$
is defined by
\[
  d \big( e_{s(\alpha)} \otimes e_{t(\alpha)} \big)
    = \alpha \otimes e_{t(\alpha)} - e_{s(\alpha)} \otimes \alpha
\]
for any arrow $\alpha$ in $Q_1$.
In particular, we have
$\Omega_{\bar{\Lambda}^e}^1(\bar{\Lambda}) = \Ker d_0$
and
$\Omega_{\bar{\Lambda}^e}^2(\bar{\Lambda}) = \Ker d$.
It follows from Propositions \ref{prop:3.1} and \ref{prop:9.1}
that $\bP_2$ is of the form
\[
  \bP_2
      = \bigoplus_{\alpha \in Q_1} P\big(s(\alpha),t(f(\alpha))\big)
      = \bigoplus_{\alpha \in Q_1} \bar{\Lambda} e_{s(\alpha)} \otimes e_{t(f(\alpha))} \bar{\Lambda} .
\]
For each arrow $\alpha$ in $Q_1$, we define the element
$\bar{\mu}_{\alpha} = e_{s(\alpha)} \bar{\mu}_{\alpha} e_{t(f(\alpha))}$
as follows
\begin{align*}
  \bar{\mu}_{\alpha} &= \alpha f(\alpha) - c_{\bar{\alpha}} A_{\bar{\alpha}}
  &&
            \mbox{if $\alpha$ is not a border loop},
  \!\!\!\!\!\!\!\!\!\!\!
  \\
  \bar{\mu}_{\alpha} &= \alpha^2 - c_{\bar{\alpha}} A_{\bar{\alpha}} - b_i B_{\bar{\alpha}}
  \!\!\!\!\!\!\!\!\!\!\!\!\!\!\!\!\!\!\!\!\!\!\!\!
  &&
            \mbox{if $\alpha$ is a border loop}.
\end{align*}
Then we define the homomorphism $R : \bP_2 \to \bP_1$
in $\mod \bar{\Lambda}^e$ by
\[
  R\big( e_{s(\alpha)} \otimes e_{t(f(\alpha))}\big) = \varrho (\bar{\mu}_{\alpha})
\]
for any arrow $\alpha$ in $Q_1$,
where $\varrho : K Q \to \bP_1$
is the $K$-linear homomorphism defined in Section~\ref{sec:bimodule}.
It follows from Lemma~\ref{lem:3.4} that
$\Im R \subseteq \Ker d$.

\begin{lemma}
\label{lem:9.2}
The homomorphism $R : \bP_2 \to \bP_1$
induces a projective cover
$\Omega_{\bar{\Lambda}^e}^2(\bar{\Lambda})$
in $\mod \bar{\Lambda}^e$.
In particular, we have $\Omega_{\bar{\Lambda}^e}^3(\bar{\Lambda}) = \Ker R$.
\end{lemma}

\begin{proof}
This follows by the arguments  in the proof of
Lemma~\ref{lem:7.2}.
\end{proof}

By Propositions \ref{prop:3.1} and \ref{prop:9.1}
the module $\bP_3$ is of the form
\[
  \bP_3
      = \bigoplus_{i \in Q_0} P(i,i)
      = \bigoplus_{i \in Q_0} \bar{\Lambda} e_i \otimes e_i \bar{\Lambda} .
\]
For each vertex $i \in Q_0$, we consider the element in $\bP_2$
\[
  \psi_i =
     \big(e_i \otimes e_{t(f(\alpha))}\big) f^2(\alpha)
     + \big(e_i \otimes e_{t(f(\bar{\alpha}))}\big) f^2(\bar{\alpha})
     - \alpha \big(e_{t(\alpha)} \otimes e_i\big)
     - \bar{\alpha} \big(e_{t(\bar{\alpha})} \otimes e_i\big)
     .
\]
Moreover, for each vertex  $i \in \partial(Q,f)$
and the border loop $\alpha$ at $i$,
we consider the elements in $\bP_2$
\begin{align*}
 \psi_i^{(1)} &= (b_i c_{\alpha}^{-1}) (\alpha \otimes \alpha + e_i \otimes \alpha^2),
\\
  \psi_i^{(2)} &= (b_i c_{\alpha}^{-1})^2 (\alpha \otimes \alpha^2 + e_i \otimes \alpha^3),
\\
  \psi_i^{(3)} &= (b_i c_{\alpha}^{-1})^3 (\alpha \otimes \alpha^3).
\end{align*}
Then, for each vertex  $i \in Q$, we define the element
$\bar{\psi}_i$ in $\bP_2$ as follows
\begin{align*}
  \bar{\psi}_i &= \psi_i
  &&
            \mbox{if $i \notin \partial(Q,f)$},
  \\
  \bar{\psi}_i &= \psi_i + \psi_i^{(1)} + \psi_i^{(2)} + \psi_i^{(3)}
  \!\!\!\!\!\!\!\!\!\!\!\!\!\!\!\!\!\!\!\!\!\!\!\!
  &&
            \mbox{if $i \in \partial(Q,f)$}.
\end{align*}

We define the homomorphism
$S : \bP_3 \to \bP_2$
in $\mod \bar{\Lambda}^e$ by
\[
  S( e_i \otimes e_i ) = \bar{\psi}_i
\]
for any vertex $i \in Q_0$.
Then we have the following analogue of
Lemma~\ref{lem:7.3}.

\begin{proposition}
\label{prop:9.3}
The homomorphism $S : \bP_3 \to \bP_2$
induces a projective cover of
$\Omega_{\bar{\Lambda}^e}^3(\bar{\Lambda})$
in $\mod \bar{\Lambda}^e$.
In particular, we have
$\Omega_{\bar{\Lambda}^e}^4(\bar{\Lambda}) = \Ker S$.
\end{proposition}

\begin{proof}
We will prove in several steps that
$R(\bar{\psi}_i) = 0$ for any vertex $i \in Q_0$.
Fix a vertex $i \in Q_0$.
If $i \notin \partial(Q,f)$
then $R(\bar{\psi}_i) = R(\psi_i) = 0$
by the identities as in the proof of
Lemma~\ref{lem:7.3}.
Assume that $i \in \partial(Q,f)$,
and let $\alpha \in Q_1$ be the border loop at $i$.
Then we have
$\alpha = f(\alpha)$,
$\bar{\alpha} = g(\alpha)$,
$\alpha = f^2(\alpha) = g^{n_{\bar{\alpha}} - 1}(\bar{\alpha})$,
$f^2(\bar{\alpha}) =g^{n_{\alpha} - 1}(\alpha)$,
$c_{\alpha} = c_{\bar{\alpha}}$,
$B_{\alpha} = B_{\bar{\alpha}}$.
We abbreviate
$c = c_{\alpha} = c_{\bar{\alpha}}$
and $b = b_i$.
We have in $\bP_1$ the following equalities describing
$R(\psi_i)$
\begin{align*}
  R(\psi_i)
   &= \varrho(\bar{\mu}_{\alpha}) f^2(\alpha)
        + \varrho(\bar{\mu}_{\bar{\alpha}}) f^2(\bar{\alpha})
        - \alpha \varrho \big(\bar{\mu}_{f(\alpha)}\big)
        - \bar{\alpha} \varrho\big(\bar{\mu}_{f(\bar{\alpha})}\big)
   \\&
     = \varrho(\bar{\mu}_{\alpha}) \alpha
        + \varrho(\bar{\mu}_{\bar{\alpha}}) f^2(\bar{\alpha})
        - \alpha \varrho ({\mu}_{\alpha})
        - \bar{\alpha} \varrho\big(\bar{\mu}_{f(\bar{\alpha})}\big)
   \\&
      = \Big(\varrho(\alpha^2) - c \varrho(A_{\bar{\alpha}}) - b \varrho(B_{\bar{\alpha}}) \Big) \alpha
       + \Big(\varrho\big(\bar{\alpha} f(\bar{\alpha})\big) - c \varrho(A_{\alpha}) \Big) f^2(\bar{\alpha})
  \\& \quad\,
        - \alpha  \Big(\varrho(\alpha^2) - c \varrho(A_{\bar{\alpha}}) - b \varrho(B_{\bar{\alpha}}) \Big)
        - \bar{\alpha} \Big(\varrho\big(f(\bar{\alpha}) f^2(\bar{\alpha})\big) - c  \varrho\big(A_{g(\bar{\alpha})}\big) \Big)
   \\&
      = e_i \otimes \alpha^2 + \alpha \otimes \alpha
        +  e_i \otimes f(\bar{\alpha}) f^2(\bar{\alpha}) +\bar{\alpha} \otimes  f^2(\bar{\alpha})
  \\& \quad\,
        - \alpha \otimes \alpha - \alpha^2 \otimes e_i
        - \bar{\alpha} \otimes  f^2(\bar{\alpha})
        - \bar{\alpha} f(\bar{\alpha}) \otimes e_i
  \\& \quad\,
        - c \varrho(A_{\bar{\alpha}}) \alpha
        - b \varrho(B_{\bar{\alpha}}) \alpha
        - c \varrho(A_{\alpha}) f^2(\bar{\alpha})
  \\& \quad\,
        +  c \alpha \varrho\big(A_{\bar{\alpha}}\big)
        +  b \alpha \varrho\big(B_{\bar{\alpha}}\big)
        +  c \bar{\alpha} \varrho\big(A_{g(\bar{\alpha})}\big)
   \\&
      = e_i \otimes c A_{\bar{\alpha}} + e_i \otimes b B_{\bar{\alpha}}
        + e_i \otimes c A_{g(\bar{\alpha})}
        - c A_{\bar{\alpha}} \otimes e_i - b B_{\bar{\alpha}} \otimes e_i
        - c A_{\alpha} \otimes e_i
  \\& \quad\,
        - c \varrho(A_{\bar{\alpha}}) \alpha
        - b \varrho(B_{\bar{\alpha}}) \alpha
        - c \varrho(A_{\alpha}) f^2(\bar{\alpha})
  \\& \quad\,
        +  c \alpha \varrho\big(A_{\bar{\alpha}}\big)
        +  b \alpha \varrho\big(B_{\bar{\alpha}}\big)
        +  c \bar{\alpha} \varrho\big(A_{g(\bar{\alpha})}\big)
   \\&
      = c \Big( e_i \otimes A_{g(\alpha)} + \alpha \varrho\big(A_{g(\alpha)}\big)
                    - \varrho(A_{\alpha}) f^2(\bar{\alpha}) - A_{\alpha} \otimes e_i \Big)
  \\& \quad\,
      + c \Big( e_i \otimes A_{g(\bar{\alpha})} + \bar{\alpha} \varrho\big(A_{g(\bar{\alpha})}\big)
                - \varrho(A_{\bar{\alpha}}) \alpha - A_{\bar{\alpha}} \otimes e_i \Big)
  \\& \quad\,
      + b \Big( e_i \otimes B_{\bar{\alpha}} - B_{\bar{\alpha}} \otimes e_i
             + \alpha \varrho(B_{\bar{\alpha}}) - \varrho(B_{\bar{\alpha}}) \alpha \Big)
   \\&
      =
        b \Big( e_i \otimes B_{\bar{\alpha}} + B_{\bar{\alpha}} \otimes e_i
             + \alpha \varrho(B_{\bar{\alpha}}) + \varrho(B_{\bar{\alpha}}) \alpha \Big)
      ,
\end{align*}
since $K$ has characteristic $2$.
We note that, if $b = b_i = 0$,
then
then $\bar{\psi}_i = \psi_i$
and $R(\bar{\psi}_i) = R(\psi_i) = 0$.
Hence we may assume that $b \neq 0$.

In order to calculate
$R(\psi_i^{(1)})$,
$R(\psi_i^{(2)})$,
$R(\psi_i^{(3)})$,
we use the following identities in $\bP_1$
\begin{enumerate}[(1)]
 \item
  $\alpha \varrho(A_{\bar{\alpha}}) \alpha = B_{\bar{\alpha}} \otimes e_i  + \alpha \varrho(B_{\bar{\alpha}})$,
 \item
  $\varrho(B_{\bar{\alpha}}) \alpha + \varrho(A_{\bar{\alpha}}) \alpha^2 = A_{\bar{\alpha}} \otimes \alpha$,
\end{enumerate}
which follow from the equalities
$\bar{\alpha} = g(\alpha)$
and
$g^{n_{\bar{\alpha}} - 1}(\bar{\alpha}) = g^{n_{\alpha}}(\alpha) = \alpha$.

We have the following equalities in $\bP_1$ describing $R(\psi_i^{(1)})$
\begin{align*}
  R\big(\psi_i^{(1)}\big)
   &= b c^{-1} \Big( R(\alpha \otimes \alpha) + R(e_i \otimes \alpha^2) \Big)
  \\&
    = b c^{-1} \Big( \alpha R(e_i \otimes e_i) \alpha + R(e_i \otimes e_i) \alpha^2 \Big)
    = b c^{-1} \Big( \alpha \varrho(\mu_{\alpha}) \alpha + \varrho(\mu_{\alpha}) \alpha^2 \Big)
  \\&
    = b c^{-1} \Big( \alpha \varrho\big(\alpha^2 + c A_{\bar{\alpha}} + b B_{\bar{\alpha}}\big) \alpha
             + \varrho\big(\alpha^2 + c A_{\bar{\alpha}} + b B_{\bar{\alpha}}\big) \alpha^2 \Big)
  \\&
    = b c^{-1} \Big( \alpha \otimes \alpha^2 + \alpha^2 \otimes \alpha
              + c \alpha \varrho(A_{\bar{\alpha}}) \alpha
              + b \alpha \varrho(B_{\bar{\alpha}}) \alpha
  \\& \qquad\qquad
              + e_i \otimes \alpha^3 + \alpha \otimes \alpha^2
              + c \varrho(A_{\bar{\alpha}}) \alpha^2
              + b \varrho(B_{\bar{\alpha}}) \alpha^2 \Big)
  \\&
    = b \Big(
              c^{-1} (\alpha^2 \otimes \alpha)
              + B_{\bar{\alpha}} \otimes e_i
              + \alpha \varrho(B_{\bar{\alpha}})
              + b c^{-1} \alpha \varrho(B_{\bar{\alpha}}) \alpha
  \\&
         \qquad\quad
              + e_i \otimes B_{\bar{\alpha}}
              + \varrho(A_{\bar{\alpha}}) \alpha^2
              + b c^{-1} \varrho(B_{\bar{\alpha}}) \alpha^2
          \Big)
      .
\end{align*}
Then we obtain the equalities
\begin{align*}
  R(\psi_i) +  R\big(\psi_i^{(1)}\big)
   &
      =
              b \varrho(B_{\bar{\alpha}}) \alpha
              + b c^{-1} (\alpha^2 \otimes \alpha)
              + b^2 c^{-1} \alpha \varrho(B_{\bar{\alpha}}) \alpha
              + b \varrho(A_{\bar{\alpha}}) \alpha^2
  \\& \quad\,
              + b^2 c^{-1} \varrho(B_{\bar{\alpha}}) \alpha^2
  \\& =
              b  \Big( \varrho(B_{\bar{\alpha}}) \alpha
                      + \varrho(A_{\bar{\alpha}}) \alpha^2 \Big)
              + b c^{-1} \Big( c A_{\bar{\alpha}} \otimes \alpha
                             + b B_{\bar{\alpha}} \otimes \alpha \Big)
  \\& \quad\,
              + b^2 c^{-1} \Big( \alpha \varrho(B_{\bar{\alpha}}) \alpha
                            + \varrho(B_{\bar{\alpha}}) \alpha^2 \Big)
  \\& =
              b A_{\bar{\alpha}} \otimes \alpha
              + b A_{\bar{\alpha}} \otimes \alpha
              + b^2 c^{-1} \Big( B_{\bar{\alpha}} \otimes \alpha
                            + \alpha \varrho(B_{{\alpha}}) \alpha
                            + \varrho(B_{\bar{\alpha}}) \alpha^2 \Big)
  \\& =
              b^2 c^{-1} \Big( B_{\bar{\alpha}} \otimes \alpha
                            + \alpha \varrho(B_{{\alpha}}) \alpha
                            + \varrho(B_{\bar{\alpha}}) \alpha^2 \Big)
  \\& =
              b^2 c^{-1} \Big( \alpha \varrho(A_{\bar{\alpha}}) \alpha^2
                            + A_{\bar{\alpha}} \otimes \alpha^2
                            + \varrho(A_{\bar{\alpha}}) \alpha^3 \Big)
  \\& =
              b^2 c^{-1} \Big( \alpha \varrho(A_{\bar{\alpha}}) \alpha^2
                            + A_{\bar{\alpha}} \otimes \alpha^2
                            + A''_{\bar{\alpha}} \alpha^3 \Big)
      ,
\end{align*}
where $A''_{\bar{\alpha}}$ is the subpath of  $A_{\bar{\alpha}}$
such that $A''_{\bar{\alpha}} g^{n_{\bar{\alpha}} - 2}(\bar{\alpha}) = A_{\bar{\alpha}}$.

We have the following expressions of  $R(\psi_i^{(2)})$
\begin{align*}
  R\big(\psi_i^{(2)}\big)
   &= \big(b c^{-1}\big)^2 \Big( R(e_i \otimes \alpha^3) + R(\alpha \otimes \alpha^2) \Big)
  \\&
    = \big(b c^{-1}\big)^2 \Big( \varrho(\mu_{\alpha}) \alpha^3 + \alpha \varrho(\mu_{\alpha}) \alpha^2 \Big)
  \\&
    = \big(b c^{-1}\big)^2
         \Big(
               \big(\varrho(\alpha^2) + c \varrho(A_{\bar{\alpha}}) + b \varrho(B_{\bar{\alpha}}) \big) \alpha^3
               + \alpha \big(\varrho(\alpha^2) + c \varrho(A_{\bar{\alpha}}) + b \varrho(B_{\bar{\alpha}}) \big) \alpha^2
               \Big)
  \\&
    = \big(b c^{-1}\big)^2
        \Big( \alpha \otimes \alpha^3
              + c \varrho(A_{\bar{\alpha}}) \alpha^3
              + b \varrho(B_{\bar{\alpha}}) \alpha^3
              + \alpha \otimes \alpha^3
              + \alpha^2 \otimes \alpha^2
  \\& \qquad\qquad\quad
              + c \alpha \varrho(A_{\bar{\alpha}}) \alpha^2
              + b \alpha \varrho(B_{\bar{\alpha}}) \alpha^2
               \Big)
  \\&
    = \big(b c^{-1}\big)^2
        \Big( \alpha^2 \otimes \alpha^2
              + c \varrho(A_{\bar{\alpha}}) \alpha^3
              + c \alpha \varrho(A_{\bar{\alpha}}) \alpha^2
              + b \varrho(B_{\bar{\alpha}}) \alpha^3
              + b \alpha \varrho(B_{\bar{\alpha}}) \alpha^2
               \Big)
      .
\end{align*}
Moreover, we have
$\alpha^2 \otimes \alpha^2 = c A_{\bar{\alpha}} \otimes \alpha^2 + b B_{\bar{\alpha}} \otimes \alpha^2$,
and $A''_{\bar{\alpha}} \otimes \alpha^3 = \varrho(A_{\bar{\alpha}}) \alpha^3$.
Then we obtain the equalities
\begin{align*}
  R(\psi_i) + R\big(\psi_i^{(1)}\big) + R\big(\psi_i^{(2)}\big)
   & =
              b^3 c^{-2} \Big(
                   \varrho(B_{\bar{\alpha}}) \alpha^3
                   + \alpha \varrho(B_{\bar{\alpha}}) \alpha^2
                   + B_{\bar{\alpha}} \otimes \alpha^2
                 \Big)
  \\& =
              b^3 c^{-2} \Big(
                   A_{\bar{\alpha}} \otimes \alpha^3
                   + \alpha \varrho(A_{\bar{\alpha}}) \alpha^3
                   + \alpha A_{\bar{\alpha}} \otimes \alpha^2
                   + B_{\bar{\alpha}} \otimes \alpha^2
                 \Big)
  \\& =
              b^3 c^{-2} \Big(
                   A_{\bar{\alpha}} \otimes \alpha^3
                   + A_{{\alpha}} \otimes \alpha^3
                   + B_{{\alpha}} \otimes \alpha^2
                   + B_{\bar{\alpha}} \otimes \alpha^2
                 \Big)
  \\& =
              b^3 c^{-2} \Big(
                   A_{\bar{\alpha}} \otimes \alpha^3
                   + A_{{\alpha}} \otimes \alpha^3
                 \Big)
      ,
\end{align*}
because $B_{\alpha} = B_{\bar{\alpha}}$.
We have also the following equalities
\begin{align*}
  R\big(\psi_i^{(3)}\big)
   &= \big(b c^{-1}\big)^3 R(\alpha \otimes \alpha^3)
    = \big(b c^{-1}\big)^3 \alpha \varrho(\mu_{\alpha}) \alpha^3
  \\&
    = \big(b c^{-1}\big)^3
         \Big(
               \alpha
               \big(\varrho(\alpha^2) + c \varrho(A_{\bar{\alpha}}) + b \varrho(B_{\bar{\alpha}}) \big)
               \alpha^3
               \Big)
  \\&
    = \big(b c^{-1}\big)^3
        \Big( \alpha^2 \otimes \alpha^3
              + c A_{{\alpha}} \otimes \alpha^3
              + b B_{\bar{\alpha}} \otimes \alpha^3
               \Big)
  \\&
    = \big(b c^{-1}\big)^3
        \Big(
              c A_{\bar{\alpha}} \otimes \alpha^3
              + b B_{\bar{\alpha}} \otimes \alpha^3
              + c A_{{\alpha}} \otimes \alpha^3
              + b B_{{\alpha}} \otimes \alpha^3
               \Big)
  \\&
    = b^3 c^{-2}
        \Big(
              A_{\bar{\alpha}} \otimes \alpha^3
              + A_{{\alpha}} \otimes \alpha^3
               \Big)
      .
\end{align*}
Summing up, we obtain the required vanishing equality
\begin{align*}
  R(\bar{\psi}_i)
  = R(\psi_i) + R\big(\psi_i^{(1)}\big) + R\big(\psi_i^{(2)}\big)  + R\big(\psi_i^{(3)}\big)
  = 0
     .
\end{align*}
Therefore we have
$\Im S \subseteq \Ker R$.
Further, it follows from the definition that the generators
$\bar{\psi}_i$, $i \in Q_0$,
of the image of $S$ are elements of $\rad \bP_2$
which are linearly independent in $\rad \bP_2 / \rad^2 \bP_2$.
Then we conclude from the form of $\bP_2$
that these elements form a minimal set of generators
of $\Ker R = \Omega_{\bar{\Lambda}^e}^3(\bar{\Lambda})$.
Hence $S : \bP_3 \to \Omega_{\bar{\Lambda}^e}^3(\bar{\Lambda})$
is a projective cover of $\Omega_{\bar{\Lambda}^e}^3(\bar{\Lambda})$
in $\mod \bar{\Lambda}^e$.
\end{proof}

\begin{theorem}
\label{th:9.4}
There is an isomorphism
$\Omega_{\bar{\Lambda}^e}^4(\bar{\Lambda}) \cong \bar{\Lambda}$
in $\mod \bar{\Lambda}^e$.
In particular, $\bar{\Lambda}$ is a periodic algebra
of period $4$.
\end{theorem}

\begin{proof}
We proceed as in the proof of
Theorem~\ref{th:7.4},
and use
\cite[part (3) on the pages 119 and 120]{ESk2}.
In particular, we fix some
basis
$\cB = \bigcup_{i \in Q_0} \cB_i$ of $\bar{\Lambda}$ over $K$,
the socle elements $\omega_i$ of $e_i \bar{\Lambda}$,
and consider
the symmetrizing form 
$(-,-) : \bar{\Lambda} \times \bar{\Lambda} \to K$ such that,
for any two elements $x \in \cB_i$ and $y \in \cB$,
we have
\[
  (x,y) = \mbox{\,the coefficient of $\omega_i$ in $x y$},
\]
when $x y$ is expressed
as a linear combination of the elements of $e_i \cB = \cB_i$ over $K$.
Moreover, we consider the dual basis
$\cB^*$ of $\cB$ with respect to $(-,-)$.
Then, for each vertex $i \in Q_0$, we define the element of $\bP_3$
\[
  \xi_i = \sum_{b \in \cB_i} b \otimes b^* .
\]
Then we conclude as in the proof of
Theorem~\ref{th:7.4},
that there is a monomorphism
in $\mod \bar{\Lambda}^e$
\[
  \theta : \bar{\Lambda} \to \bP_3
\]
such that
$\theta (e_i) = \xi_i$
for any $i \in Q_0$.
It follows also from
Theorem~\ref{th:2.4}
and
Proposition~\ref{prop:9.1}
that
$\Omega_{\bar{\Lambda}^e}^4(\bar{\Lambda}) \cong {}_1 \bar{\Lambda}_{\sigma}$
in $\mod \bar{\Lambda}^e$
for some $K$-algebra automorphism $\sigma$ of $\bar{\Lambda}$.
Hence, we conclude that
$\dim_K \bar{\Lambda} = \dim_K \Omega_{\bar{\Lambda}^e}^4(\bar{\Lambda})$.
Moreover, by
Proposition~\ref{prop:9.3},
we have
$\Omega_{\bar{\Lambda}^e}^4(\bar{\Lambda}) = \Ker S$.
Therefore,
in order to show that $\theta$ induces an isomorphism
$\theta : \bar{\Lambda} \to \Omega_{\bar{\Lambda}^e}^4(\bar{\Lambda})$
in $\mod \bar{\Lambda}^e$,
it remains to prove that
$S(\xi_t) = 0$ for any $t \in Q_0$.
Since $K$ has characteristic $2$,
applying
\cite[part (3) on the pages 119 and 120]{ESk2},
we conclude that
for any vertex $i \in \partial(Q,f)$
and the border loop $\alpha$ at $i$,
the following equalities hold in $\bP_2$
\begin{align*}
    \sum_{b \in \cB_t e_i} b (\alpha \otimes \alpha + e_i \otimes \alpha^2) b^* &= 0 , \\
    \sum_{b \in \cB_t e_i} b (\alpha \otimes \alpha^2 + e_i \otimes \alpha^3) b^* &= 0 , \\
    \sum_{b \in \cB_t e_i} b (\alpha \otimes \alpha^3) b^* &= 0 ,
\end{align*}
because $\alpha^4 = 0$.
Then, for any $t \in Q_0$,
we obtain the equalities
\begin{align*}
  S(\xi_t)
     &= \sum_{b \in \cB_t} S ( b \otimes b^* )
     = \sum_{b \in \cB_t} \sum_{j \in Q_0} S ( b e_j \otimes e_j b^* )
  \\&
     = \sum_{b \in \cB_t} \sum_{j \in Q_0} b S ( e_j \otimes e_j ) b^*
     = \sum_{b \in \cB_t} \sum_{j \in Q_0} b \bar{\psi}_j b^*
     = 0
     .
\end{align*}
This completes the proof that $\bar{\Lambda}$ is a periodic algebra
of period $4$.
\end{proof}

\section{The representation type}\label{sec:reptype}

In this section we discuss the representation type of weighted
surface algebras and their socle deformations.
In particular, we complete the proofs of
Theorems \ref{th:main1} and \ref{th:main4}.

Let $A = K Q/I$ be a string algebra.
For a given arrow $\alpha \in Q_1$,
we denote by $\alpha^{-1}$ the formal inverse of $\alpha$
and set
$s(\alpha^{-1}) = t(\alpha)$
and
$t(\alpha^{-1}) = s(\alpha)$.
By a \emph{walk} in $(Q,I)$ we mean
a sequence $w = \alpha_1 \dots \alpha_n$,
where each $\alpha_i$ is an arrow or the inverse
of an arrow in $Q$, satisfying the following conditions:
\begin{enumerate}[(i)]
 \item
  $t(\alpha_i) = s(\alpha_{i+1})$ for any $i \in \{1,\dots,n-1\}$;
 \item
  $\alpha_{i+1} \neq \alpha_i^{-1}$ for any $i \in \{1,\dots,n-1\}$;
 \item
  $w$ does not contain a subpath $v$ such that
  $v$ or $v^{-1}$ belongs to $I$.
\end{enumerate}
Moreover,
$w$ is said to be a \emph{bipartite walk}
if, for any  $i \in \{1,\dots,n-1\}$,
exactly one of $\alpha_{i}$ and $\alpha_{i+1}$
is an arrow.
A walk $w = \alpha_1 \dots \alpha_n$ in $(Q,I)$
with $s(\alpha_1) = t(\alpha_n)$
is called a \emph{closed walk}.
Following \cite{SW,WW},
we say that a closed walk $w$ in $(Q,I)$
is a \emph{primitive walk}
if the following conditions are satisfied:
\begin{enumerate}[(i)]
 \item
  $w^m$ is a walk in $(Q,I)$ for any
  positive integer $m$;
 \item
  $w \neq v^r$ for any closed walk $v$ in $(Q,I)$ and
  positive integer $r$.
\end{enumerate}
It is known that a string algebra $A = K Q/I$
is representation-infinite if and only if $(Q,I)$ admits
a primitive walk (see \cite[Theorem~1]{SW}).
Moreover, if $A = K Q/I$ is a  representation-infinite
string algebra then the primitive walks in $(Q,I)$
create one-parameter families of stable tubes of rank $1$
in the Auslander-Reiten quiver $\Gamma_A$
(see \cite{BR,WW}).

We need the following combinatorial lemma.

\begin{lemma}
\label{lem:10.1}
Let $A = K Q/I$ be a string algebra
with $Q$ a $2$-regular quiver.
Then, for any arrow $\alpha \in Q_1$,
there is a bipartite primitive walk $w(\alpha)$ containing
the arrow $\alpha$.
\end{lemma}

\begin{proof}
Since $Q$ is a $2$-regular quiver, we have two
involutions
$\bar{} : Q_1 \to Q_1$
and
${}^* : Q_1 \to Q_1$
of the set $Q_1$ of arrows of $Q$.
The first involution assigns to each arrow $\alpha \in Q_1$
the arrow $\bar{\alpha}$
with
$s(\alpha) = s(\alpha^*)$
and
$\alpha \neq \bar{\alpha}$.
The second involution assigns to each arrow $\alpha \in Q_1$
the arrow $\alpha^*$
with
$t(\alpha) = t(\alpha^*)$
and
$\alpha \neq \alpha^*$.
Consider the automorphisms
$h : Q_1 \to Q_1$
such that
$h(\alpha) = \overbar{\alpha^*}$
for any arrow $\alpha \in Q_1$.
Clearly, $h$ has finite order.
In particular, for a given arrow $\alpha \in Q_1$,
there exists a minimal positive integer $r$
such that $h^r(\alpha) = \alpha$.
Then the required bipartite primitive walk $w(\alpha)$
is of the form
\[
   \alpha (\alpha^*)^{-1}
   h(\alpha) \big(h(\alpha)^*\big)^{-1}
  \dots
   h^{r-1}(\alpha) \big(h^{r-1}(\alpha)^*\big)^{-1}
  .
\]
\end{proof}

Let $(Q,f)$ be a triangulation quiver,
$m_{\bullet} : \cO(g) \to \bN^*$
a weight function, and
$c_{\bullet} : \cO(g) \to K^*$
a parameter function.
We consider the bound quiver algebra
\[
  \Gamma(Q,f,m_{\bullet},c_{\bullet})
   = K Q / L (Q,f,m_{\bullet},c_{\bullet}),
\]
where $L (Q,f,m_{\bullet},c_{\bullet})$
is the admissible ideal in the path algebra $KQ$ of $Q$ over $K$
generated by the elements
$\alpha f(\alpha)$ and $A_{\alpha}$,
for all  arrows $\alpha \in Q_1$.
Then
$\Gamma(Q,f,m_{\bullet},c_{\bullet})$
is a string algebra, called the string algebra
of the weighted triangulation algebra
$\Lambda(Q,f,m_{\bullet},c_{\bullet})$.
We note that it is the largest string quotient algebra of
$\Lambda(Q,f,m_{\bullet},c_{\bullet})$,
with respect to dimension.
Observe also that
$\Gamma(Q,f,m_{\bullet},c_{\bullet})$
is a quotient algebra of the special biserial degeneration algebra
$B(Q,f,m_{\bullet},c_{\bullet})$ of
$\Lambda(Q,f,m_{\bullet},c_{\bullet})$.
Moreover, if the border
$\partial(Q,f)$ of $(Q,f)$ is not empty and
$b_{\bullet} : \partial(Q,f) \to K$ is a border function,
then
$\Gamma(Q,f,m_{\bullet},c_{\bullet})$
is a quotient algebra of the socle deformed weighted
triangulation algebra
$\Lambda(Q,f,m_{\bullet},c_{\bullet},b_{\bullet})$.

\begin{proposition}
\label{prop:10.2}
Let $(Q,f)$ be a triangulation quiver,
$m_{\bullet} : \cO(g) \to \bN^*$
a weight function, and
$c_{\bullet} : \cO(g) \to K^*$
a parameter function.
Then the following statements hold:
\begin{enumerate}[(i)]
 \item
  $\Gamma(Q,f,m_{\bullet},c_{\bullet})$
  is a representation-infinite tame algebra.
 \item
  If there is an arrow $\alpha \in Q_1$
  with $n_{\alpha} \geq 4$ or $m_{\alpha} \geq 2$,
  then $\Gamma(Q,f,m_{\bullet},c_{\bullet})$
  is of non-polynomial growth.
\end{enumerate}
\end{proposition}

\begin{proof}
We write
$\Gamma = \Gamma(Q,f,m_{\bullet},c_{\bullet})$
and
$L = L (Q,f,m_{\bullet},c_{\bullet})$.


(i)
Since a string algebra is special biserial, it is tame, by
Proposition~\ref{prop:2.1}.
Moreover, since $Q$ is a $2$-regular quiver,
it follows from
Lemma~\ref{lem:10.1}
that there is a  (bipartite) primitive
walk in $(Q,L)$, and consequently
$\Gamma$ is representation-infinite.


(ii)
Assume that there is an arrow $\alpha \in Q_1$
such that $n_{\alpha} \geq 4$ or $m_{\alpha} \geq 2$.
Recall that we assume 
$m_{\alpha} n_{\alpha} \geq 3$.
Hence, if $n_{\alpha} = 1$, then $m_{\alpha} \geq 3$.
Moreover,
$\overbar{f(\alpha)} = g(\alpha)$,
$m_{\alpha} = m_{g(\alpha)}$,
$n_{\alpha} = n_{g(\alpha)}$.
Then
$u = g(\alpha) g^2(\alpha) \dots g^{n_{\alpha}-2}(\alpha)$
is a path of length $\geq 2$ and is a proper subpath
of $A_{g(\alpha)}$, and consequently $u$ does not
belong to $L$.
Observe also that
$g(\bar{\alpha}) = \overbar{f(\bar{\alpha})}$
and $m_{\bar{\alpha}} n_{\bar{\alpha}} \geq 3$,
again by our general assumption.
Hence,
$\bar{u} = g(\bar{\alpha}) \dots g^{n_{\bar{\alpha}}-2}(\bar{\alpha})$
is a path of length $\geq 1$ and is a proper subpath
of $A_{g(\bar{\alpha})}$, and then $\bar{u}$ does not
belong to $L$.
Consider the following closed walk in $(Q,f)$
\[
  v = u f(\bar{\alpha})^{-1} \bar{u} f(\alpha)^{-1}
\]
and observe that it is a primitive walk.
Applying
Lemma~\ref{lem:10.1},
we may also consider the bipartite primitive walk $w = w(g(\alpha))$.
We note that $w$ is of the form
$w = g(\alpha) \dots f(\alpha)^{-1}$.
In particular, we conclude that,
for any prime number $q$
and positive integers
$r_1, s_1, \dots, r_t, s_t$
with $\sum_{i=1}^t (r_i + s_i) = q$,
the closed walks in $(Q,L)$ of the form
\[
  v_1^{r_1}
  w_1^{s_1}
  \dots
  v_t^{r_t}
  w_t^{s_t}
\]
are primitive walks.
Then it follows by the arguments applied
in the proof of \cite[Lemma~1]{Sk0} that
the string algebra $\Gamma$ is not
of polynomial growth.
\end{proof}

Let $\Lambda = \Lambda(S,\vec{T},m_{\bullet},c_{\bullet})$
be a weighted surface algebra, and
$(Q,f) = (Q(S,\vec{T}),f)$
its triangulation quiver.
It follows from
Proposition~\ref{prop:5.8}
that $\Lambda$ is a tame algebra.
Further, the associated string algebra
$\Gamma = \Gamma(Q,f,m_{\bullet},c_{\bullet})$
is a quotient algebra of $\Lambda$ and
is representation-infinite, by
Proposition~\ref{prop:10.2}~(i).
Hence, $\Lambda$ is representation-infinite.
Applying
Lemma~\ref{lem:6.6},
we conclude that $\Lambda$
is a tetrahedral algebra if and only if
$n_{\alpha} = 3$ and $m_{\alpha} = 1$
for any arrow $\alpha \in Q_1$.
Moreover, if this is the case, then $\Lambda$
is of polynomial growth if and only if $\Lambda$
is a non-singular tetrahedral algebra, by
Propositions \ref{prop:6.3} and \ref{prop:6.4}.
Assume now that $\Lambda$ is not a tetrahedral algebra.
Then it follows from
Proposition~\ref{prop:10.2}~(ii)
that the string quotient algebra $\Gamma$ of $\Lambda$
is not of polynomial growth.
Hence, $\Lambda$ is of non-polynomial growth.

Let $A$ be a basic, indecomposable, symmetric algebra
which is socle equivalent to $\Lambda$.
We may assume that $A$ is not isomorphic to $\Lambda$.
Then it follows from
Propositions \ref{prop:8.2} and \ref{prop:8.3}
that the border $\partial(Q,f)$ of $(Q,f)$
is not empty, $K$ has characteristic $2$,
and $A$ is isomorphic to an algebra
$\bar{\Lambda} = \Lambda(Q,f,m_{\bullet},c_{\bullet},b_{\bullet})$,
for some border function $b_{\bullet}$ of $(Q,f)$.
In particular, we know that $\Lambda$ is not a tetrahedral algebra.
Then the string algebra $\Gamma$ of $\Lambda$
is an algebra of non-polynomial growth
and a quotient algebra of $\bar{\Lambda}$.
Therefore, $A$ is a tame algebra of non-polynomial growth.

We mention that in the special case when $\Lambda$ is the 
Jacobian algebra of an orientable surface with empty boundary
and non-empty collection of punctures, different from a sphere
with at most four punctures, the fact that $\Lambda$ is of
non-polynomial growth was proved in \cite[Theorem]{VD}.

\section*{Acknowledgements}

We are grateful to S. Ladkani for pointing out to us (September~2013)
that there are tame symmetric algebras whose indecomposable 
non-projective modules are periodic of period dividing $4$,
coming from cluster theory,
which we had not been aware of.

The research was done during the visits of the second named author
at the Mathematical Institute in Oxford (March 2014) and
the first named author at the Faculty of Mathematics and Computer Sciences
in Toru\'n (October 2014).
The results of the paper were partially presented 
at the conferences 
``Advances of Representation Theory of Algebras''
(Montreal, June 2014),
``Cluster Algebras and Geometry'' (M\"unster, March 2016),
``Workshop on Brauer Graph Algebras''
(Stuttgart, March 2016).

There is an overlap of the arXiv update 1608.00321 of \cite{La3}
from August~2016 with our paper,
which we incidentally submitted the first time early~2016.
We think it makes the paper harder to read if we were
to cite the details of the overlap and we have
therefore decided not to do this.



\end{document}